\documentclass[10pt]{article}

\usepackage[english]{babel}

\usepackage{url}
\usepackage{inputenc}
\usepackage{amsfonts}
\usepackage{amsmath}
\usepackage{empheq}
\usepackage{graphicx}
\usepackage{multirow}
\graphicspath{{images/}}
\usepackage{parskip}
\usepackage{fancyhdr}
\usepackage{bbold}
\usepackage{amssymb}
\usepackage{vmargin}
\usepackage{array}
 \usepackage{amsthm}
\usepackage{hyperref}
\usepackage{caption}
\usepackage{float}
\usepackage{fourier-orns}
\usepackage{listings}
\usepackage{epsfig}
\usepackage{color} 
\usepackage{bigints}
\usepackage[numbers]{natbib}

 \newtheorem{theorem}{Theorem}[section]

 \newtheorem{lemma}[theorem]{Lemma}
 \newtheorem{prop}[theorem]{Proposition}
 \newtheorem{remark}[theorem]{Remark}
 \newtheorem{definition}[theorem]{Definition}
 \newtheorem{assumption}[theorem]{Assumption}

\usepackage{listings}

\newcommand{\tu}{\tilde{u}}

\newcommand{\bydef}{\stackrel{\mbox{\tiny\textnormal{\raisebox{0ex}[0ex][0ex]{def}}}}{=}}

\newcommand{\R}{\mathbb{R}}
\newcommand{\om}{{\Omega_0}}
\newcommand{\mL}{\mathbb{L}}
\newcommand{\iL}{\mathbb{L}^{-1}}
\newcommand{\A}{\mathbb{A}}
\newcommand{\G}{\mathbb{G}}
\newcommand{\tU}{\tilde{U}_0}

\newcommand{\B}{\mathbb{B}}

\newcommand{\F}{\mathbb{F}}
\newcommand{\oF}{\overline{\mathbb{F}}}
\newcommand{\oL}{\overline{\mathbb{L}}}
\newcommand{\oA}{\overline{\mathbb{A}}}
\newcommand{\oB}{\overline{\mathbb{B}}}
\newcommand{\mG}{\mathbb{G}}
\newcommand{\cha}{\mathbb{1}_{\Omega_0}}
\newcommand{\out}{\mathbb{1}_{\mathbb{R}^m\setminus\Omega_0}}
\newcommand{\gdag}{{\gamma^\dagger}}
\newcommand{\og}{\overline{\gamma}}
\newcommand{\ogd}{\overline{\gamma}^\dagger}
\newcommand{\oGG}{\overline{\Gamma}}
\newcommand{\oGGD}{\overline{\Gamma}^\dagger}

\begin{document}
\title{Rigorous computation of solutions of semi-linear PDEs on unbounded domains via spectral methods}
\author{
Matthieu Cadiot
\footnote{McGill University, Department of Mathematics and Statistics, 805 Sherbrooke Street West, Montreal, QC, H3A 0B9, Canada. {\tt matthieu.cadiot@mail.mcgill.ca}}
\and
Jean-Philippe Lessard \footnote{McGill University, Department of Mathematics and Statistics, 805 Sherbrooke Street West, Montreal, QC, H3A 0B9, Canada. {\tt jp.lessard@mcgill.ca}}
\and
Jean-Christophe Nave \footnote{McGill University, Department of Mathematics and Statistics, 805 Sherbrooke Street West, Montreal, QC, H3A 0B9, Canada. {\tt jean-christophe.nave@mcgill.ca}}
}


\maketitle

\begin{abstract}
In this article we present a general method to rigorously prove existence
of strong solutions to a large class of autonomous semi-linear PDEs in a Hilbert
space $H^{l}\subset H^{s}(\mathbb{R}^{m})$ ($s\geq1$) via computer-assisted
proofs. Our approach is fully spectral and uses Fourier series to
approximate functions in $H^{l}$ as well as bounded linear operators
from $L^{2}$ to $H^{l}$. In particular, we construct approximate inverses of differential
operators via Fourier series approximations. Combining
this construction with a Newton-Kantorovich approach, we develop a
numerical method to prove existence of strong solutions. To
do so, we introduce a finite-dimensional trace theorem from which
we build smooth functions with support on a hypercube. The method
is then generalized to systems of PDEs with extra equations/parameters
such as eigenvalue problems. As an application, we prove the existence
of a traveling wave (soliton) in the Kawahara equation in $H^{4}(\mathbb{R})$
as well as eigenpairs of the linearization about the soliton. These
results allow us to prove the stability of the aforementioned traveling
wave. 
\end{abstract}
\begin{center}
\textbf{\small{}Key words.}{\small{} Semi-linear PDEs, Unbounded domains,
Computer-assisted proof, Fourier analysis, Newton-Kantorovich theorem,
Stability of traveling waves} 
\par\end{center}

\section{Introduction}

In this paper, we present a novel Fourier spectral method, to prove
(constructively) existence of strong solutions for a wide class of
autonomous semi-linear Partial Differential Equations (PDEs) defined on $\mathbb{R}^{m}$,
possibly with extra equations, such as eigenvalue problems. More specifically, this paper focuses on solutions vanishing at infinity (sometimes referred as localized states). In particular,
we develop tools to build an approximate inverse for the linearization
of semi-linear PDE operators around an approximate solution. Moreover,
the precision of the approximation can be verified rigorously on the
computer. We present the method theoretically in a general framework
and we provide tools for the numerical computations of the required
bounds. More specifically, given a domain $\Omega \subset \R^m$, the autonomous semi-linear PDEs on $\Omega$ we consider may
be written as 
\vspace{-.2cm}
\begin{equation}
\mathbb{F}(u) \bydef \mathbb{L}u + \mathbb{G}(u) =0,\label{eq:general_F=00003D0_1}
\vspace{-.2cm}
\end{equation}
 where $\mathbb{F}:X\to Y$
is smooth and defined between Hilbert spaces $X$ and $Y$ with norms $\|\cdot\|_{X}$,
$\|\cdot\|_{Y}$, respectively. 
$\mathbb{L}$ is defined to be the linear part of $\mathbb{F}$ and $\mathbb{G}$ is the  nonlinear part.
Moreover, we assume that $X$
is a closed subset of a Sobolev space $H^{s}(\Omega)$ and $Y$ is
a closed subset of $L^{2}(\Omega)$ (or $H^{s}(\Omega)$ in the weak
formulation setting). Moreover, $X$ and $Y$ take into account the
boundary conditions and any possible symmetries the problem may possess.
This set up is natural if one is looking for strong solutions of a
given PDE. The use of computer-assisted proofs is motivated by the
fact that constructively proving existence of solutions of semi-linear
PDEs is  notoriously nontrivial. While non-constructive existence
results for weak and strong solutions are vast and numerous, the results
and available techniques are seldom if one is interested in quantitative
explicit results regarding the solutions (e.g. their shape, their
specific location in the function space, the number of oscillations,
etc). This is perhaps not surprising, as PDEs are often nonlinear
and defined on infinite dimensional function spaces.
As a consequence of these hurdles and with the recent increase in
computing power, computer-assisted proofs (CAPs) in nonlinear analysis
have become a popular tool for the rigorous computation of
solutions of nonlinear equations. From the early pioneering work on
the Feigenbaum conjectures \cite{MR648529}, to the existence of chaos
and global attractor in the Lorenz equations \cite{MR1276767,MR1701385,MR1870856},
to the more recent works on Wright's conjecture \cite{MR3779642},
chaos in the Kuramoto-Sivashinsky PDE \cite{MR4113209}, 3D gyroids
patterns in materials \cite{jb_stationnary_states}, periodic orbits
in Navier-Stokes \cite{periodic_navier_spontaneous}, blowup in 3D
Euler \cite{blowup_3D_Euler} and imploding solutions for 3D compressible
fluids \cite{imploding_sols_3D_fluids}, the role of CAPs has become
increasingly present in the analysis of ordinary differential equations,
dynamical systems and PDEs. We refer the interested reader to the
book \cite{plum_numerical_verif} and the survey papers \cite{MR1849323,MR1420838,MR3444942,MR3990999}
for further details.

In this paper, we consider the domain of the PDE \eqref{eq:general_F=00003D0_1}
to be $\Omega\bydef\mathbb{R}^{m}$ and we choose $X\bydef H^{l}\subset H^{s}(\mathbb{R}^{m})$
(see the definition in \eqref{eq:H^l}) and $Y\bydef L^{2}(\mathbb{R}^{m})$
as Hilbert spaces of interest. By construction, notice that smooth solutions to \eqref{eq:general_F=00003D0_1} are localized. Then, assumptions on the operators $\mathbb{L}$ and $\mathbb{G}$ are given in Section \ref{ssec:formulation} and complement \eqref{eq:general_F=00003D0_1}. In particular, $\mathbb{L}$ is assumed to be invertible  and $\mathbb{G}$ is polynomial with smaller degrees of differentiability than the ones present in $\mathbb{L}$. Note that the invertibility of $\mathbb{L}$ is fundamental in our analysis (see Remark \ref{rem : essential spectrum of L} for instance) and will be used all along the paper.\\ The main objective of this article
is, given $u_{0}\in H^{l}$, to compute an approximate inverse
of $D\mathbb{F}(u_{0}):H^{l}\to L^{2}(\mathbb{R}^{m})$. More specifically,
that is to construct $\mathbb{A}:L^{2}(\R^{m})\to H^{l}$ which approximates
$D\mathbb{F}(u_{0})^{-1}:L^{2}(\mathbb{R}^{m})\to H^{l}$ such that
the $H^{l}$ operator norm satisfies $\|I_{d}-\mathbb{A}D\mathbb{F}(u_{0})\|_{l} <1$ and 
can be computed explicitly. In particular, the condition $\|I_{d}-\mathbb{A}D\mathbb{F}(u_{0})\|_{l}<1$
provides that both $D\mathbb{F}(u_{0}):H^{l}\to L^{2}$ and $\mathbb{A} : L^{2} \to H^l$ have a bounded inverse (cf. Theorem \ref{th : approximation of inverse and norm}). Moreover, it yields $\|D\mathbb{F}(u_{0})^{-1}\|_{2}\leq\frac{\|\mathbb{A}\|_{2,l}}{1-\|I_{d}-\mathbb{A}D\mathbb{F}(u_{0})\|_{l}}$.
In practice, one wants to have an explicit representation of $\mathbb{A}$
from which standard quantities can easily be computed (e.g. norm,
spectrum, action on a vector, etc). By writing $D\mathbb{F}(u_{0})^{-1}=\mathbb{A}+(I_{d}-\mathbb{A}D\mathbb{F}(u_{0}))D\mathbb{F}(u_{0})^{-1}$,
one can then approximate the action of $D\mathbb{F}(u_{0})^{-1}$.
Therefore, the spectrum of $D\mathbb{F}(u_{0})^{-1}$ may be controlled from the spectrum of $\mathbb{A}$. This allows for instance to control the spectral gap of $D\mathbb{F}(u_{0})^{-1}$ (this is achieved in this paper by computing an upper bound for the operator norm of $D\mathbb{F}(u_{0})^{-1}$). Moreover, building approximate inverses for $D\mathbb{F}(u_{0}) - \lambda I_d$ with different values for $\lambda$ provides non-existence of eigenvalues in some regions of the complex plane. This allows to tackle stability problems (as illustrated in Section \ref{sec : kdv eigen}). Finally, the operator $\mathbb{A}$ is essential in our computer-assisted approach as it allows to construct a Newton-like fixed point operator. This point is detailed in the next paragraph.
Consequently, building an approximate inverse not only allows to perform CAPs but also provides a mean to tackle problems related to the linearized operator.
In order to be able to obtain a representation of $\mathbb{A}$ in
the computer and use rigorous numerics tools to control its action,
we require the construction of $\mathbb{A}$ to be based on Fourier
series. In other words, we want to possess a representation
of $\mathbb{A}$ as an operator on Fourier coefficients from which the quantity $\|I_{d}-\mathbb{A}D\mathbb{F}(u_{0})\|_{l}$
can be evaluated through norm computations on sequences. 

In fact, our main motivation for building such an approximate inverse
$\mathbb{A}$ is the construction of a fixed point operator. Going
back to the general setting of \eqref{eq:general_F=00003D0_1}, 
if $u_{0}\in X$ satisfies $\mathbb{F}(u_0) \approx 0$, then one can build a Newton-like
fixed point operator $\mathbb{T}(u)\bydef u-\mathbb{A}\mathbb{F}(u)$
such that fixed points of $\mathbb{T}$ correspond to zeros of $\mathbb{F}$.
The goal is then to prove that the operator $\mathbb{T}$ is well
defined, and to find some $r>0$ such that $\mathbb{T}:\overline{B_{r}(u_{0})}\to\overline{B_{r}(u_{0})}$
is a contraction, where $\overline{B_{r}(u_{0})}\subset X$ is the
closed ball of radius $r$ centered at $u_{0}\in X$. As a result,
the Banach fixed point theorem yields the existence of a unique fixed
point of $\mathbb{T}$ in $\overline{B_{r}(u_{0})}$. In particular,
if both $u_{0}$ and $\mathbb{A}$ have a rigorous representation in the computer, then one may hope to prove the existence of solutions of \eqref{eq:general_F=00003D0_1} using a CAP.
In fact, that is why computing a rigorous inverse for $D\mathbb{F}(u_{0})$
is a stepping stone in the field of CAP as it is often the key step
in the analysis of such problems (although using the full inverse is sometimes possible, as presented later on in this introduction). Note that there are other ways
to build the fixed point operator. However, in general, most methods
require (at some point) to handle the inverse of the linearization
(either in the weak or the strong formulation). In some cases, although leading to a less precise construction, it is enough to obtain an upper bound for the norm of $D\mathbb{F}(u_{0})^{-1}$. We present the different existing approaches to either approximate $D\mathbb{F}(u_{0})^{-1}$ or compute an upper bound for its operator norm.

As was touched upon earlier, in order to produce an approximate
inverse $\mathbb{A}$, a subsidiary problem is the representation
of functions on $\mathbb{R}^{m}$ in the computer (at least of the approximate solution $u_{0}$) using Fourier series. This has partially been
done by Plum et al. \cite{plum_numerical_verif,plum_thesis_navierstokes} by building functions on $H^{1}(\Omega)$
with $\Omega$ unbounded, by considering $\Omega_{0}\subset\Omega$
bounded and defining $\tu_{0}\in H_{0}^{1}(\Omega_{0})$. Then,
using a zero extension of $\tu_{0}$ outside of $\Omega_{0}$, they
obtain a function $u_{0}\in H^{1}(\Omega)$. In particular, they used
finite elements in \cite{plum_thesis_navierstokes} and a sine Fourier
series in \cite{plum_numerical_verif} to build $\tu_{0}$. This approach
results in a representation of $u_{0}$ in the computer. We generalize this process 
by building functions in $H^{s}(\mathbb{R}^{m})$ using \emph{only}
Fourier series, that is, by considering a hypercube $\Omega_{0}$ on which we construct a trace-free approximation $\tu_{0}\in H_{0}^{s}(\Omega_{0})$.
Hence, we develop a finite-dimensional trace
theorem enabling us to project a given Fourier series from $H^{s}(\Omega_{0})$
into $H_{0}^{s}(\Omega_{0})$ (see Section \ref{sec : computer assisted analysis}).
The advantage of this approach is its numerical feasibility and compatibility
with the CAPs methods developed for periodic solutions (e.g. see \cite{jb_stationnary_states,MR2718657}).
Consequently, we can rigorously build functions in $H^{s}(\mathbb{R}^{m})$ having a representation as a sequence of Fourier coefficients,
which is, to the best of our knowledge, a new contribution. 
Depending on the spaces $X,Y$, several results exist
in CAPs in the study of $D\mathbb{F}(u_{0})^{-1}$.
We make a distinction between two main problems. First, when $X$ and $Y$ are spaces of periodic functions on an hypercube $\Omega$, the analysis on Fourier series
can be used and the PDE is turned into a zero finding problem on sequence
spaces. 
Second, when studying non-periodic PDEs, $X$ and $Y$ are Sobolev spaces on a domain $\Omega$ with non-periodic boundary conditions. In this setting, finite elements have been the main
tools to study $D\mathbb{F}(u_{0})^{-1}$. We now present a literature
review of both problems.

In the periodic case, the PDE \eqref{eq:general_F=00003D0_1} is turned
into $F(U)=0$ where $F:\hat{X}\to\hat{Y}$ is the new zero-finding
problem in sequence space with $\hat{X},\hat{Y}$ two Banach spaces
(on sequences of Fourier coefficients). Hence, for an approximate solution $U_{0}\in\hat{X}$,
the equivalent problem is to bound the norm of $DF(U_{0})^{-1}$.
One of the main ingredients to achieve such a goal is to exploit the
fact that the Fréchet derivatives are (asymptotically) diagonally
dominant (see \cite{periodic_navier_spontaneous,breden2022computer,period_kuramoto,Sander_equilibrium}
for some illustrations). Indeed, in autonomous semi-linear PDEs, the linear part
$DF(0)$ dominates for high-order modes. Therefore, the tail of $DF(U_{0})$
can be seen as a diagonally-dominant infinite dimensional matrix.
From this, one can approximate the inverse of $DF(U_{0})$ as a finite
matrix acting on a finite part of the sequence, and a tail operator
(which is diagonal) which acts on the tail of the sequence. Neumann
series arguments then give a way to invert $DF(U_{0})$ and approximate
 the inverse. Another way to understand this approach is
to notice that $DF(U_{0})-DF(0)$ is relatively compact with respect
to $DF(0)$. This feature justifies the use of finite matrices to
approximate infinite operators. Moreover, since the problem is set
up on a bounded domain, the resolvant of $DF(0)$ (in the cases where
it makes sense) is compact. It is this crucial point that differentiates
our problem to that in the periodic case (but they are still very
close due to similarities between Fourier series and Fourier transform
as we will show later). Note that it is also possible to obtain compactness from using some features of the equation and not necessarily $DF(0)$. Indeed,  in \cite{breden2022computer} the author uses the inverse Laplacian to leverage compactness in the case of nonlinear diffusion problems. In other cases, the inverse can also be computed explicitly.   Indeed, if
$U_{0}$ is simple enough, it is shown in \cite{cyranka2018construction,MR4379799} 
that one may analytically build the inverse operator and estimate an upper bound.
The linearization $DF(U_{0})$ is, in that case, a tridiagonal operator
for which the inverse operator can be built thanks to recurrence relations.
Using the asymptotics of these sequences, the inverse may then be
approximated. Similarly, the approach in \cite{MR3392647} uses recurrence
relations to approximate inverses of operators with tridiagonally-dominant
linear parts.

For bounded domains $\Omega$ with non-periodic boundary conditions,
different approaches have been used to control the norm of $D\mathbb{F}^{-1}(u_{0}):Y\to X$
or its weak formulation. The first method, that we denote as Nakao's
method \cite{nakao1990numerical,nakao2001numerical}, is based on
finite elements in which the inverse of $DF(u_{0})$ is approximated
by the inverse of a finite-dimensional projection on a finite element
space. In a sense, this approach is very similar to the approach based on diagonal dominance in the periodic case. Indeed compactness once again allows to produce an approximate inverse which one can write as a compact perturbation of the identity. In particular, the compact perturbation is obtained numerically and is represented by a matrix. Under certain assumptions on the finite elements space, the authors in \cite{nakao1990numerical,nakao2001numerical}
are able to invert the linearization and obtain a CAP technique. In
fact, this method has recently been improved using the Schur complement
\cite{MR4182090}.  Additionally, Watanabe et al. extended this result
and developed specific tools for the inversion of $D\mathbb{F}(u_{0})$
in the elliptic case. Then, using a finite dimensional subspace of
Sobolev spaces, they obtained explicit invertibility criteria and
the value of the upper bound. For instance, the method presented in
\cite{watanabe2005verify_invert} gives rigorous estimates for our
desired upper bound. The work \cite{watanabe2019improved}
presents further related results extending the method and give sharper
criteria for which the inverse can be defined. Also, on bounded domains,
Liu et al. \cite{liu_self_adjoint,liu_inverse} developed a method to rigorously estimate
eigenvalues of the linearized operator.
Using a finite dimensional subspace again, the eigenvalues of the
linearized operator are bounded, from above and below, using the eigenvalues
of a finite-dimensional operator. Then, using the theory of symmetric
operators, the norm of $D\mathbb{F}(u_{0})^{-1}$ is determined using
the minimal eigenvalue (in amplitude). Indeed, in the case where $D\mathbb{F}(u_{0})$
is symmetric, then we have that 
\vspace{-.2cm}
\[
\|D\mathbb{F}(u_{0})u\|_{L^{2}(\Omega)}\geq|\lambda|\|u\|_{L^{2}(\Omega)}
\vspace{-.2cm}
\]
where $\lambda$ is the eigenvalue of $D\mathbb{F}(u_{0})$ of minimal
amplitude. This method is particularly efficient if $D\mathbb{F}(u_{0})$
is self-adjoint itself and it may be extended to a general operator
by considering $D\mathbb{F}(u_{0})^{*}D\mathbb{F}(u_{0})$ where $D\mathbb{F}(u_{0})^{*}$
is the ajoint operator. Finally, the method by Plum et al. \cite{plum_numerical_verif},
in a similar fashion to the approaches in \cite{liu_self_adjoint,liu_inverse}, uses an approximation of the
eigenvalues of $D\mathbb{F}(u_{0})$ to compute the norm of its inverse. Indeed, using an homotopy method and the Temple-Lehmann-Goerisch method, they
are able to find a lower and upper bound for the eigenvalue of $D\mathbb{F}(u_0)$. Note that the control of the spectrum provides a norm for the inverse of $D\mathbb{F}(u_0)$ in the weak formulation. However,
using results on equivalence of norms, a bound for the norm in the strong formulation can easily be obtained. This feature is for instance displayed in \cite{plum_numerical_verif}.

The common ingredient of these methods is the use, at some point,
of compactness which either leads to a natural approximation of infinite
operators by matrices or to the spectral theorem that enables to evaluate
the norm via eigenvalues. However, as we pointed out earlier, in our
case, $\Omega=\mathbb{R}^{m}$ is unbounded and these compactness
results do not hold. Consequently, the methods presented above cannot readily apply and the analysis has to be modified to gain back the notion of compactness.

To the best of our knowledge, the only general method that can be
extended to unbounded domains is the method described in \cite{plum_numerical_verif}.
Indeed, they consider the weak formulation of $\mathbb{F}(u)=0$ and
try to build an upper bound for the norm of the inverse of the weak
operator associated to $D\mathbb{F}(u_{0})$. This operator is defined
via the inner product of the function space. In that setting, they
can use an analysis that falls in the set up of the case of a bounded
domain. For the homotopy method part, they use an extra homotopy that
links the unbounded domain case to the bounded one. In fact, they
use a domain decomposition homotopy \cite{Plum2000eigenvalue_domain_decomp}
to link the eigenvalues of a problem defined on a bounded domain $\Omega_{0}$
to the eigenvalues of the full problem. As a result, they may use the
coefficients homotopy as in the bounded domain case. This technique
allows to prove weak solution to PDEs on unbounded domain as demonstrated
in \cite{plum_numerical_verif} (where Schrödinger's equation on $\mathbb{R}^{2}$
is treated) or in \cite{plum_thesis_navierstokes} where the Navier-Stokes
equations are considered on an infinite strip with obstacle. It is
worth mentioning that constructive existence of solutions of PDEs
defined on the line can be obtained using the parameterization method
(e.g. see \cite{MR1976079,MR1976080,MR2177465}) to set-up a projected
boundary value problem that can then be solved using Chebyshev series
or splines (e.g. see \cite{MR2821596,MR3741385,MR4068579}). Recently, an extension of the parametrization method has been developed in \cite{radia_jb_jp_henot} to treat radially symmetric solutions to elliptic semiliear PDEs on $\R^d$. Reducing the problem to a non-autonomous ODE system, the authors were able to provide a Lipschitz bound for the local graph of a centre-stable manifold using Lyapunov–Perron method. As an illustration, they presented a proof of existence of a radially symmetric solution in the cubic Klein–Gordon equation on $\R^3$, in the planar Swift–Hohenberg equation and in a three-component FitzHugh–Nagumo system.   

Now, in the context of the present work, we wish to build an
approximate inverse of the strong formulation of the operator.  In this case, the challenging part comes from
the fact that we cannot use the theory of bi-linear forms of the operators
as done in \cite{plum_numerical_verif}. The
main advantage is, however, to be able to prove strong solutions to the PDE,
as opposed to weak solutions, which is fundamental from the point
of view of PDE theory. For instance, having a strong solution allows to study spectral stability (as illustrated in Section \ref{sec : stability kawahara}), whereas the linearized operator might not even be defined in the weak setting. Now, in terms of the understanding of the PDE, building an approximate inverse
allows not only to compute an upper bound for the norm of $D\mathbb{F}(u_{0})^{-1}$,
but also to approximate its action on functions, as presented earlier. These problems are
answered in Section \ref{sec:bound_inverse} of this work and lead
to a new method to build  approximate inverses of strong PDE
operators on unbounded domains. In particular, we introduce a bound $\mathcal{Z}_u$, which, in terms of computer-assisted analysis, is the quantity that separates periodic boundary value problems (which have been widely studied) and problems on $\R^m$. In practice, once this bound in known, then the rest of the computations are very similar the ones required during a CAP of a periodic solution using Fourier series, as presented in \cite{period_kuramoto, periodic_navier_spontaneous, Sander_equilibrium, breden2022computer} for instance.

As a consequence, the novelty of our approach is the combination of
the theory of Sobolev spaces and Fourier series analysis, resulting in a computer-assisted method to tackle PDEs on $\R^m$. The presented
method is fully spectral and the required bounds in the computer-assisted
approach are computed following the approach in \cite{periodic_navier_spontaneous,period_kuramoto}. Moreover, the success of our analysis is due to recovery of the notion of compactness. Indeed, instead of considering $\mathbb{L}^{-1}$ alone (which can be done on bounded domains), we study $D\mathbb{G}(u_0)\mathbb{L}^{-1}$ and prove this this operator is compact from $L^2(\R^m)$ to itself (cf. Lemma \ref{lem:compact}). This theoretical result is fundamental as it justifies our spectral construction on $\Omega_0$.   Consequently, after building an approximate solution $u_{0}$ through
its Fourier coefficients, the central result (see Section~\ref{sec:bound_inverse})
of this article enables us to build an approximate inverse
for $D\mathbb{F}(u_{0})$ using a bounded linear operator in $\ell^{2}.$
The construction of this approximation depends on the corresponding
 operator  $DF(U_{0})$ on sequences of Fourier coefficients. Moreover,
we show in Section \ref{sec:bound_inverse} that the periodic problem
``tends'' to the unbounded problem at the rate $\mathcal{Z}_{u}$,
where this quantity is exponentially decaying with the size of the
domain $\om$ (cf. Theorem \ref{th : Zu computation}). 

The paper is organized as follows. We begin in Section~\ref{sec:set_up}
by establishing the necessary tools and assumptions in order to set
up our method. Then we expose in Section~\ref{sec:bound_inverse}
the spectral method to compute an approximate inverse for
$D\mathbb{F}(u_{0})$. In Section~\ref{sec : computer assisted analysis},
we show how the construction of the approximate inverse may be used
to prove (constructively) existence of strong solutions via a computer-assisted
analysis. In Section~\ref{sec:constraigned} we show how the method
can be generalized to parameter dependent PDEs with extra equations
(in particular eigenvalue problems). Finally, Section~\ref{sec:kawahara}
provides an application of our approach for the Kawahara equation
(sometimes called 5th order KdV equation). In particular, we are able
to  prove the existence and stability of a soliton solution. All codes to perform the computer-assisted proofs
are available at \cite{julia_cadiot}.


\section{Set-up of the problem}\label{sec:set_up}

We begin this section by introducing the usual notations for Sobolev spaces that we will use throughout this article. We  refer the interested reader  to \cite{mclean2000strongly} for a complete presentation of such spaces and their usual properties. 

\subsection{Notations on Sobolev spaces}

Let $m \in \mathbb{N}$ and let $\Omega_0 \subset \mathbb{R}^m$ be a bounded Lipschitz  domain. Denote by $H^s(\Omega_0)$ the usual Sobolev space of order $s \in \mathbb{R}$ on $\Omega_0$ and $H^s(\mathbb{R}^m)$ the one on $\mathbb{R}^m$. Moreover,  $(\cdot,\cdot)_{H^s(\R^m)}$ and $\| \cdot \|_{H^s(\R^m)}$  represent the usual inner product and norm on $H^s(\R^m)$, while  $(\cdot,\cdot)_{H^s(\Omega_0)}$ and $\| \cdot \|_{H^s(\Omega_0)}$ represent the ones on $H^s(\Omega_0)$. Denote by $H^s_0(\Omega_0)$ the closure in $H^s(\Omega_0)$ of the space of infinitely differentiable compactly supported functions on $\Omega_0$.
When $s=0$, we use the Lebesgue notation $L^2 = L^2(\mathbb{R}^m)$ or $L^2(\Omega_0).$ More generally, $L^p = L^p(\R^m)$ is the usual $p$ Lebesgue space associated to its  norm $\| \cdot \|_{p}$.
Denote by $\mathcal{S} \bydef \left\{ u \in C^{\infty}(\mathbb{R}^m) : \sup_{x\in \mathbb{R}^m}|x^\beta \partial^\alpha u(x)| < \infty \text{ for all } \beta, \alpha \in \mathbb{N}^m \right\}$ the {\em Schwartz space} on $\mathbb{R}^m$. Denote by $\mathcal{B}(L^2)$ the space of bounded linear operators on $L^2$ and given $\B \in \mathcal{B}(L^2)$, denote by $\B^*$ the adjoint of $\B$ in $L^2.$ Moreover, given $z \in \mathbb{C}$, denote by $z^*$ the complex conjugate of $z$.

Denote by  $\mathcal{F} :L^2 \to L^2$ the \textit{Fourier transform} operator and
write $\mathcal{F}(f) \bydef \hat{f}$, where $\hat{f}$ is defined as $\hat{f}(\xi) \bydef \int_{\R^m}f(x)e^{-i2\pi x\cdot \xi}dx$ for all $\xi \in \R^m$. Similarly, the \textit{inverse Fourier transform} operator is expressed as $\mathcal{F}^{-1}$ and defined as $\mathcal{F}^{-1}(f)(x) \bydef \int_{\R^m}f(\xi)e^{i2\pi x\cdot \xi}d\xi$. In particular, recall the classical Plancherel's identity
\vspace{-.2cm}
\begin{equation}\label{eq : plancherel definition}
    \|f\|_2 = \|\hat{f}\|_2, \quad \text{for all }f \in L^2.
    \vspace{-.2cm}
\end{equation}
Finally, for $f_1,f_2 \in \mathcal{S}$, the continuous convolution of $f_1$ and $f_2$ is represented by $f_1*f_2$.

\subsection{Formulation of the problem}\label{ssec:formulation}

Given a function $\psi : \R^m \to \R$, such that $\psi\in L^2(\R^m)$, we consider the following class of semi-linear PDEs on $\mathcal{S}$
\vspace{-.2cm}
\begin{equation}\label{eq : f(u)=0 on S}
    \mathbb{L}u + \mathbb{G}(u) = \psi
    \vspace{-.2cm}
\end{equation}
 where $\mathbb{L}$ is a linear differential operator and $\mathbb{G}$ is a purely nonlinear ($\mathbb{G}$ does not contain linear operators) polynomial differential operator of order $N_{\mathbb{G}} \in \mathbb{N}$  where $N_{\mathbb{G}} \geq 2$, that is we can decompose it as a finite sum 
 \vspace{-.3cm}
\begin{equation}\label{def: G and j}
     \mathbb{G}(u) \bydef \displaystyle\sum_{i = 2}^{N_{\mathbb{G}}}\mathbb{G}_i(u) 
     \vspace{-.2cm}
\end{equation}
where $\mathbb{G}_i$ is a sum of monomials of degree $i$ in the components of $u$ and some of its partial derivatives.
In particular, for $i \in \{2,\dots,N_{\mathbb{G}}\}$, $\mathbb{G}_i$ can be decomposed as follows
\vspace{-.2cm}
\[
\mathbb{G}_i(u) \bydef \sum_{k \in J_i} (\mathbb{G}^1_{i,k}u) \cdots (\mathbb{G}^i_{i,k}u)
\vspace{-.2cm}
\]
where $J_i \subset \mathbb{N}$ is a finite set of indices, and where $\mathbb{G}^p_{i,k}$ is a linear differential operator for all $1\leq p \leq i$ and $k \in J_i$. $\mathbb{G}(u)$ depends on the derivatives of $u$ but we only write the dependency in $u$ for simplification. For instance, if $m=2$ and $\mathbb{G}(u) = u^2 + u \partial_1 u + u^2\Delta u$, then we define $\mathbb{G}_{2,1}^1 = \mathbb{G}_{2,1}^2 = \mathbb{G}_{2,2}^1 = \mathbb{G}_{3,1}^1 = \mathbb{G}_{3,1}^2 =  I_d$, $\mathbb{G}_{2,2}^2 = \partial_1$ and $\mathbb{G}_{3,1}^3 = \Delta$. Then we have $\mathbb{G}_2(u) = u^2 + u \partial_1 u = \left(\mathbb{G}_{2,1}^1u\right)\left(\mathbb{G}_{2,1}^2u\right) + \left(\mathbb{G}_{2,2}^1u\right)\left(\mathbb{G}_{2,2}^2u\right)$. Similarly, $\mathbb{G}_3(u) = u^2\Delta u = \left(\mathbb{G}_{3,1}^1 u\right) \left(\mathbb{G}_{3,1}^2 u\right) \left(\mathbb{G}_{3,1}^3 u\right).$  We emphasize later on the semi-linearity of the equation in Assumption~\ref{ass : LinvG in L1}. Letting
\vspace{-.2cm}
\begin{equation}\label{eq : f(u)=0 on H^l}
    \mathbb{F}(u) \bydef \mathbb{L}u + \mathbb{G}(u) - \psi= 0,
    \vspace{-.2cm}
\end{equation}
the zeros of $\mathbb{F}$ are solutions of \eqref{eq : f(u)=0 on S}. Since $\mathbb{G}$ is purely nonlinear,  the linear operator is given by
\vspace{-.2cm}
\[
\mathbb{L} = D\mathbb{F}(0).
\vspace{-.2cm}
\]
We complement equation (\ref{eq : f(u)=0 on S}) with two assumptions on the operators $\mathbb{L}$ and $\mathbb{G}$, which are natural to achieve our objectives. \\
We first assume that $\mathbb{L} = D\mathbb{F}(0) : \mathcal{S} \to \mathcal{S}$ is invertible. Assuming the invertibility of $\mathbb{L}$ is a strong assumption. Indeed, our objective is to invert $D\mathbb{F}(u_0)$, for $u_0$ being an approximate solution of our problem. Consequently, asking for the invertibility of $ D\mathbb{F}(0)$ can seem to be unrelated.  However, we justify this assumption later on in Remark \ref{rem : essential spectrum of L}. Moreover, the invertibility of $\mathbb{L}$   is fundamental in the analysis done in Section \ref{sec:bound_inverse}. We translate the invertibility of $\mathbb{L}$ using its Fourier transform $l$.

\begin{assumption}\label{ass:A(1)}
Assume that the Fourier transform of the linear operator $\mathbb{L}$ is given by 
\begin{equation} \label{eq:assumption1_on_L}
\mathcal{F}\big(\mathbb{L}u\big)(\xi) = l(\xi) \hat{u}(\xi), \quad \text{for all } u\in \mathcal{S},
\end{equation}
where $l(\cdot)$ is a polynomial in $\xi$. Moreover, assume that
\begin{equation} \label{eq:l>0}
|l(\xi)| >0, \qquad \text{for all } \xi \in \mathbb{R}^m.
\end{equation}
\end{assumption}

Using the definition of the polynomial $l$ from Assumption~\ref{ass:A(1)} and using \eqref{eq:l>0}, define 
\begin{equation} \label{eq:H^l}
H^l \bydef \left\{ u \in L^2 
: \int_{\mathbb{R}^m}|\hat{u}(\xi)|^2|l(\xi)|^2d\xi < \infty \right\},
\end{equation}
which is a Hilbert space that provides a natural domain for the linear operator $\mathbb{L}$, and is endowed with the inner product 
\[
(u,v)_l \bydef  \displaystyle\int_{\mathbb{R}^m}\hat{u}(\xi)\hat{v}(\xi)^*|{l}(\xi)|^2d\xi.
\]
Denote by $\|u\|_l \bydef \sqrt{(u,u)_l}$ the induced norm on $H^l$. By definition of the inner product on $H^l$ and by Plancherel's identity \eqref{eq : plancherel definition}, we obtain from \eqref{eq:assumption1_on_L} that
\begin{equation} \label{eq : norm_equality_Hl_L2}
\|u\|_l = \| \mathbb{L} u\|_2, \quad \text{for all } u \in H^l.
\end{equation}

\begin{remark}
    Note that $H^l$ is an abuse of notation not to be confused with the usual Sobolev spaces $H^s$. In general,  if $l(\xi) = \mathcal{O}(|\xi|^s)$  for some $s>0$, we have $H^l \subset H^s(\R^m)$ (using Plancherel's identity), when $H^l$ and $H^s(\R^m)$ are seen as sets of functions. However, the difference lies in the definition of the norm, which is a crucial point in our analysis as $\mL : H^l \to L^2$ is an isometry (cf. \eqref{eq : norm_equality_Hl_L2}).  To differentiate these spaces, we explicitly write the dependency in the domain when using a classical Sobolev space (e.g. $H^s(\R^m)$ or $H^s(\om)$) and keep $H^l$ for the specific Hilbert space defined in \eqref{eq:H^l}.
\end{remark}

Recall that we assume the PDE \eqref{eq : f(u)=0 on S} to be semi-linear, which helps with the analysis (e.g. see \cite{precup2012semilinear} or \cite{evans2010partial} for examples).  In other words, the degree of differentiability in $\mathbb{G}$ needs to be strictly smaller than the one in $\mathbb{L}$. Using the Hilbert space defined in \eqref{eq:H^l}, we consider $\psi \in L^2$ and look for $u \in H^l$ a solution of the PDE \eqref{eq : f(u)=0 on S}.  This requires introducing an assumption which implies that $\mathbb{G}$ is well-defined on $H^l$ and that $\mathbb{G}_i$ (see \eqref{def: G and j}) is a bounded multi-linear differential operator of order $i$ on $H^l$. The next  Assumption~\ref{ass : LinvG in L1} then provides a sufficient condition under which the operator $\mathbb{G}$ is well-defined and smooth on $H^l$ (as proven later on in Lemma \ref{Banach algebra}). Moreover, it complements the semi-linearity by specifying some requirements on the monomials of $\mathbb{G}$. These requirements will be useful later on for the analysis of Section \ref{sec:bound_inverse} (see Lemma \ref{lem : computation of falpha}). 

\begin{assumption}\label{ass : LinvG in L1}
For all $2 \leq i \leq N_{\mathbb{G}}$, $k \in J_i$ and $1 \leq p \leq i$, define $g_{i,k}^p$ such that
\vspace{-.2cm}
\[
\mathcal{F}\bigg(\mathbb{G}_{i,k}^p u\bigg)(\xi) = g_{i,k}^p(\xi)\hat{u}(\xi),
\quad \text{for all } \xi \in \mathbb{R}^m, \text{ and  assume that } ~\frac{g_{i,k}^p}{l} \in L^1.
\vspace{-.2cm}
\]
\end{assumption}
Assumption~\ref{ass : LinvG in L1} is stronger than the usual semi-linearity assumption as we also assume additional decay for each $g_{i,k}^p.$  Under Assumptions \ref{ass:A(1)} and \ref{ass : LinvG in L1}, one can prove that $\mathbb{G} : H^l \to H^1(\R^m)$ is well-defined. The fact that $\G$ maps $H^l$ to $H^1(\R^m)$ allows to obtain the regularity of the solutions a posteriori (cf. Proposition~\ref{prop : regularity of the solution}). Hence, looking for solutions in $H^l$ becomes equivalent to looking for classical solutions. Moreover, we compute explicit constants $\kappa_i >0$ such that $\|\G_i(u)\|_2 \leq \kappa_i \|u\|_l$ for all $u \in H^l$, which are used to perform the nonlinear analysis required in our Newton-Kantorovich approach (see Lemma~\ref{lem : Z2 bound} and Lemma~\ref{lem : Z_2 bound Kdv} for instance). The next Lemma~\ref{Banach algebra} is central in the set-up of this paper as it provides the aforementioned properties for $\G$ and each $\G_i$.

\begin{lemma}\label{Banach algebra}
Suppose that Assumptions \ref{ass:A(1)}  and \ref{ass : LinvG in L1} are satisfied. Then,
\begin{align}\label{eq : semi linear regularity}
    \mathbb{G}(u) \in H^1(\mathbb{R}^m), \quad \text{for all } u \in H^l.
\end{align}
For each $i \in \{2,\dots,N_{\mathbb{G}}\}$, let $\kappa_i >0$ satisfying
\begin{equation}\label{def : definition of kappai banach algebra}
    \kappa_i \geq \sum_{k \in J_i} \min_{p \in \{1, \dots, i\}}\left\|\frac{g^1_{i,k}}{l}\right\|_2 \cdots \left\|\frac{g^{p-1}_{i,k}}{l}  \right\|_2 \left\|\frac{g^{p}_{i,k}}{l}\right\|_\infty   \left\|\frac{g^{p+1}_{i,k}}{l}\right\|_2  \cdots \left\|\frac{g^{i}_{i,k}}{l} \right\|_2.
\end{equation}
Then, for all $u \in H^l$,
\vspace{-.1cm}
\begin{equation}\label{eq : multilinearity of G}
    \|\mathbb{G}_i(u)\|_2 \leq  \kappa_i \|u\|_l^i.
    \vspace{-.1cm}
\end{equation}

\end{lemma}
\begin{proof}
Let $u \in H^l$ and let $2\leq i \leq N_{\mathbb{G}}$, $k \in J_{i}$, $1 \leq p_1 \leq i$ and $1 \leq p_2 \leq i$. Recall the notations of Assumption~\ref{ass : LinvG in L1}, and denote $h_1 \bydef g^{p_1}_{i,k}$ and $h_2 \bydef g^{p_2}_{i, k}$. Then,
\vspace{-.2cm}
\begin{align*}
    \left|(1+|\xi|)\mathcal{F}\left( (\mathbb{G}^{p_1}_{i, k}u)(\mathbb{G}^{p_2}_{i, k}u)\right)(\xi)\right|
    &= \left|(1+|\xi|)\left((h_1\hat{u})*(h_2\hat{u})\right)(\xi)\right| \\
    &\le \int_{\mathbb{R}^m} |(1+|x|)h_1(x)\hat{u}(x)h_2(\xi-x)\hat{u}(\xi-x)| dx \\ 
    & \quad + \int_{\mathbb{R}^m} |h_1(x)\hat{u}(x)(1+|\xi-x|)h_2(\xi-x)\hat{u}(\xi-x)| dx.
    \vspace{-.2cm}
\end{align*}

Now, using that $l, h_1$ and $h_2$ are polynomials, notice that $\frac{h_1}{l} \in L^1$ by Assumption \ref{ass : LinvG in L1}, therefore $\frac{(1+|\cdot|)h_1(\cdot)}{l(\cdot)}, \frac{(1+|\cdot|)h_2(\cdot)}{l(\cdot)} \in L^\infty$. Therefore we get
\begin{align*}
    \left|(1+|\xi|)\mathcal{F}\left( (\mathbb{G}^{p_1}_{i, k}u)(\mathbb{G}^{p_2}_{i, k}u) \right)(\xi)\right|
    &\leq 
    \left( \max_{x \in \mathbb{R}^m}\frac{(1+|x|)|h_1(x)|}{|l(x)|}\right) (|h_2\hat{u}| * |l\hat{u}|)(\xi) \\
    & \quad + \left( \max_{x \in \mathbb{R}^m}\frac{(1+|x|)|h_2(x)|}{|l(x)|}\right) (|h_1\hat{u}| * |l\hat{u}|)(\xi).
\end{align*}
 Note that $\|l\hat{u}\|_2 = \|u\|_l$ by definition of the norm on $H^l$, then using Young's inequality for the convolution we get
\begin{align*}
    \left\|(1+|\xi|)\mathcal{F}\left( (\mathbb{G}^{p_1}_{i, k}u)(\mathbb{G}^{p_2}_{i, k}u) \right)(\xi)\right\|_2 &\leq 
 \max_{x \in \mathbb{R}^m}\frac{(1+|x|)|h_1(x)|}{|l(x)|} \|h_2\hat{u}\|_1 \|{u}\|_l \\
    & \quad + \max_{x \in \mathbb{R}^m}\frac{(1+|x|)|h_2(x)|}{|l(x)|}\|h_1\hat{u}\|_1 \|u\|_l.
\end{align*}
 Moreover, since $\frac{h_1}{l}, \frac{h_2}{l} \in L^1$ by Assumption \ref{ass : LinvG in L1} and $h_1, h_2, l$ are polynomials, we have  $\frac{h_1}{l}, \frac{h_2}{l} \in L^2$. Therefore, 
 \vspace{-.2cm}
\begin{equation}\label{eq : cauchy S lem banach}
    \|h_1\hat{u}\|_1  = \left\|\frac{h_1}{l} l\hat{u}\right\|_1  
    \leq \left\|\frac{h_1}{l}\right\|_2\|l\hat{u}\|_2 = \left\|\frac{h_1}{l}\right\|_2\|u\|_l
     \vspace{-.2cm}
\end{equation}
by Holder's inequality. Similarly, $\|h_2\hat{u}\|_1 \leq \|\frac{h_2}{l}\|_2\|u\|_l$. This implies that there exists $C>0$ such that 
\begin{align*}
  \left\|(1+|\xi|)\mathcal{F} \left( (\mathbb{G}^{p_1}_{i, k}u)(\mathbb{G}^{p_2}_{i, k}u) \right)(\xi)\right\|_2  \leq C \|u\|_l^2
\end{align*}
for all $u \in H^l.$ Now, recall that if $(1 + |\cdot|)\hat{f}(\cdot) \in L^2$ then $f \in H^1(\R^m).$ In particular,  it implies that $(\mathbb{G}^{p_1}_{i, k}u)(\mathbb{G}^{p_2}_{i, k}u) \in H^1(\mathbb{R}^m)$ for all $u \in H^l.$ Using the same reasoning recursively, we obtain that $\mathbb{G}(u) \in H^{1}(\mathbb{R}^m)$ for all $u \in H^l.$

Similarly, applying Plancherel's identity \eqref{eq : plancherel definition} first, one can prove that 
\begin{align*}
    \|(\mathbb{G}^{p_1}_{i, k}u)(\mathbb{G}^{p_2}_{i, k}u)\|_2 =\|\mathcal{F}( (\mathbb{G}^{p_1}_{i, k}u)(\mathbb{G}^{p_2}_{i, k}u))\|_2 &\leq \max_{x \in \mathbb{R}^m}\frac{|h_1(x)|}{|l(x)|} \left\|\frac{h_2}{l}\right\|_2 \|u\|_l^2
    =  \left\|\frac{h_1}{l}\right\|_\infty \left\|\frac{h_2}{l}\right\|_2 \|u\|_l^2.
\end{align*}
Notice that the role of $h_1$ and $h_2$ can be permuted in the previous inequality. Consequently, we also get that 
\begin{align*}
    \|(\mathbb{G}^{p_1}_{i, k}u)(\mathbb{G}^{p_2}_{i, k}u)\|_2 \leq   \left\|\frac{h_1}{l}\right\|_2 \left\|\frac{h_2}{l}\right\|_\infty \|u\|_l^2,
\end{align*}
which implies that
\begin{equation*}
     \|(\mathbb{G}^{p_1}_{i, k}u)(\mathbb{G}^{p_2}_{i, k}u)\|_2 \leq  \min \left\{ \left\|\frac{h_1}{l}\right\|_\infty \left\|\frac{h_2}{l}\right\|_2,~ \left\|\frac{h_1}{l}\right\|_2 \left\|\frac{h_2}{l}\right\|_\infty \right\}\|u\|_l^2.
\end{equation*}
Consequently, we get
\begin{align*}
 \|(\mathbb{G}^{1}_{i, k}u)\dots(\mathbb{G}^{i}_{i, k}u)\|_2 \leq   \left\|\frac{g^1_{i,k}}{l}\right\|_2 \cdots \left\|\frac{g^{p-1}_{i,k}}{l}  \right\|_2 \left\|\frac{g^{p}_{i,k}}{l}\right\|_\infty   \left\|\frac{g^{p+1}_{i,k}}{l}\right\|_2  \cdots \left\|\frac{g^{i}_{i,k}}{l} \right\|_2
\end{align*}
for all $p \in \{1, \dots,i\}$. This yields
$\|\mathbb{G}_i(u)\|_2 \leq \kappa_i \|u\|_l^i $
for all $u \in H^l$ where $\kappa_i$ satisfies \eqref{def : definition of kappai banach algebra}.
\end{proof}

Under Assumptions \ref{ass:A(1)} and \ref{ass : LinvG in L1},  $\mL$ maps $H^l$ to $L^2$ (cf. Assumption~\ref{ass:A(1)}) and $\mathbb{G}$ maps $H^l$ to $H^1(\R^m) \subset L^2$ (cf. Lemma \ref{Banach algebra}). Recalling \eqref{eq : f(u)=0 on S}, given $\psi \in L^2$, $\mathbb{F}$ maps $H^l$ to $L^2$ and the PDE \eqref{eq : f(u)=0 on H^l} is well-defined on $H^l$. In addition, in the cases where $\psi$ is smooth, the next Proposition~\ref{prop : regularity of the solution} shows that looking for classical solutions to \eqref{eq : f(u)=0 on H^l} is actually equivalent to looking for zeros of $\mathbb{F}$ in $H^l$. 

\begin{prop}\label{prop : regularity of the solution}
Assume $\psi \in H^\infty(\mathbb{R}^m)$ and let $u \in H^l$ such that $u$ solves \eqref{eq : f(u)=0 on H^l}. Then $u \in H^\infty(\mathbb{R}^m)$ and $u$ is a classical solution of \eqref{eq : f(u)=0 on H^l}.
\end{prop}
\begin{proof}
By assumption $\mathbb{F}(u)=0$ and hence 
\begin{equation}\label{eq : bootstrap}
  u = - \mathbb{L}^{-1}\mathbb{G}(u) + \mathbb{L}^{-1}\psi  
\end{equation}
using Assumption \ref{ass:A(1)}. Then, using Lemma \ref{Banach algebra}, we know that $\mathbb{G}(u) \in H^1(\R^m)$ so $\mathbb{L}^{-1}\mathbb{G}(u) \in H^{l+1}$, where we define $$H^{l+1} \bydef \left\{ u \in L^2 
: \int_{\mathbb{R}^m}|\hat{u}(\xi)|^2(1+|\xi|)^{2}|l(\xi)|^2d\xi < \infty \right\}.$$ Moreover, as $\psi \in H^\infty(\mathbb{R}^m)$ by assumption, we have $\mathbb{L}^{-1}\psi \in H^\infty(\mathbb{R}^m)$. Therefore, using \eqref{eq : bootstrap}, it yields that $u \in H^{l+1}$. Repeating this bootstrapping argument we obtain that $u \in H^\infty(\mathbb{R}^m)$ (see \cite{adams2003sobolev} for instance). Consequently, $u$ is a classical solution of \eqref{eq : f(u)=0 on H^l} (using that $H^\infty(\mathbb{R}^m) \subset C^\infty(\mathbb{R}^m)$ by Sobolev embeddings). 
\end{proof}

We provide some extra notation required for the analysis of this paper. Given $u \in L^\infty$, denote by
\begin{equation}\label{def : multiplication operator}
    \mathbb{u} : L^2 \to L^2
\end{equation}
the multiplication operator associated to $u$. More specifically, $\mathbb{u}v = uv$ for all $v \in L^2.$ 
Finally, denote by $\|\cdot\|_{l,2}$ the operator norm for any bounded linear operator between the two Hilbert spaces $H^l$ and $L^2$. Similarly denote by $\|\cdot\|_{l}$, $\|\cdot\|_2$ and $\|\cdot\|_{2,l}$ the operator norms for bounded linear operators on $H^l \to H^l$, $L^2 \to L^2$ and $L^2 \to H^l$ respectively. More generally, for two given Banach spaces $X$ and  $Y$, denote by $\|\cdot\|_{X,Y}$ the operator norm for bounded linear operators on $X \to Y$.

\begin{remark}
    Note that Assumption \ref{ass : LinvG in L1} not only allows $\mathbb{G}$ to be well defined on $H^l$, but also provides a posteriori regularity for the solution (assuming $\psi$ is smooth). Moreover, this assumption is fundamental in our Fourier series analysis as it generates diagonal dominance (cf. Remark \ref{rem : Z1 periodic and diag dominance}). We can consequently use to our advantage the widely developed literature in the subject, as described in the introduction.  
\end{remark}

\begin{remark}\label{rem : essential spectrum of L}
The above Assumptions~\ref{ass:A(1)} and \ref{ass : LinvG in L1} are natural in the theoretical and numerical study of PDEs, the goal being to obtain computer-assisted proofs of solutions in $H^l$. In particular, many PDEs, such as the Korteweg-De Vries (KdV), Kawahara and Swift-Hohenberg equations fall under these assumptions. Moreover, for a given $u_0 \in H^l$, Assumption~\ref{ass:A(1)} is necessary to have invertibility of $D\mathbb{F}(u_0)$. Indeed, under  Assumption~\ref{ass : LinvG in L1}, $D\mathbb{G}(u_0)$ is relatively compact with respect to $\mathbb{L}$ (cf. Lemma \ref{lem:compact}) and using Weyl's perturbation theory (see \cite{kato2013perturbation}), the essential spectrum of $D\mathbb{F}(u_0) = \mathbb{L} + D\mathbb{G}(u_0)$ is equal to the essential spectrum of $\mathbb{L}$. Consequently, if there exists $\xi$ such that $l(\xi) = 0$, then $0$ is in the essential spectrum of both $\mathbb{L}$ and $D\mathbb{F}(u_0)$. We thus cannot expect $D\mathbb{F}(u_0) : H^l \to L^2$ to be invertible in this set-up. This is an intrinsic characteristic of the unboundedness of the domain.
Finally, note that Assumptions~\ref{ass:A(1)}~and~\ref{ass : LinvG in L1} are crucial for the approach introduced in the current paper, and relaxing them is work under investigation (see the conclusion in Section~\ref{sec:conclusion} for more details).
\end{remark}

\subsection{Linearization operators}\label{ssec:lin_op}

Given a known element $u_0 \in H^l$ (which will be an approximate solution), we look for solutions to \eqref{eq : f(u)=0 on H^l} of the form $u+u_0$, where we solve for the unknown function $u$. We introduce some notations of the linearization around $u_0$ that will be used throughout this article. In particular, $D\mathbb{G}(u_0): H^l \to L^2$ denotes the linearization about $u_0$ of the nonlinearity of \eqref{eq : f(u)=0 on H^l}. Note that $D\mathbb{G}(u_0)$ is a linear combination of operators of the form $\mathbb{v}_\alpha\partial^{\alpha}$ (composition of a multiplication by $v_\alpha$, given in \eqref{def : multiplication operator}, and a differential operator $\partial^{\alpha}$) where $\alpha \in \mathbb{N}^m$  and $v_\alpha$ is a function in $H^1(\R^m)$ depending on $u_0$ (and possibly on its derivatives). In other words, we have
\begin{equation}\label{eq:assumption_on_G}
    D\mathbb{G}(u_0) \bydef \sum_{\alpha \in J_{\mathbb{G}}} \mathbb{v}_\alpha \partial^\alpha
\end{equation}
where $J_\mathbb{G} \subset \mathbb{N}^m$ is a finite set. Moreover, under Assumption \ref{ass : LinvG in L1}, one can prove that the multiplication operator $\mathbb{v}_\alpha : L^2 \to L^2$ is well-defined for all $\alpha \in J_\mathbb{G}$. This result will be useful in the proof of Theorem \ref{th : approximation of inverse and norm} below.

\begin{prop}\label{prop : L1 property of J_G}
    Let $\alpha \in J_\mathbb{G}$, then $v_\alpha \in L^\infty$ (using the notations of \eqref{eq:assumption_on_G}). In particular, the multiplication operator $\mathbb{v}_\alpha : L^2 \to L^2$ is well-defined.
\end{prop}
\begin{proof}
Let $i \in \{2, \dots, N_{\mathbb{G}}\}$, $k \in J_i$ and $p \in \{1, \dots, i\}$, then using \eqref{eq : cauchy S lem banach} we have
\[
g^p_{i,k}\hat{u} \in L^1
\]
for all $u \in H^l$. Therefore, by definition of $v_\alpha$ in \eqref{eq:assumption_on_G}, we know that $\widehat{v_\alpha}$ is a finite linear combination of functions of the form $g^p_{i,k}\hat{u}$ with $u \in H^l$. This implies that $\widehat{v_\alpha} \in L^1$. Finally, we conclude the proof using that for all $x \in \R^m$
\[
|v_\alpha(x)|\leq  \int_{\R^m}|\widehat{v_\alpha}(\xi)| d\xi = \|\widehat{v_\alpha}\|_1.
\]
\end{proof}

\begin{remark}\label{rem : L1 property of J_G}
Notice from Assumption \ref{ass : LinvG in L1} that $\frac{\xi^\alpha}{l(\xi)} \in L^1$ for all $\alpha \in J_\mathbb{G}$.  Although this result has not been used so far (we used instead that $\frac{\xi^\alpha}{l(\xi)} \in L^2$), it will be useful later on in the proof of Lemma \ref{lem : computation of falpha}. Also, note that if we replace $L^1$ by $L^2$ in Assumption \ref{ass : LinvG in L1}, Assumption \ref{ass : LinvG in L1} can be seen as a generalization of the usual condition one requires in order to get a Banach algebra property for Sobolev spaces. Indeed, if $H^l = H^s(\R^m)$ for some $s >0$, then having $uv \in H^l$ for all $u,v \in H^l$ can be obtained by requiring $\frac{1}{l} \in L^2$ (cf. Lemma \ref{Banach algebra}) or equivalently $s > \frac{m}{2}$, leading to the usual condition for having a Banach algebra (see \cite{adams2003sobolev} for instance). 
\end{remark}

\subsection{Periodic Sobolev spaces}\label{ssec:periodic_space}

The use of Fourier analysis is widespread in both the theoretical and numerical study of PDEs. We refer the interested reader to \cite{bahouri2011fourier} or \cite{iorio_fourier} for a review of some applications. In particular, periodic Sobolev spaces are efficient tools for our approach. In this section, we introduce the needed notations for our Fourier analysis on periodic Sobolev spaces. In particular, the objects of Sections~\ref{ssec:formulation}~and~\ref{ssec:lin_op} in $H^l$ will have their corresponding representation in periodic Sobolev spaces.

We consider an open cube $\Omega_0$ given by 
\begin{equation} \label{eq:definition_Omega_0}
\Omega_0 \bydef (-d_1,d_1) \times \cdots  \times (-d_m,d_m) \subset \R^m.
\end{equation}
Moreover, to simplify the analysis, we assume that 
$$
d \bydef d_1 = \cdots = d_m.
$$
Note that non-uniform hypercubes can be treated in the same fashion.
We begin by introducing the tilde notation for indices that will be regularly used throughout this paper.
\begin{definition}\label{def: tilde index}
Given $n \in \mathbb{Z}^m$, define $\tilde{n} \in \mathbb{\R}^m$ as $\Tilde{n}_i \bydef \frac{n_i}{2d}$ for all $1 \leq i \leq m$.
\end{definition}

Given $s \geq 0$, denote the Hilbert space
\[
    X^s \bydef \{ U=(u_n)_{n\in \mathbb{Z}^m} ~:~ (U,U)_{X^s} < \infty \} \text{ where } (U,V)_{X^s}  \bydef \sum\limits_{\underset{}{n \in \mathbb{Z}^m}} u_n v_n(1+|\tilde{n}|^2)^{s}.\vspace{-.3cm}
\]
The space $X^s$ which will be used as spaces for the sequences of  Fourier coefficients.
Similarly as in the continuous case, define the Hilbert space $X^l$ as 
\[
     X^l \bydef \{ U=(u_n)_{n\in \mathbb{Z}^m} ~:~ (U,U)_l < \infty\} \text{ where } (U,V)_l \bydef \sum\limits_{\underset{}{n \in \mathbb{Z}^m}} u_n v_n|l(\tilde{n})|^{2}.\vspace{-.4cm}
\]
Denote by $\|U\|_l \bydef \sqrt{(U,U)_l}$ the induced norm on $X^l$. 
Finally, $\ell^p$ denotes the usual $p$ Lebesgue space for sequences associated to its natural norm $\| \cdot \|_{p}$.

Since elements in $X^s$ (respectively $X^l$) have a periodic function representation, we denote by $H^s_{per}(\Omega_0)$ (similarly $H^l_{per}(\Omega_0)$) the representation in function space of the elements of $X^s$ (respectively $X^l$). We use capital letters for sequences $U$ and lowercase letters for functions $u$.

As in the continuous case, we introduce the notion of a convolution, which in this context is the discrete convolution between sequences $U=(u_n)_n$ and $V = (v_n)_n$, and is given by
\[
    (U*V)_n \bydef \sum\limits_{\underset{}{k \in \mathbb{Z}^m}} u_kv_{n-k},
    \quad \text{for all } n \in \mathbb{Z}^m.\vspace{-.2cm}
\]
The symbol ``$*$'' is used for both discrete and continuous convolutions when no confusion arises. Now, our operators in (\ref{eq : f(u)=0 on H^l}) have a representation in $X^l$ for the  corresponding periodic boundary value problem on $\Omega_0$.   In fact the linear part $\mathbb{L}$ becomes an operator ${L} : X^l \to \ell^2$ defined as $ LU = (l(\tilde{n})u_n)_{n\in \mathbb{Z}^m}$ for all $U \in X^l.$
By definition of the inner product on $X^l$, we obtain that
\begin{equation} \label{eq:norm_equality_Xl_L2}
\|U\|_l = \| {L} U\|_2, \quad \text{for all } U \in X^l.
\end{equation}
Similarly, the nonlinearity $\mathbb{G}$ given in \eqref{def: G and j} has a representation ${G}$ in $X^l$ that can be written as 
$
G(U) \bydef \sum_{i=1}^{N_\mathbb{G}}{G}_i(U) $
where ${G}_i$ is a multi-linear operator on $X^l$ of order $i$, and where products are interpreted as discrete convolutions. Hence, we can define $F : X^l \to \ell^2$ and the equivalent of the zero finding problem \eqref{eq : f(u)=0 on H^l} on $X^l$ as
\begin{equation}\label{eq : f(u)=0 on X^l}
  F(U) \bydef LU + G(U) - \Psi = 0, \quad \text{for } U \in X^l,  
\end{equation}
 where $\Psi$ is the sequence of Fourier coefficients in $\ell^2$ of $\psi|_{\Omega_0}$ (the restriction of $\psi$ to $\Omega_0$). Similarly as in the continuous case, given $U= (u_n)_{n \in \mathbb{Z}^m} \in \ell^1$, 
\begin{equation}\label{def : discrete conv operator}
    \mathbb{U} : \ell^2 \to \ell^2
\end{equation}
is the discrete convolution operator associated to $U$, that is $(\mathbb{U}V)_n = \displaystyle\sum_{k\in \mathbb{Z}^m} u_kv_{n-k}$ for all $n \in \mathbb{Z}^m$ and all $V=(v_n)_{n \in \mathbb{Z}^m} \in \ell^2.$ Note that $U*V \in \ell^2$ for all $U \in \ell^1$ and all $V \in \ell^2$ using Young's inequality for the convolution. Finally, for $\alpha \in J_\mathbb{G}$, we slightly abuse notation and denote by $\partial^\alpha$ the linear operator given by 
 \begin{align*}
    \partial^\alpha \colon &X^l \to \ell^2\\
    &U \mapsto \partial^\alpha U = ((i2\pi\tilde{n})^\alpha u_n)_{n\in \mathbb{Z}^m}.
\end{align*}
In other words, given $U \in X^l$ and $u_{per} \in H^l_{per}(\Omega_0)$ its function representation on $\om$, then $\partial^\alpha U$ is the sequence of Fourier coefficients of $\partial^\alpha u_{per}.$

\section{Approximation of the inverse of \boldmath$D\mathbb{F}(u_0)$\unboldmath}\label{sec:bound_inverse}

The main difficulty in the study of solutions of (\ref{eq : f(u)=0 on H^l}) on $\mathbb{R}^m$ is that despite the invertibility of $\mathbb{L}$, the operator $\mathbb{L}^{-1}$ is not compact, as it would be the case on a bounded domain. Compactness is often used when one needs to numerically approximate operators in a PDE. In fact, it is a central ingredient in computer-assisted proofs, as it justifies approximating infinite-dimensional operators by matrices. Hence, this leads to an important difference with the study of PDEs on unbounded domains. From that perspective, our challenge is to find a way to gain back some notion of compactness. This is achieved by considering $D\mathbb{G}(u_0)\mathbb{L}^{-1}$ instead of directly $\mathbb{L}^{-1}$. In this section we fix some $u_0 \in H^l$ such that supp$(u_0) \subset \overline{\Omega_0}$, representing an approximate solution of \eqref{eq : f(u)=0 on H^l}.

\subsection{Preliminary results}\label{ssec : preliminary results section 3}

Given $u_0 \in H^l$ with supp$(u_0) \subset \overline{\Omega_0}$, we  develop a method to approximate~$D\mathbb{F}(u_0)^{-1}~:~L^2~\to~H^l$.  Recall that Assumption~\ref{ass:A(1)} provides that $\mathbb{L} : H^l \to L^2$ has a continuous inverse $\mathbb{L}^{-1}:L^2 \to H^l$. Using this property, we can prove the following result.

\begin{lemma}\label{lem:compact}
Let $w_0 \in H^{l}$, then $D\mathbb{G}(w_0)\mathbb{L}^{-1} \colon L^2 \to L^2$ is compact.
\end{lemma}

\begin{proof}
Suppose first that supp$(w_0) \subset B_r$ for some $r>0$ where $B_r$ is the open ball centered at $0$ in $\R^m$. Let $u \in H^l$, then by Lemma \ref{Banach algebra}, we have that $D\mathbb{G}(w_0)u \in H^{1}(\R^m)$. Moreover supp$(D\mathbb{G}(w_0)u) \subset B_r$. Let $v$ be the restriction of $D\mathbb{G}(w_0)u$ on $B_r$. By construction, we obtain $v \in H^{1}_0(B_r)$. From the Rellich-Kondrachov theorem (see \cite{mclean2000strongly}), we obtain that $H^{1}_0(B_r)$ is compactly embedded in $L^2(B_r)$. As a consequence, using that $D\mathbb{G}(w_0) : H^l \to H^1(\R^m)$ is bounded (cf. Lemma \ref{Banach algebra}), it yields that $D\mathbb{G}(w_0)\mathbb{L}^{-1}$ is a compact operator from $L^2$ to itself. Then using that the set of smooth functions with compact support is dense in $H^l$ (see \cite{adams2003sobolev}) and that the set of compact operators is closed, we obtain the proof for all $w_0 \in H^l$.  
\end{proof}

Knowing that $D\mathbb{G}(u_0)\mathbb{L}^{-1} : L^2 \to L^2$ is compact, we can approximate it by finite range operators. In particular, it will justify the construction presented in Section \ref{ssec : construction of the approximate inverse}. 

Now, recall that our goal is to construct an approximate inverse for $D\mathbb{F}(u_0)$ using a representation as an operator on Fourier coefficients. Consequently, we need to introduce some transformations and associated spaces to switch from operators on $L^2(\R^m)$ to operators on $\ell^2$.
Let $\gamma : L^2 \to \ell^2$ be defined as 
\begin{align*}
    \gamma(u)  \bydef  \left(\frac{1}{|\om|}\int_\om u(x)e^{-2\pi i \tilde{n}\cdot x}dx \right)_{n \in \mathbb{Z}^m}
\end{align*}
for all $u \in L^2$. Intuitively, given $u \in L^2$, $\gamma(u)$ is the sequence of Fourier coefficients of the restriction of $u$ on $\om$, which is in $L^2(\om)$. Since Fourier series form an orthogonal basis on $L^2(\om)$, $\gamma : L^2 \to \ell^2$ is a bounded linear operator. Now similarly, define $\gamma^\dagger : \ell^2 \to L^2$ by
\begin{align*}
    \left(\gamma^\dagger(U)\right)(x) \bydef \cha(x) \sum_{n \in \mathbb{Z}^m}u_ne^{2\pi i \tilde{n}\cdot x}
\end{align*}
for all $U = (u_n)_{n \in \mathbb{Z}^m} \in \ell^2$ and all $x \in \R^m$, where $\cha$ is the characteristic function on $\om$.  Conversely, given a sequence of Fourier coefficients $U \in \ell^2$, $\gamma^\dagger(U)$ is the extension by zero on $\R^m$ of the function representation of $U$ in $L^2(\om)$. The transformations $\gamma$ and $\gdag$ are natural to study functions in $L^2$ having a support contained in $\om$.

 Now, given a Hilbert space $H$ of functions defined on $\mathbb{R}^m$, denote by $H_\om \subset H$ the subspace of $H$ defined as
\begin{equation}\label{def : Homega}
   H_\om \bydef \{u \in H ~: ~ \text{supp}(u) \subset \overline{\om} \}.
\end{equation}
For instance, we have $L^2_\om \bydef \{u \in L^2 ~:~ \text{supp}(u) \subset \overline{\om} \}$ or $H_\om^l \bydef \{u \in H^l ~:~ \text{supp}(u) \subset \overline{\om} \}$. Then, recall that $\mathcal{B}(L^2)$ (respectively $\mathcal{B}(\ell^2)$) denotes the space of bounded linear operators on $L^2$ (respectively $\ell^2$) and denote by $\mathcal{B}_\om(L^2)$ the following subspace of $\mathcal{B}(L^2)$
\begin{equation}\label{def : Bomega}
    \mathcal{B}_\om(L^2) \bydef \{\mathbb{B}_\om \in \mathcal{B}(L^2) ~:~  \mathbb{B}_\om = \cha \mathbb{B}_\om \cha\}.
\end{equation}
Finally, define $\Gamma : \mathcal{B}(L^2) \to \mathcal{B}(\ell^2)$ and $\Gamma^\dagger : \mathcal{B}(\ell^2) \to \mathcal{B}(L^2)$ as follows
\begin{align}\label{def : Gamma and Gamma dagger}
    \Gamma(\mathbb{B}) \bydef \gamma \mathbb{B} \gdag ~~ \text{ and } ~~  \Gamma^\dagger(B) \bydef \gamma^\dagger {B} \gamma. 
\end{align}
We now derive some useful properties of the previously defined transformations that will help us pass from $\ell^2$ to $L^2$ and vice-versa.
\begin{lemma}\label{lem : gamma and Gamma properties}
    The map $\sqrt{|\om|} \gamma : L^2_\om \to \ell^2$ (respectively $\Gamma : \mathcal{B}_\om(L^2) \to \mathcal{B}(\ell^2)$) is an isometric isomorphism whose inverse is given by $\frac{1}{\sqrt{|\om|}} \gdag : \ell^2 \to L^2_\om$ (respectively $\Gamma^\dagger :   \mathcal{B}(\ell^2) \to \mathcal{B}_\om(L^2)$).\\
    In particular,
    \begin{align}\label{eq : parseval's identity}
        \|u\|_2 = \sqrt{\om}\|U\|_2 \text{ and } \|\mathbb{B}_\om\|_2 = \|B\|_2
    \end{align}
    for all $u \in L^2_\om$ and $\mathbb{B}_\om \in \mathcal{B}_\om(L^2)$, and where $U \bydef \gamma(u)$ and $B \bydef \Gamma(\mathbb{B}_\om)$.
\end{lemma}
\begin{proof}
    The isomorphic property of both $\gamma$ and $\Gamma$ is easily obtained thanks to the fact that Fourier series form an orthogonal basis in $L^2(\om)$. Then the isometry is a direct consequence of Parseval's identity.
\end{proof}
The previous lemma is essential in our analysis as it allows not only to for objects on $L^2_\om$ to have a one-to-one representation in $\ell^2$, but also to compute their norm using their representation in $\ell^2.$ This property is of particular interest for our computer-assisted approach (see Section \ref{sec : computer assisted analysis}) since we are able to use the usual analysis on $\ell^2$ for norm computations. Similarly, given a bounded linear operator $B : \ell^2 \to \ell^2$,  Lemma \ref{lem : gamma and Gamma properties} allows to define a corresponding bounded linear operator $\mathbb{B} \bydef \Gamma^\dagger\left(B\right) : L^2_\om \to L^2_{\om}$ such that $\|\mathbb{B}\|_2 = \|B\|_2.$ This point is fundamental in our strategy for constructing an approximate inverse of $D\mathbb{F}(u_0)$ (cf. \eqref{def : definition of mathbb A}).

\begin{remark}
   Note that we have 
\[
\gamma(u) \in X^l
\]
whenever $u \in H^l_\om$, but $\gdag(U)$ is not necessarily in  $H^l_\om$ given  $U\in X^l$. Indeed, $\gdag(U) \in H^l_\om$ is achieved if the trace of $\gdag(U)|_\om$ on $\partial \om$ satisfies certain conditions. This point is discussed later on in Section \ref{ssec : finite dimensional trace theorem}. Moreover, notice that $\gamma \gamma^\dagger = I_d$, which is the identity on $\ell^2$, but $\gamma^\dagger \gamma = \cha \neq I_d.$ This implies that $\mathbb{B} = \Gamma^\dagger\left(\Gamma\left(\mathbb{B}\right)\right)$ if and only if $\mathbb{B} \in \mathcal{B}_\om(L^2).$ 
\end{remark}

The construction of a numerical approximation requires considering finite dimensional projections. 
Let $N \in \mathbb{N}$, consider $X= \ell^2$ or $X=X^l$, and define the projections $\pi^N : X \to X$ and $\pi_N : X \to X$ as follows
\vspace{-.2cm}
\begin{align}\label{def : projection on size N}
    \left(\pi^N(U)\right)_n  =  \begin{cases}
          U_n,  & n \in I^N \\
              0, &n \notin I^N
    \end{cases} 
     ~~ \text{and} ~~
     \left(\pi_N(U)\right)_n  =  \begin{cases}
          0,  & n \in I^N \\
              U_n, &n \notin I^N
    \end{cases}
    \vspace{-.2cm}
\end{align}
for all $n \in \mathbb{Z}^m$, where $I^N \bydef \{n \in \mathbb{Z}^m,~  |n_1| \leq N,\dots,|n_m|\leq N\}.$ Specifically, the above projection operators will help us separate finite-dimensional computations (which will be handled on the computer) and theoretical estimates (see Section \ref{sec : computer assisted analysis} ). We are now set-up to describe the construction of $\mathbb{A}$.

\subsection{Construction of the approximate inverse}\label{ssec : construction of the approximate inverse}

Recall that we want to build $\mathbb{A} : L^2 \to H^l$ such that
$
\|I_d - \mathbb{A}D\mathbb{F}(u_0)\|_l<1.
$
 Now, notice that we have 
\begin{align*}
    D\mathbb{F}(u_0) = \left(I_d + D\mathbb{G}(u_0)\mathbb{L}^{-1}\right)\mathbb{L} : H^l \to L^2.
\end{align*}
Therefore, we can equivalently look for an approximate inverse $\mathbb{B} : L^2 \to L^2$ of $I_d + D\mathbb{G}(u_0)\mathbb{L}^{-1}$ and define $\mathbb{A} \bydef \mathbb{L}^{-1} \mathbb{B}$. Indeed, because $\mathbb{L} : H^l \to L^2$ is an isometric isomorphism, $\mathbb{A} : L^2 \to H^l$ would be a well-defined bounded operator. Moreover, using Lemma \ref{lem:compact}, we know that $D\mathbb{G}(u_0)\mathbb{L}^{-1} : L^2 \to L^2$ is compact. Consequently, since $I_d + D\mathbb{G}(u_0)\mathbb{L}^{-1}$ is a compact perturbation of the identity, we can construct $\mathbb{B}$ as a compact perturbation of the identity as well. Having this strategy in mind, we present the construction of the approximate inverse.

We start by constructing a bounded linear operator $B : \ell^2 \to \ell^2$ approximating the inverse of $I_d + DG(U_0)L^{-1}$.
In particular, 
\begin{align}\label{eq : decomposition of B}
  B \bydef B^N + \pi_N  
\end{align}
where $B^N$ satisfies $B^N = \pi^N B^N \pi^N$. In other words, $B^N$ can be seen as a matrix approximating the inverse of  $\pi^N\left(I_d + DG(U_0)L^{-1}\right)\pi^N$. The construction of $B^N$ is obtained numerically using floating point arithmetic. Under Assumption \ref{ass : LinvG in L1}, we know that $DG(U_0)L^{-1} :\ell^2 \to \ell^2$ is compact (using diagonal dominance for instance). Consequently $B$ is expected to be a compact perturbation of the identity as well. This justifies the requirement $\pi_NB = \pi_N$. Now, we define 
\begin{align*}
    \mathbb{B}_\om \bydef \Gamma^\dagger(B)
\end{align*}
and we have $\mathbb{B}_\om \in \mathcal{B}_\om(L^2)$ by construction (cf. Lemma \ref{lem : gamma and Gamma properties}). Then, since we look for $\mathbb{B}$ as a compact perturbation of the identity, we define $\mathbb{B} : L^2 \to L^2$ as 
\begin{align*}
    \mathbb{B} \bydef \out + \mathbb{B}_\om,
\end{align*}
where $\out$ is the characteristic function on $\R^m \setminus \om$ and has to be seen as a multiplication operator on $L^2$ in the above equation. Finally, define $\mathbb{A} : L^2 \to H^l$ as 
\begin{align}\label{def : definition of mathbb A}
    \mathbb{A} \bydef \mathbb{L}^{-1} \mathbb{B} = \mathbb{L}^{-1}\left(\out + \Gamma^\dagger\left(\pi_N + B^N\right)\right).
\end{align}
For coherence, we also define $A : \ell^2 \to X^l$ as 
\begin{align}\label{def : definition of A}
    A \bydef L^{-1}B = L^{-1}\left(\pi_N + B^N\right).
\end{align}
Similarly, $A : \ell^2 \to X^l$ is well-defined as $L$ is an isometric isomorphism between $X^l$ and $\ell^2$ and $A$  approximates the inverse of $DF(U_0)$.

Since the above construction is based on the ``matrix" $B^N$, for which we have a rigorous representation on the computer,  the computations associated to $\mathbb{A}$, such as norms, spectrum, etc, can be obtained  using interval arithmetic \cite{Moore_interval_analysis}. As an illustration, the next lemma provides the operator norm computation of $\mathbb{A}$ using the norm of $B^N$.

\begin{lemma} \label{lem:computation_norm_A}
    Let $\mathbb{A} : L^2 \to H^l$ be given in \eqref{def : definition of A}. Then,
    \begin{equation} \label{eq:computation_norm_A}
        \|\mathbb{A}\|_{2,l} = \|\mathbb{B}\|_2 = \max\{1, \|B^N\|_2\}.
    \end{equation}
\end{lemma}
\begin{proof}
     Using \eqref{eq : norm_equality_Hl_L2}, we get $
      \|\A\|_{2,l} = \|\mL \A \|_{2} = \|\B\|_2.$
   Let $u \in L^2$, by construction $\B = \out + \B_\om$ where $\B_\om = \Gamma^\dagger\left(\pi_N + B^N\right) \in \mathcal{B}_\om(L^2)$ and therefore
  \begin{align*}
      \|\B u\|_2^2 = \|\mathbb{1}_{\mathbb{R}^m \setminus \om} u \|^2_2 + \|\B_\om u\|_2^2 &\leq \|\mathbb{1}_{\mathbb{R}^m \setminus \om} u \|^2_2 + \|\B_\om\|_2^2 \|\cha u\|_2^2\\
      &\leq \max\{1, \|\B_\om\|_2^2\}\|u\|_2^2.
  \end{align*}
  But notice that
  \begin{align*}
      \|\mathbb{B}\|_2 = \sup_{u \in L^2, \|u\|_2=1} \|\mathbb{B}u\|_2 \geq \sup_{u \in L^2_\om, \|u\|_2=1} \|\mathbb{B}u\|_2 = \|\B_\om\|_2
  \end{align*}
  and similarly,
   \begin{align*}
      \|\mathbb{B}\|_2 = \sup_{u \in L^2, \|u\|_2=1} \|\mathbb{B}u\|_2 \geq \sup_{u \in L^2, ~\|u\|_2=1 \atop u =  \mathbb{1}_{\mathbb{R}^m \setminus \om} u} \|\mathbb{B}u\|_2 = 1.
  \end{align*}
  Therefore $
      \|\B\|_2 = \max\{1, \|\B_\om\|_2\}.
 $
 Now, using \eqref{eq : parseval's identity}, we have 
 \[
  \|\B_\om\|_2 = \|\pi_N + B^N\|_2.
 \]
 Then, repeating the above computations, we obtain that $\|\pi_N + B^N\|_2 = \max\{1 , \|B^N\|_2\}.$
\end{proof}

\subsection{Upper bounds for \boldmath$\|I_d - \mathbb{A}D\mathbb{F}(u_0)\|_l$\unboldmath~and~\boldmath$\|D\mathbb{F}(u_0)^{-1}\|_{2,l}$\unboldmath}
\label{ssec : upper bound for the norm of the inverse}

Using the construction of Section \ref{ssec : construction of the approximate inverse}, we build an approximate inverse $\mathbb{A} \bydef \mathbb{L}^{-1}\left(\out + \Gamma^\dagger\left(\pi_N + B^N\right)\right) : L^2 \to H^l$ for $D\mathbb{F}(u_0)$. More specifically, $B^N$ is obtained numerically using floating point arithmetic and approximates the inverse of  $\pi^N\left(I_d + DG(U_0)L^{-1}\right)\pi^N$. Given $\mathbb{A}$ constructed in such a way, Theorem~\ref{th : approximation of inverse and norm} below provides an upper bound for $\|I_d - \mathbb{A}D\mathbb{F}(u_0)\|_l$, i.e. it measures how well $\mathbb{A}$ approximates $D\mathbb{F}(u_0)^{-1}$. Moreover, under the assumption that $\|I_d - \mathbb{A}D\mathbb{F}(u_0)\|_l<1$, Theorem \ref{th : approximation of inverse and norm} provides an upper bound for $\|D\mathbb{F}(u_0)^{-1}\|_{2,l}$.\\
Before stating Theorem~\ref{th : approximation of inverse and norm}, the analysis requires some additional notations. Indeed, we want  a rigorous upper bound for $\|I_d - \mathbb{A}D\mathbb{F}(u_0)\|_l$ on the computer, which requires to approximate $D\mathbb{F}(u_0) = \mathbb{L} + D\mathbb{G}(u_0)$. To do so, we approximate instead $D\mathbb{F}(u_0)\mathbb{L}^{-1} = I_d +  D\mathbb{G}(u_0)\mathbb{L}^{-1}$. The operator $D\mathbb{G}(u_0)\mathbb{L}^{-1} : L^2 \to L^2$ being compact (see Lemma \ref{lem:compact}), it can be approximated by finite range operators, meaning that a matrix approximation on a computer can be justified and quantified. Now, using \eqref{eq:assumption_on_G}, note that
 \vspace{-.2cm}
\begin{equation} \label{eq:assumption_on_G_2}
D\mG(u_0)\iL = \sum_{\alpha \in J_{\mG}} \mathbb{v}_\alpha \partial^\alpha \iL.
 \vspace{-.2cm}
\end{equation}
In particular, since $u_0 \in H^l_\om$, then for all $\alpha \in J_\mathbb{G}$
\begin{equation}\label{eq : valpha has compact support}
    v_\alpha \in L^2_\om\bigcap L^\infty \text{ and we define } V_\alpha \bydef \gamma(v_\alpha) \in \ell^2.
\end{equation}
In practice, we consider the adjoint in $L^2$ of this operator as it will lead to bounds which we can compute using the arithmetic on intervals (cf. Theorem \ref{th : Zu computation}). We have 
\[
(D\mG(u_0)\iL)^* = \sum_{\alpha \in J_{\mG}} (\partial^\alpha\iL)^* \mathbb{v}_\alpha^*.
\]
We derive in the next result an upper bound $\mathcal{Z}_1$ for the operator norm $\|I_d - \mathbb{A}D\mathbb{F}(u_0)\|_l$. Note that $\mathcal{Z}_1$ is the sum of two components, namely a bound $Z_1$ which depends on the periodic problem only and measures how well $A = L^{-1}(\pi_N + B^N)$ approximates the inverse for $DF(U_0)$; and a bound $\mathcal{Z}_{u}$ which arises from the unboundedness of the domain and which  depends on the approximation of the operator $(\partial^\alpha\iL)^* \mathbb{v}_\alpha^*$  by the operator   $\Gamma^\dagger\left((\partial^\alpha L^{-1})^*\right) \mathbb{v}_\alpha^*$.

\begin{theorem}\label{th : approximation of inverse and norm}
Let $\mathbb{A}  : L^2 \to H^l$ and $A : \ell^2 \to X^l$ be given in \eqref{def : definition of mathbb A} and \eqref{def : definition of A} respectively.  Assume that 
\begin{equation}\label{def : definition of Z1 and mathcal Z11}
\begin{aligned}
    \|I_d - ADF(U_0)\|_l &\leq Z_1\\
     \max\{1,\|B^N\|_{2}\} \sum_{\alpha
 \in J_\mathbb{G}}\|\big(\Gamma^\dagger\left((\partial^\alpha L^{-1})^*\right) - (\partial^\alpha\iL)^*\big)\mathbb{v}_\alpha^*\|_2 &\leq \mathcal{Z}_{u},
 \end{aligned}
\end{equation}
for some bounds $Z_1$ and $\mathcal{Z}_{u}$. Then, defining $\mathcal{Z}_1 \bydef Z_1 + \mathcal{Z}_{u}$, we have
\begin{equation}\label{ineq : upper bound Z1 in lemma}
    \|I_d - \mathbb{A}D\mathbb{F}(u_0)\|_l \leq  \mathcal{Z}_{1}.
\end{equation}
Moreover, if $\mathcal{Z}_{1} <1$, then both $\mathbb{A} : L^2 \to H^l$ and $D\mathbb{F}(u_0): H^l \to L^2$ have a bounded inverse  and
\begin{equation}
    \|D\mathbb{F}(u_0)^{-1}\|_{2,l} \leq \frac{\max\{1,\|B^N\|_{2}\}}{1-\mathcal{Z}_{1}}.
\end{equation}
\end{theorem}
\begin{proof}

First, recall that $\B = \mathbb{L} \mathbb{A}$, then using \eqref{eq : norm_equality_Hl_L2} and \eqref{eq:assumption_on_G_2} we get
\begin{align}
\nonumber
     \|I_d - \mathbb{A}D\mathbb{F}(u_0)\|_l &= \sup_{u \in H^l, \|u\|_l =1} \|\mathbb{L}(I_d - \mathbb{A}D\mathbb{F}(u_0))u\|_2 \\
     \nonumber
     &= \sup_{v \in L^2, \|v\|_2 = 1} \|\mathbb{L}(I_d - \mathbb{A}D\mathbb{F}(u_0))\mathbb{L}^{-1}v\|_2\\
     \nonumber
     &= \|I_d - \mathbb{B} ( I_d + \sum_{\alpha \in J_{\mG}} \mathbb{v}_\alpha \partial^\alpha \iL )\|_2 \\
     & \le \|I_d - \mathbb{B}(I_d + \sum_{\alpha \in J_\mathbb{G}} \mathbb{v}_\alpha\Gamma^\dagger\left(\partial^\alpha L^{-1}\right) )\|_2 
     + \| \mathbb{B}\sum_{\alpha \in J_\mathbb{G}} \mathbb{v}_\alpha\left(\Gamma^\dagger\left(\partial^\alpha L^{-1}\right) - \partial^\alpha\iL\right)\|_2,
     \label{ineq : first step in Zu}
\end{align}
where \eqref{ineq : first step in Zu} follows from triangle's inequality. Since $\B = \out + \B_\om$,
\[
    \|I_d - \mathbb{B}(I_d + \sum_{\alpha \in J_\mathbb{G}} \mathbb{v}_\alpha\Gamma^\dagger\left(\partial^\alpha L^{-1}\right) )\|_2 = \|\cha - \mathbb{B}_\om(\cha +  \sum_{\alpha \in J_\mathbb{G}} \mathbb{v}_\alpha\Gamma^\dagger\left(\partial^\alpha L^{-1}\right) )\|_2,
    \]
where we also used that $v_\alpha \in L^2_\om$ (cf. \eqref{eq : valpha has compact support}). 
However, notice that 
\begin{align*}
    \cha - \mathbb{B}_\om\left(\cha +  \sum_{\alpha \in J_\mathbb{G}} \mathbb{v}_\alpha\Gamma^\dagger\left(\partial^\alpha L^{-1}\right) \right)& = \Gamma^\dagger\left(I_d - \hspace{-.1cm}B\left(I_d + \sum_{\alpha \in J_\mathbb{G}}\mathbb{V}_\alpha \partial^\alpha L^{-1}\right)\right) \\
    &= \Gamma^\dagger\left(I_d - BDF(U_0)L^{-1}\right)
\end{align*}
as $\mathbb{B}_\om\in \mathcal{B}_\om(L^2)$ by assumption and $\gamma^\dagger\left(V_\alpha\right) = v_\alpha$ from \eqref{eq : valpha has compact support}. Therefore, using  \eqref{eq : parseval's identity}, we get
\begin{align}
     \|\cha - \mathbb{B}_\om(\cha +  \sum_{\alpha \in J_\mathbb{G}} \mathbb{v}_\alpha\Gamma^\dagger\left(\partial^\alpha L^{-1}\right) )\|_2
     =  \|I_d - BDF(U_0)L^{-1} \|_2  \le Z_1
     \label{eq:Z1_inequality_proof}
     \end{align}
     as $A = L^{-1}B$ and $\|U\|_l = \|LU\|_2$ for all $U \in X^l.$ Let us now focus on the second term $\| \mathbb{B}\sum_{\alpha \in J_\mathbb{G}} \mathbb{v}_\alpha\left(\Gamma^\dagger\left(\partial^\alpha L^{-1}\right) - \partial^\alpha\iL\right)\|_2$ in \eqref{ineq : first step in Zu}. Recall that for a bounded linear operator $\mathbb{T} : L^2 \to L^2$, then $\|\mathbb{T}\|_2 = \|\mathbb{T}^*\|_2$ where $\mathbb{T}^*$ is the adjoint of $\mathbb{T}.$ Using this identity, we get
     \vspace{-.2cm}
     \begin{align}
     \nonumber
         \| \mathbb{B}\sum_{\alpha \in J_\mathbb{G}} \mathbb{v}_\alpha\left(\Gamma^\dagger\left(\partial^\alpha L^{-1}\right) - \partial^\alpha\iL\right)\|_2 &\leq \| \mathbb{B}\|_2 \|\sum_{\alpha \in J_\mathbb{G}} \mathbb{v}_\alpha\left(\Gamma^\dagger\left(\partial^\alpha L^{-1}\right) - \partial^\alpha\iL\right)\|_2\\
         \nonumber
         &\leq \| \mathbb{B}\|_2 \sum_{\alpha \in J_\mathbb{G}} \left\|\left(\Gamma^\dagger\left((\partial^\alpha L^{-1})^*\right) - (\partial^\alpha\iL)^*\right)\mathbb{v}_\alpha^*\right\|_2 \\
         &= \max\{1,\|B^N\|_2\} \sum_{\alpha \in J_\mathbb{G}} \left\|\left(\Gamma^\dagger\left((\partial^\alpha L^{-1})^*\right) - (\partial^\alpha\iL)^*\right)\mathbb{v}_\alpha^*\right\|_2,
         \label{eq:Zu_inequality_proof}
         \vspace{-.3cm}
     \end{align}
     where the last equality follows from formula \eqref{eq:computation_norm_A} in Lemma~\ref{lem:computation_norm_A}.
     We conclude that \eqref{ineq : upper bound Z1 in lemma} holds by combining \eqref{ineq : first step in Zu}, \eqref{eq:Z1_inequality_proof} and \eqref{eq:Zu_inequality_proof}. 
     Finally, let us compute an upper bound for $\|D\mathbb{F}(u_0)^{-1}\|_{2,l}$. Let $u \in H^l$, then
     \vspace{-.2cm}
     \begin{align*}
     \|u\|_l = \|u - \mathbb{A}D\mathbb{F}(u_0)u + \mathbb{A}D\mathbb{F}(u_0)u\|_l 
         &\leq \|u - \mathbb{A}D\mathbb{F}(u_0)u \|_l + \|\mathbb{A}\|_{2,l}\|D\mathbb{F}(u_0)u\|_{2} \\
         &\leq \left(Z_1 + \mathcal{Z}_{u}\right)\|u\|_l + \|\mathbb{A}\|_{2,l}\|D\mathbb{F}(u_0)u\|_{2}.
          \vspace{-.2cm}
     \end{align*}
     If $1-\mathcal{Z}_1 \bydef 1 - Z_1 - \mathcal{Z}_{u} >0$, then $
         \|D\mathbb{F}(u_0)u\|_{2} \geq \frac{1-\mathcal{Z}_1}{\|\A\|_{2,l}} \|u\|_l$
     for all $u \in H^l$.
  But Lemma \ref{lem:compact} yields that   $I_d + D\mathbb{G}(u_0)\iL$ is a Fredholm operator, therefore, we obtain that $D\mathbb{F}(u_0)$ is invertible and 
  \[
       \|D\mathbb{F}(u_0)^{-1}\|_{2,l} \leq  \frac{\|\A\|_{2,l}}{1-\mathcal{Z}_1} = \frac{\max\{1,\|B^N\|_{2}\} }{1-\mathcal{Z}_1},
  \]
  where we used \eqref{eq:computation_norm_A}. Since $\|I_d -\mathbb{A} D\mathbb{F}(u_0)\|_l<1$ a Neumann series argument yields that $\mathbb{A} D\mathbb{F}(u_0) : H^l \to H^l$ has a bounded inverse and hence $\mathbb{A} : L^2 \to H^l$ has a bounded inverse as well.
\end{proof}

\begin{remark}\label{rem : link norm of inverses}
 $Z_1$ is the usual bound one has to compute for the proof of a periodic solution using the radii polynomial approach (see \cite{periodic_navier_spontaneous} or \cite{period_kuramoto} for instance), the only difference being that here the bound needs to be computed with the $\ell^2$ norm, as opposed to a more standard (weighed) $\ell^1$ norm computation. Because we deal with an unbounded domain, the extra term $\mathcal{Z}_{u}$ has to be taken into account. One can expect that $\mathcal{Z}_{u} \to 0$ as $d \to \infty$, which is the result we derive in the next Section \ref{ssec : computation of Z11 theoretical}. Furthermore, if one knows that $DF(U_0)$ is invertible, then one can choose $A = DF(U_0)^{-1}$ in the previous Theorem \ref{th : approximation of inverse and norm} (so that $Z_1 =0$). Moreover, the above reasoning provides 
\[
 \|D\mathbb{F}(u_0)^{-1}\|_{2,l} \leq \frac{\max\{1, \|DF(U_0)^{-1}\|_{2,l}\}}{1-\mathcal{Z}_{u}}
\]
where this time $\mathcal{Z}_{u} >0$ satisfies
\[
\max\{1, \|DF(U_0)^{-1}\|_{2,l}\} \sum_{\alpha \in J_\mathbb{G}} \left\|\left(\Gamma^\dagger\left((\partial^\alpha L^{-1})^*\right) - (\partial^\alpha\iL)^*\right)\mathbb{v}_\alpha^*\right\|_2 \leq \mathcal{Z}_{u}.
\]
If one can compute an upper bound for $\|DF(U_0)^{-1}\|_{2,l}$ (without necessarily using an approximate inverse $\mathbb{A}$), the previous lemma provides an upper bound for $D\mathbb{F}(u_0)^{-1}$ under the condition that $\sum_{\alpha \in J_\mathbb{G}} \|\left(\Gamma^\dagger\left((\partial^\alpha L^{-1})^*\right) - (\partial^\alpha\iL)^*\right)\mathbb{v}_\alpha^*\|_2$ is sufficiently small. This condition is verified using Theorem \ref{th : Zu computation}.
\end{remark}

\subsection{Computation of the bound \boldmath$\mathcal{Z}_{u}$\unboldmath}\label{ssec : computation of Z11 theoretical}

In Section \ref{ssec : upper bound for the norm of the inverse}, given an approximate inverse $\mathbb{A}$ defined in Section \ref{ssec : preliminary results section 3}, we derived a strategy to compute an upper bound $\mathcal{Z}_1$ satisfying $\|I_d - \mathbb{A}D\mathbb{F}(u_0)\|_l \le \mathcal{Z}_1 = Z_1 + \mathcal{Z}_{u}$, where both $Z_1$ and $\mathcal{Z}_{u}$ are defined in \eqref{def : definition of Z1 and mathcal Z11}. The bound $Z_1$ can easily be obtained as its  computation is well-known in the field of computer-assisted proofs in nonlinear analysis (see \cite{period_kuramoto} or \cite{periodic_navier_spontaneous} for instance). The novelty lies in the bound $\mathcal{Z}_{u}$ and this section focuses on its computation. As presented in the introduction, this quantity is the one separating the computer-assisted analysis of periodic solutions using Fourier series and the one of problems on $\R^m$. Specifically, having a sharp $\mathcal{Z}_u$ so that $\mathcal{Z}_1<1$ is essential as it is necessary to prove that both $\mathbb{A}$ and $D\mathbb{F}(u_0)$ have a bounded inverse (cf. Theorem \ref{th : approximation of inverse and norm}). We derive in this section a strategy to compute explicitly the bound $\mathcal{Z}_u$, having in mind a computer-assisted approach.

First, recall that $\mathcal{Z}_u$ depends on the quantity~$\|\left(\Gamma^\dagger\left((\partial^\alpha L^{-1})^*\right) - (\partial^\alpha\iL)^*\right)\mathbb{v}_\alpha^*\|_2$, with $\alpha \in J_\mathbb{G}$. Now, notice that $(\partial^\alpha\iL)^*$ has a Fourier transform given by $\xi \to \frac{(-i2\pi \xi)^\alpha}{l(-\xi)}$. Let us define $f_\alpha : \R^m \to \R$ through its Fourier transform as
\begin{equation} \label{eq:fourier_transform_f_alpha}
    \widehat{f_\alpha}(\xi) \bydef \frac{(-i2\pi \xi)^\alpha}{l(-\xi)}, \quad 
    \text{or equivalently } f_\alpha(x) = \mathcal{F}^{-1}\left(\frac{(-i2\pi\xi)^\alpha}{l(-\xi)}\right)(x)
\end{equation}
for all $\xi,x \in \R^m$. Then, fixing $u \in L^2$ and $\alpha \in J_\mathbb{G}$, we get that
\begin{equation} \label{eq:f_alpha_convolution}
(\partial^\alpha\iL)^* u = \mathcal{F}^{-1}( \widehat{f_\alpha} \hat{u})  =  f_\alpha*u.    
\end{equation}
The next lemma demonstrates that $f_\alpha$ is exponentially decaying. This decay is essential in the approximation of $(\partial^\alpha\iL)^*$ by $\Gamma^\dagger\left((\partial^\alpha L^{-1})^*\right).$

\begin{lemma}\label{lem : computation of falpha}
Let $\alpha \in J_\mathbb{G}$ and recall $f_\alpha$ from \eqref{eq:fourier_transform_f_alpha}. 
Then, there exist $C_\alpha>0$ and $a >0$ such that 
\begin{equation} \label{eq:assymptotic_bound_f_alpha}
    |f_\alpha(x)| \leq C_\alpha e^{-a|x|_1} 
\end{equation}
for all $x \in \mathbb{R}^m$ where $|x|_1 = |x_1| + \cdots + |x_m|$. Moreover, $a$ is independent of $\alpha$.
\end{lemma}
\begin{proof}
Fix $\alpha \in J_\mathbb{G}$. Remark~\ref{rem : L1 property of J_G} implies that  $\frac{(-i2\pi\xi)^\alpha}{l(-\xi)} \in L^1$. Denote by $\{b_j + ia_j\}_j$ the finite set of roots of the polynomial $l(\cdot)$. Let $\tilde{a} \bydef \min_{j} |a_j|$, using Assumption~\ref{ass:A(1)}, we know that $|l(x)| >0$ for all $x \in \R^m$,  so $\tilde{a} >0$. In particular, $\tilde{a}$ only depends on $l$ but not on $\alpha$. Moreover, choosing $0<{a}<\tilde{a}$, Cauchy's theorem (see Chapter IX in \cite{reed1975ii}) implies that there exists $C_\alpha$ such that \eqref{eq:assymptotic_bound_f_alpha} holds. 
\end{proof}

\begin{remark}
    Note that the previous lemma provides the existence of constants $C_\alpha$ ($\alpha \in J_\mathbb{G})$, which are not known explicitly. The sharp computation of these constants constitute a key point in this approach, which has to be dealt with in the case to case scenario. In practice, one can for example compute $C_\alpha$ by computing explicitly the function $f_\alpha$ (this is what is achieved in Lemma \ref{lem:kdv_f}). Another approach can be to use Cauchy's theorem and use integral estimates.
\end{remark}

The previous Lemma \ref{lem : computation of falpha} provides an exponential decay for $f_\alpha$, a result which we now use to control  the error when approximating $(\partial^\alpha\iL)^*\mathbb{v}_\alpha^*$ by $\Gamma^\dagger\left((\partial^\alpha L^{-1})^*\right) \mathbb{v}_\alpha^*$. To see this, for each $\alpha \in J_\mathbb{G}$, let us first define the following quantities
\begin{equation} \label{def : Zu1Zu2}
    \mathcal{Z}_{u,1}^\alpha \bydef   \|\mathbb{1}_{\R^m \setminus \om } (\partial^\alpha\iL)^*\mathbb{v}_\alpha^*\|_2,
    \qquad
     \mathcal{Z}_{u,2}^\alpha \bydef  \left\|\mathbb{1}_{\om } ((\partial^\alpha\iL)^*-\Gamma^\dagger\left((\partial^\alpha L^{-1})^*\right))\mathbb{v}_\alpha^*\right\|_2
\end{equation}
In particular, notice that if $\mathcal{Z}_u$ satisfies
\begin{equation}\label{eq : Zu defnition with Zu1 Zu2}
   \mathcal{Z}_u \geq \max\{1, \|B^N\|_2\}\sum_{\alpha \in J_\mathbb{G}}\sqrt{(\mathcal{Z}_{u,1}^\alpha)^2 + (\mathcal{Z}_{u,2}^\alpha)^2}, 
\end{equation}
then it also satisfies \eqref{def : definition of Z1 and mathcal Z11}. Given $\alpha \in J_\mathbb{G}$, the next Theorem \ref{th : Zu computation} provides an upper bound for both $Z^\alpha_{u,1}$ and $Z^\alpha_{u,2}$. Then, using \eqref{eq : Zu defnition with Zu1 Zu2}, one can compute a value for $\mathcal{Z}_u$.

\begin{theorem}\label{th : Zu computation}
Let $\alpha \in J_\mG$ and let $E_1, E_2$ be the sequences of Fourier coefficients of the functions $x \mapsto e^{2ad} - e^{2ad}\prod_{i=1}^m\left(1 -e^{-2ad}\cosh(2ax_i)\right)$ and  $x \mapsto \prod_{i=1}^m \cosh(2ax_i)$ on $\om$ respectively. In particular,
\begin{align}\label{eq : coefficients of E1 and E2}
\nonumber
    \left(E_1\right)_n &= e^{2ad}\delta_n - \frac{e^{2ad}}{|\om|}\prod_{k=1}^m\left(2d\delta_{n_k} - \frac{2a(-1)^{n_k}(1-e^{-4ad})}{4a^2 + \tilde{n}_k^2} \right)\\
    \left(E_2\right)_n &= \frac{(2a)^m\left(e^{2ad}-e^{-2ad}\right)^m}{|\om|}\prod_{k=1}^m \frac{(-1)^{n_k}}{4a^2 + \tilde{n}_k^2}
\end{align}
for all $n \in \mathbb{Z}^m$ where $\delta$ is the usual Kronecker delta.
Then, recalling $V_\alpha$ defined in \eqref{eq : valpha has compact support},
\begin{equation}\label{ineq : bound Zu1 in lemma}
     (\mathcal{Z}_{u,1}^\alpha)^2 \leq e^{-2ad}|\om|\frac{C_\alpha^2}{a^m}\left(V_\alpha^*,V_\alpha*E_1\right)_2.
\end{equation}
Moreover, there exists $C(d)>0$ depending algebraically on $d$ and independent of $\alpha$, such that
\begin{equation}\label{ineq : bound on Kalpha in Z1}
   (\mathcal{Z}_{u,2}^\alpha)^2 \leq (\mathcal{Z}_{u,1}^\alpha)^2 + e^{-4ad}C(d)C_\alpha^2|\om| \left(V_\alpha^*,V_\alpha*E_2\right)_2.
\end{equation}
\end{theorem}

\begin{proof}
First, the computation of the coefficients of $E_1$ and $E_2$ is easily obtained through some straightforward integrals of exponentials. Then, let $u \in L^2$ such that $\|u\|_2 =1$ and let $v \bydef v_\alpha^* u$. First, recall that
\begin{equation} \label{eq:decomposition_K_alpha_approx}
     \left\| ((\partial^\alpha\iL)^*-\Gamma^\dagger\left((\partial^\alpha L^{-1})^*\right))v\right\|_2^2 = \left\|\out(\partial^\alpha\iL)^* v\right\|_2^2 + \left\|\cha((\partial^\alpha\iL)^*-\Gamma^\dagger\left((\partial^\alpha L^{-1})^*\right))v\right\|_2^2
\end{equation}
as $\Gamma^\dagger\left((\partial^\alpha L^{-1})^*\right) = \cha \Gamma^\dagger\left((\partial^\alpha L^{-1})^*\right)$ by construction. We begin by focusing our attention on bounding the first term in the right-hand side of \eqref{eq:decomposition_K_alpha_approx}, that is $\|\out(\partial^\alpha\iL)^* v\|_2$. 
Recall that 
\[
((\partial^\alpha\iL)^*v)(x)  = (f_\alpha*v)(x)
\]
for all $x \in \mathbb{R}^m$ by definition of $f_\alpha.$ Then, using Lemma \ref{lem : computation of falpha} and \eqref{eq : valpha has compact support}, we obtain
\begin{align} \label{eq : Zu1alpha_ineq1}
\nonumber
   \|\mathbb{1}_{\R^m \setminus \Omega_0}(\partial^\alpha\iL)^* v\|_2^2
  & = \int_{\R^m \setminus \om} \left(\int_\om f_\alpha(y-x)v_\alpha(x)^*u(x) dx \right)^2dy\\
  & \leq C_\alpha^2 \int_{\R^m \setminus \om} \left(\int_\om e^{-a|y-x|_1}|v_\alpha(x)u(x)| dx \right)^2dy,
\end{align}
where we used supp$(v_\alpha) \subset \overline{\Omega_0}$. From Cauchy-Schwartz inequality and Fubini's theorem,
\begin{align}\label{eq : cauchy S for E2}
\nonumber
\int_{\R^m \setminus \om} \left(\int_\om e^{-a|y-x|_1}|v_\alpha(x)u(x)| dx \right)^2dy
     &\leq  \int_\om |u(x)|^2dx  \int_{\R^m \setminus \om} \int_\om e^{-2a|y-x|_1}|v_\alpha(x)|^2 dx dy\\
     &\leq \int_\om |v_\alpha(x)|^2 \int_{\R^m \setminus \om}  e^{-2a|y-x|_1} dydx
\end{align}
as $\|u\|_2 = 1$ by assumption. Then, let $x \in \om$, we have
\begin{align*}
    \int_{\R^m \setminus \om}  e^{-2a|y-x|_1} dy = \frac{1}{a^m} - \prod_{i=1}^m\int_{-d}^d e^{-2a|y_i-x_i|} dy_i
     = \frac{1}{a^m} - \frac{1}{a^m}\prod_{i=1}^m\left(1 -e^{-2ad}\cosh(2ax_i) \right).
\end{align*}
Therefore, combining \eqref{eq : Zu1alpha_ineq1} and \eqref{eq : cauchy S for E2}, we get
\[
     \|\mathbb{1}_{\R^m \setminus \Omega_0}(\partial^\alpha\iL)^* v\|_2^2 \leq \frac{C_\alpha^2}{a^m} \int_{\om} |v_\alpha(x)|^2\left(1 - \prod_{i=1}^m\left(1 -e^{-2ad}\cosh(2ax_i) \right)\right) dx.
\]
Recall that by hypothesis the vector $E_1$ is the sequence of Fourier coefficients of the function $x \mapsto \displaystyle e^{2ad} - e^{2ad}\prod_{i=1}^m\left(1 -e^{-2ad}\cosh(2ax_i)\right)$ on $\om$, and hence we obtain
\begin{align*}
    \int_{\om} |v_\alpha(x)|^2\left(1 - \prod_{i=1}^m\left(1 -e^{-2ad}\cosh(2ax_i) \right)\right) dx = |\om|e^{-2ad}\left(V_\alpha^*,V_\alpha*E_1\right)_2
\end{align*}
using Parseval's identity. Then, the first term in the right-hand side of \eqref{eq:decomposition_K_alpha_approx} is bounded by
\begin{equation}\label{eq : first term in this lemma}
    \|\mathbb{1}_{\R^m \setminus \Omega_0} (\partial^\alpha\iL)^* v\|_2^2 \leq |\om|e^{-2ad}\frac{C_\alpha^2}{a^m}\left(V_\alpha^*,V_\alpha*E_1\right)_2.
\end{equation}
Let us now find an explicit upper bound for the second term in the right-hand side of \eqref{eq:decomposition_K_alpha_approx}, namely for the term $\|\cha((\partial^\alpha\iL)^*-\Gamma^\dagger\left((\partial^\alpha L^{-1})^*\right))v\|_2$. Let $g \bydef \cha  ((\partial^\alpha\iL)^*-\Gamma^\dagger\left((\partial^\alpha L^{-1})^*\right)) v$. By construction $g \in L^2_\om$, consequently, thanks to Lemma \ref{lem : gamma and Gamma properties}, we can define $(g_n)_{n \in \mathbb{Z}^m} \bydef \gamma(g) \in \ell^2$. In particular, we have $\|g\|_2 = \sqrt{|\om|} \|(g_n)_{n \in \mathbb{Z}^m}\|_2$. Let $n \in \mathbb{Z}^m$, then notice that 
\begin{align}\label{eq : Z1_fourier_computation_1}
\nonumber
  \int_{\om} (\partial^\alpha\iL)^* v(x) e^{-i2\pi\tilde{n}\cdot x}dx &= \int_{\mathbb{R}^m} (\partial^\alpha\iL)^* v(x) e^{-i2\pi\tilde{n}\cdot x}dx - \int_{\mathbb{R}^m\setminus\om} (\partial^\alpha\iL)^* v(x) e^{-i2\pi\tilde{n}\cdot x}dx \\
  &= \frac{(-i2\pi\tilde{n})^{\alpha}\hat{v}(\tilde{n})}{{l(-\tilde{n})}} - \int_{\mathbb{R}^m\setminus\om} (\partial^\alpha\iL)^* v(x) e^{-i2\pi\tilde{n}\cdot x}dx
\end{align}
by definition of $\mathbb{L}$. But as $v=v_\alpha u$ has support on $\Omega_0$ (since supp$(v_\alpha) \subset \overline{\Omega_0}$),  we have 
\[
 \hat{v}(\tilde{n}) = \int_{\Omega_0}v(x) e^{-i2\pi\tilde{n}\cdot x} dx \bydef |\Omega_0| v_n
\]
where $(v_n)_{n \in \mathbb{Z}^m} \bydef \gamma(v)$. 
Therefore, 
\begin{equation}\label{eq : Z1_fourier_computation_2}
    \frac{(-i2\pi\tilde{n})^{\alpha}\hat{v}(\tilde{n})}{|\Omega_0|{l(-\tilde{n})}} = \frac{(-i2\pi\tilde{n})^{\alpha}v_n}{{l(-\tilde{n})}}.
\end{equation}
Now, using the definition of $\Gamma^\dagger$ in \eqref{def : Gamma and Gamma dagger}, we get
\begin{equation}\label{eq : Z1_fourier_computation_3}
    \frac{1}{|\om|} \int_{\om} \Gamma^\dagger\left((\partial^\alpha L^{-1})^*\right) v(x) e^{-i2\pi\tilde{n}\cdot x}dx =\frac{(-i2\pi\tilde{n})^{\alpha}}{{l(-\tilde{n})}|\om|}\int_{\om}  v(x) e^{-i2\pi\tilde{n}\cdot x}dx = \frac{(-i2\pi\tilde{n})^{\alpha}v_n}{{l(-\tilde{n})}}.
\end{equation}

Therefore, combining \eqref{eq : Z1_fourier_computation_1}, \eqref{eq : Z1_fourier_computation_2} and \eqref{eq : Z1_fourier_computation_3}, we obtain
\begin{equation*}
    g_n = -\frac{1}{|\om|} \int_{\mathbb{R}^m\setminus\om} (\partial^\alpha\iL)^* v(x) e^{-i2\pi\tilde{n}\cdot x}dx.
\end{equation*}
Then, denote by $\om + 2dn \bydef (-d+2dn_1,d+2dn_1)\times \cdots \times (-d+2dn_m,d+2dn_m)$ and using Parseval's identity, we obtain
\begin{align*}
   \|\cha  ((\partial^\alpha\iL)^*-\Gamma^\dagger\left((\partial^\alpha L^{-1})^*\right)) v \|_2^2 
   &= |\om| \sum_{n \in \mathbb{Z}^m} |g_n|^2 \\
  \nonumber
    &= \frac{1}{|\om|} \int_{\mathbb{R}^m\setminus\om} \hspace{-.8cm} (\partial^\alpha\iL) v^*(y)  \int_{\mathbb{R}^m\setminus\om} \hspace{-.8cm} (\partial^\alpha\iL)^* v(x) \sum_{n \in \mathbb{Z}^m}e^{i2\pi\tilde{n}\cdot (y-x)} dydx \\ \nonumber
    &= \sum_{n \in \mathbb{Z}^m} \int_{\mathbb{R}^m\setminus\om} \hspace{-.7cm}(\partial^\alpha\iL) v^*(y)\int_{\mathbb{R}^m\setminus\om} \hspace{-.7cm}  (\partial^\alpha\iL)^* v(x) \delta(y-x - 2dn) dydx \\ 
    &= \sum_{n \in \mathbb{Z}^m} \int_{\mathbb{R}^m\setminus (\om \cup (\om +2dn))}\hspace{-.4cm}(\partial^\alpha\iL)  v^*(y)(\partial^\alpha\iL)^*  v(y-2dn) dy,
\end{align*}
where we used that $ \displaystyle  \frac{1}{|\om|} \sum_{n \in \mathbb{Z}^m}e^{i2\pi\tilde{n}\cdot (y-x)} = \sum_{n \in \mathbb{Z}^m}\delta (y-x - 2dn)$, which is the Dirac comb.
Let $n \in \mathbb{Z}^m$, then using \eqref{eq:f_alpha_convolution} and Lemma \ref{lem : computation of falpha} we get
\begin{align*}
\nonumber
&
\hspace{-1cm}
\left|\int_{\mathbb{R}^m\setminus (\om \cup (\om +2dn))} \left((\partial^\alpha\iL)^* v(y)\right)^* (\partial^\alpha\iL)^*  v(y-2dn)   dy\right|\\ \nonumber
= & ~  \left|\int_{\mathbb{R}^m\setminus (\om \cup (\om +2dn))} \left( \int_{\om} f_\alpha(y-x)^*v(x)^*dx \right) \left( \int_{\om}  f_\alpha(y-2dn-z)v(z)dz  \right) dy\right|\\ \nonumber
\leq & ~ C_\alpha^2\int_{\mathbb{R}^m\setminus (\om \cup (\om +2dn))} \int_{\om}\int_{\om} e^{-a|y-x|_1}e^{-a|y-2dn-z|_1}|v(x)v(z)|dz dx   dy\\
= & ~ C_\alpha^2\int_{\om}\int_{\om} |v(x)v(z)|\left(\int_{\mathbb{R}^m\setminus (\om \cup (\om +2dn))} e^{-a|y-x|_1}e^{-a|y-2dn-z|_1}dy\right) dz dx 
\end{align*}
using Fubini's theorem. Then, notice that if $n =0$, we have
\begin{align*}
    \int_{\mathbb{R}^m\setminus (\om \cup (\om +2dn))}\left((\partial^\alpha\iL)^* v(y)\right)^* (\partial^\alpha\iL)^*  v(y-2dn)   dy 
    =  \|\mathbb{1}_{\R^m \setminus \Omega_0}(\partial^\alpha\iL)^* v\|_2^2,
\end{align*}
which is the quantity we computed in \eqref{eq : first term in this lemma}.
Moreover, one can prove that there exists a constant $C(d)>0$, depending on $d$ algebraically and independent of $\alpha$, such that
\begin{align}\label{eq : quantity needed for Z1}
    \sum_{n \in \mathbb{Z}^m, n \neq 0}\int_{\mathbb{R}^m\setminus (\om \cup (\om +2dn))} e^{-a|y-x|_1}e^{-a|y-2dn-z|_1}dy \leq C(d)e^{-4ad}\prod_{i=1}^m \cosh(ax_i)\cosh(az_k)
\end{align}
using some standard integration techniques of exponentials. Therefore, we obtain
\begin{align*}
    & \hspace{-1cm}
\sum_{n \in \mathbb{Z}^m} \int_{\mathbb{R}^m\setminus (\om \cup (\om +2dn))}\left((\partial^\alpha\iL)^* v(y)\right)^*(\partial^\alpha\iL)^*  v(y-2dn)   dy \\
    \leq & ~ \|\mathbb{1}_{\R^m \setminus \Omega_0}(\partial^\alpha\iL)^* v\|_2^2 + e^{-4ad}C(d)C_\alpha^2\int_{\om}\int_{\om} |v(x)v(z)|\prod_{i=1}^m \cosh(ax_i)\cosh(az_k) dz dx \\
    = & ~ \|\mathbb{1}_{\R^m \setminus \Omega_0}(\partial^\alpha\iL)^* v\|_2^2 + e^{-4ad}C(d)C_\alpha^2\left(\int_{\om} |v(x)|\prod_{i=1}^m \cosh(ax_i) dx \right)^2.
\end{align*}
Then, recall that $v \bydef v_\alpha^* u$ and using Cauchy Schwartz inequality, we get
\begin{align*}
  \left(\int_{\om} |v(x)|\prod_{i=1}^m \cosh(ax_i) dx \right)^2 \leq \|u\|^2_2\int_{\om} |v_\alpha(x)|^2\prod_{i=1}^m \cosh(ax_i)^2 dx.
\end{align*}
Finally, using that $\cosh(x)^2 \leq \cosh(2x)$ for all $x \in \R$, we obtain
\begin{align*}
    & \hspace{-1.5cm} \sum_{n \in \mathbb{Z}^m} \int_{\mathbb{R}^m\setminus (\om \cup (\om +2dn))}\left((\partial^\alpha\iL)^* v(y)\right)^*(\partial^\alpha\iL)^*  v(y-2dn)   dy\\
    \leq &~ \|\mathbb{1}_{\R^m \setminus \Omega_0}(\partial^\alpha\iL)^* v\|_2^2 + e^{-4ad}C(d)C_\alpha^2\int_{\om} |v_\alpha(x)|^2\prod_{i=1}^m \cosh(2ax_i) dx.
\end{align*}
We conclude the proof using the definition of $E_2$ and  Parseval's identity (cf. \eqref{eq : first term in this lemma}).
\end{proof}

\begin{remark}
We illustrate the computation of the constant $C(d)$ in the case of the Kawahara equation in Section \ref{sec:kawahara}.
Once $C(d)$  is determined theoretically, then the upper bound given in \eqref{ineq : bound on Kalpha in Z1} can easily be obtained numerically using the Fourier coefficients  $E_1$ and $E_2$ given in \eqref{eq : coefficients of E1 and E2}.
Then, assuming that $V_\alpha$ only has a finite number of non-zero coefficients, then we only need a finite number of coefficients of $E_1$ and $E_2$ to compute $(V_\alpha, V_\alpha*E_i)$ ($i \in \{1, 2\}$).
\end{remark}

\begin{remark}
Notice from \eqref{def : Zu1Zu2} that $\mathcal{Z}_{u,2}^\alpha \approx \mathcal{Z}_{u,1}^\alpha $ if $d$ is sufficiently big and if $C(d)$ is computed accurately. Therefore, heuristically, if one chooses $d$ sufficiently big such that \\$\sqrt{2}\max\{1, \|B\|_2\}\sum_{\alpha \in J_\mathbb{G}}\mathcal{Z}_{u,1}^\alpha \ll 1,$ then using \eqref{def : definition of Z1 and mathcal Z11}, one should obtain $\mathcal{Z}_u \ll 1$.
\end{remark}

\section{Existence of solutions via computer-assisted analysis}\label{sec : computer assisted analysis}

In this section, we present a method to prove constructively the existence of solutions of the PDE \eqref{eq : f(u)=0 on S} in the space $H^l$. This approach begins with the construction of two important objects, namely an approximate solution $u_0 \in H^l_\om$ and a linear operator $\A : L^2 \to H^l$ which is an approximate inverse of $D\mathbb{F}(u_0)$. Once these two objects are constructed, the (computer-assisted) proof of existence is obtained by showing that the Newton-like fixed point operator $\mathbb{T}:H^l\to H^l$ defined as $\mathbb{T}(u) \bydef u - \A D\F(u) $ possesses a unique fixed point in $\overline{B_r(u_0)}$, where $B_r(u_0)$ is the open ball of radius $r$ centered at $u_0$. This is achieved by verifying the hypotheses of a Newton-Kantorovich approach (see Section \ref{ssec : radii polynomial}).
In particular, the crucial hypothesis \eqref{condition radii polynomial} to verify depends on three bounds $\mathcal{Y}_0$, $\mathcal{Z}_1$ and $\mathcal{Z}_2$ that need to be computed explicitly. Using the construction of $\A$ presented in details in Section~\ref{sec:bound_inverse}, we present in Section~\ref{ssec : bounds pde case} the construction of such bounds using Fourier series objects. We begin this section by presenting the (computer-assisted) construction of the approximate solution $u_0 \in H^l_\om.$

\subsection{Construction of the approximate solution \boldmath$u_0 \in H^l_\om$\unboldmath}\label{ssec : construction of u0}

Assume we have an approximation to \eqref{eq : f(u)=0 on X^l} given by $\tilde{U}_0 \in X^l$ such that $\tU = \pi^{N_0} \tU$ (for some $N_0 \in \mathbb{N}$). In practice, $\tilde{U}_0$ is obtained numerically by solving $\pi^{N_0}F(\pi^{N_0}U)=0$, which is a finite dimensional projection of \eqref{eq : f(u)=0 on X^l}. Using Lemma \ref{lem : gamma and Gamma properties}, $\tU$ represents a function $\tilde{u}_0 \bydef \gdag\left(\tU\right) \in L^2_\om$. However,  $\tilde{u}_0$ does not a priori have regularity and $\tilde{u}_0 \notin H^l_\om$ in general. Consequently, one needs to build $u_0 \in H^l_\om$, represented by a Fourier series $U_0 \in X^l$, such that $U_0$ is close to $\tU$ in $X^l$. Moreover, $U_0$ needs to be represented on the computer, so we also require $U_0 = \pi^{N_0} U_0$. Equivalently, we want to project $\tU$ into the set of vectors of size $N_0$ in $X^l$ representing functions in $H^l_\om$. This construction is achieved using some trace theory on hypercubes, that we introduce next.



 
 \subsubsection{Trace theorem for periodic functions}
 The idea is to first build $u_{\Omega_0}$, the restriction of $u_0$ on $\Omega_0$, such that $u_{\Omega_0} \in H^l_{per}(\Omega_0)$ and such that $u_{\Omega_0}$ has a null trace of order $k$ (see \cite{mclean2000strongly} for the definition of the trace on Sobolev spaces). Then, if $k$ is big enough, by extending $u_{\Omega_0}$ by zero outside of $\Omega_0$, we obtain a smooth $u_0$ with support on $\Omega_0$. Lemma \ref{lem:ext_zero} below justifies such a reasoning. Note that $H^k_0(\Omega_0)$ denotes the set of {\em trace-free elements} in $H^k(\Omega_0)$ (see \cite{adams2003sobolev} or \cite{mclean2000strongly} for a complete presentation). Consider the closed subspace $H^k_{0,per}(\Omega_0) \bydef H^k_0(\Omega_0) \bigcap H^k_{per}(\Omega_0) $ and denote by $X^k_0$ the correspondence of $H^k_{0,per}(\Omega_0)$ in sequence space.

Define  $\tilde{X}^k$ as the equivalent of $X^k$ for sequences indexed on $\mathbb{Z}^{m-1}$ and define $\mathbb{X}^k$ as follows
 \vspace{-.2cm}
\[
\tilde{X}^k \bydef \left\{ (u_n)_{n \in \mathbb{Z}^{m-1}} : \sum_{n \in \mathbb{Z}^{m-1}} |u_n|^2 (1+|\tilde{n}|^2)^k < \infty \right\} ~~\text{ and }~~\mathbb{X}^k \bydef \prod_{i=1}^{m}\prod_{j=0}^{k-1}\tilde{X}^{k-j-\frac{1}{2}},
\vspace{-.2cm}
\]
where $\tilde{n}$ is defined in Definition \ref{def: tilde index}. Then, a continuous trace operator $\mathcal{T}_k : X^k \to \mathbb{X}^k$ can be defined by
\vspace{-.2cm}
\begin{equation}\label{eq : formula_trace}
\mathcal{T}_k \bydef \prod_{j=0}^{k-1}\mathcal{T}^{1,j} \times \cdots \times \prod_{j=0}^{k-1}\mathcal{T}^{m,j}, ~~ \text{where }  (\mathcal{T}^{i,j}(U))_n \bydef \displaystyle\sum_{p \in \mathbb{Z}}(-1)^p\big(\frac{\pi p}{d}\big)^j u_{(n_1,\dots,n_{i-1},p,n_{i+1},...,n_m)} 
\vspace{-.2cm}
\end{equation}
 for all $n = (n_1,...,n_{i-1},n_{i+1},...,n_m) \in \mathbb{Z}^{m-1}$. Note that $\prod_{j=0}^{k-1}\mathcal{T}^{i,j}$ represents the trace component on the $i^{th}$ ($m-1$)-cube of $\partial\Omega_0$. In fact, the periodicity enables to simplify the definition of the trace operator and to only consider half the faces of $\Omega_0.$


\subsubsection{Finite dimensional trace theorem}\label{ssec : finite dimensional trace theorem}

 Let $k \in \mathbb{N}$ and let  $N_0  \in \mathbb{N}$. Recall that we want to build $U_0 \in X^k_0$ such that $U_0 = \pi^{N_0}U_0$ and $\|\Tilde{U}_0-U_0\|_l$ is small.

First, define  $\mathcal{T}^{N_0}_k$ as $\mathcal{T}^{N_0}_k(U)\bydef \mathcal{T}_k(\pi^N(U))$ for all $U \in X^k$. In fact, $\mathcal{T}^{N_0}_k$ can be seen as a finite dimensional operator $\mathcal{T}^{N_0}_k : \mathbb{C}^{(2N+1)^m} \to \mathbb{C}^{mk(2N+1)^{m-1}}$  that can be represented by an $mk(2N+1)^{m-1}$ by $(2N+1)^m$ matrix, that we denote $\mathbb{M}$. Now, we define $X^{N_0,k}_0 \bydef Ker(\mathcal{T}^{N_0}_k)$ and obtain that
$(X^{N_0,k}_0)^\perp = R(\mathbb{M}^*)$.
This equality is an abuse of notation as elements in $R(\mathbb{M}^*)$ are finite dimensional vectors. We need to pad with zeros these elements for the equality to be rigorous.
In particular $X^{N_0,k}_0 \subset X^k_0$.
Now, a characterization of different projections on $X^{N_0,k}_0$, depending on some invertible matrices $D$,  can be obtained. Again we abuse notation in the next Theorem \ref{th:trace_theorem} by considering elements $U \in X^l$ such that $\pi^N(U)=U$ as vectors in $\mathbb{C}^{(2N+1)^m}$.
\begin{theorem}\label{th:trace_theorem}
Suppose $Rank(\mathbb{M}^*) = r >0$ and $M$ is a $r$ by $(2N+1)^m$ matrix such that $R(\mathbb{M}^*) = R(M^*)$, then if $D$ is an invertible matrix of size $(2N+1)^m$ 
\vspace{-.2cm}
\begin{equation}\label{eq : projection in X^k_0}
    P^N_k(U) = U - DM^*(MDM^*)^{-1}M\pi^N(U) \in X^{N_0,k}_0
    \vspace{-.2cm}
\end{equation}
for all $U \in X^l$ such that $\pi^N(U)=U$.
\end{theorem}

\begin{remark}\label{re:projection}
In the above Theorem \ref{th:trace_theorem}, the choice of the matrix $D$ allows to choose the rate of decay of $U$. In particular, $P^N_k(U)$ minimizes $\|D^{-1}(U-V)\|_2$ with $V \in X^k$ such that $\pi^N(V) = V$. A possible choice for  $D$ is therefore to chose the diagonal matrix with entries $(\frac{1}{l(\tilde{n})})_n$ on the diagonal. One could potentially play with the matrix $D$ to optimize the computation of the bounds introduced in Theorem \ref{th: radii polynomial}. Moreover, in practice we assume $2N+1 > mk$. Indeed, as $\mathbb{M}$ is a $mk(2N+1)^{m-1}$ by $(2N+1)^m$ matrix, assuming $2N+1 > mk$ allows the projection on $X^{N_0,k}_0$ to be non-trivial. 
\end{remark}

\begin{remark}
    In practice, one has to compute explicitly the rank of $\mathbb{M}^*$ in order to build $M$. This computation can be obtained thanks to some classical results on trace theory on hypercubes. In particular, \cite{trace_polynomials} provides that the size of the rank can be obtained by computing the number of required compatibility conditions. Once $r$ is known, $mk(2N+1)^{m-1} - r$ rows of $\mathbb{M}$ can be removed. The choice of these rows can be obtained and justified numerically. Indeed, one can validate that $M^*$ has a rank $r$ by verifying the existence of the inverse of $MM^*$  using interval arithmetic.
\end{remark}

\subsubsection{Construction of functions in \boldmath$H^l_\om$\unboldmath}\label{ssec : construction of u0 in Hlom}

We go back to the initial goal of this section which was to build functions in $H^l_\om$ starting from a vector of Fourier series.
Suppose that we have access to some approximation $\tilde{u}_0 \in L^2_\om$, where we denote $\tilde{U}_0 \bydef (a_n)_{n \in \mathbb{Z}}$ the Fourier coefficients of $\tilde{u}_0$. In particular we assume that $\pi^{N_0}\tilde{U}_0 = \tilde{U}_0$ and we want to build a smooth function $u_0 \in H^l$ from $\tilde{u}_0 \in L^2_\om.$ This construction is achieved combining the next Lemma \ref{lem:ext_zero} from \cite{mclean2000strongly} and the construction \eqref{eq : projection in X^k_0}.
    \begin{lemma}\label{lem:ext_zero}
    Let $\widetilde{H}^k(\Omega_0) \bydef \{u \in L^2(\Omega_0) : \cha u \in H^k(\mathbb{R}^m)\}$. 
Then $\widetilde{H}^k(\Omega_0) = H^k_0(\Omega_0)$.
\end{lemma}
    The previous Lemma \ref{lem:ext_zero} provides that a smooth extension by zero can be obtained for functions in $H^k_0(\Omega_0)$. Consequently, if $k$ is chosen big enough so that $H^l \subset H^k(\R^m)$, then the construction \eqref{eq : projection in X^k_0} allows to build $u_0 \in H^l_\om$.  Let us fix $k \in \mathbb{N}$ that satisfies the aforementioned requirement.
    
   Now let $D$ be the matrix described in Remark~\ref{re:projection}, then let $U_0 \bydef P_{k}^{N_0}(\tilde{U}_0) \in X^{N_0,k}_0 $ as in Theorem~\ref{th:trace_theorem}. The finite dimensional vector $U_0$ has a function representation $u_0$ in $H^k_{per,0}(\Omega_0)$ and by Lemma~\ref{lem:ext_zero} $u_0$ has a zero extension to a function in $H^k(\R^m) \subset H^l$. To keep our notations, we still denote $u_0 \in H^k(\R^m) \subset H^l$ the zero extension on $\mathbb{R}^m$. By construction, we obtain $u_0 \bydef \gamma^\dagger\left(U_0\right) \in H^l_\om$.

\begin{remark}
Using some rigorous numerical interval arithmetic (such as IntervalArithmetics.jl \cite{julia_interval} on Julia), Theorem \ref{th:trace_theorem} provides a way to  obtain trace-free elements numerically. Therefore the approximate solution $u_0 = \gamma^\dagger\left(U_0\right) \in H^l_\om$ will be defined by its finite number of Fourier coefficients using the construction \eqref{eq : projection in X^k_0}.
 \end{remark}

\subsection{Newton-Kantorovich approach}\label{ssec : radii polynomial}

 As explained in the previous Section~\ref{ssec : construction of u0 in Hlom}, we now have a way to construct an approximate solution $u_0 \in H^l_\om$ (via finite Fourier series in $X^l$) and thanks to Theorem~\ref{th : approximation of inverse and norm}, we can construct an approximate inverse $\mathbb{A}$ of $D\mathbb{F}(u_0)$ using a Fourier series representation. We have now all the necessary ingredients to use a Newton-Kantorovich approach, which we 
present in a general setting, and which can either be applied to the PDE case described in Section~\ref{sec:bound_inverse} or to the constrained PDE case in Section~\ref{sec:constraigned}.

\begin{theorem}[\bf Newton-Kantorovich Theorem] \label{th: radii polynomial}
Let $X,Y$ be Banach spaces. Moreover, let $\mathbb{F} : X\to Y$ be a Fréchet differentiable map and let $x \in X$. Then, let $\mathbb{A} : Y \to X$ be an injective bounded linear operator. Let $\mathcal{Y}_0, \mathcal{Z}_1$ be non-negative constants and $\mathcal{Z}_2 : (0, \infty) \to [0,\infty)$ be a non-negative function such that
  \begin{align}
    \|\mathbb{A}\mathbb{F}(x)\|_X \leq &\mathcal{Y}_0\label{def : Y0}\\
    \|I_d - \mathbb{A}D\mathbb{F}(x)\|_{X} \leq &\mathcal{Z}_1\label{def : Z1}\\
    \|\mathbb{A}\left({D}\mathbb{F}(v) - D\mathbb{F}(x)\right)\|_X \leq &\mathcal{Z}_2(r)r, ~~ \text{for all } v \in \overline{B_r(x)} \text{ and all } r>0.\label{def : Z2}
\end{align}  
If there exists $r>0$ such that
\begin{equation}\label{condition radii polynomial}
    \frac{1}{2}\mathcal{Z}_2(r)r^2 - (1-\mathcal{Z}_1)r + \mathcal{Y}_0 < 0 \text{ and } \mathcal{Z}_1 + \mathcal{Z}_2(r)r < 1,
 \end{equation}
then there exists a unique $\tilde{x} \in \overline{B_r(x)} \subset X$ such that $\mathbb{F}(\tilde{x})=0$, where $B_r(x)$ is the open ball of $X$ centered at $x$. 
\end{theorem}

\begin{proof}
The proof of the theorem can be found in \cite{periodic_navier_spontaneous}. In particular, the map $\mathbb{T} \colon \overline{B_r(x)} \to \overline{B_r(x)}$ defined by $T(u) = u - \mathbb{A} \mathbb{F}(u)$ is well-defined and contracting if \eqref{condition radii polynomial} is satisfied.
\end{proof}

\begin{remark}\label{rem : Z1 < 1 remark}
    Notice that if $r>0$ satisfies \eqref{condition radii polynomial}, then $(1 - \mathcal{Z}_1)r >  \frac{1}{2}\mathcal{Z}_2(r)r^2 + \mathcal{Y}_0 > 0$ and therefore $1 - \mathcal{Z}_1>0$. Consequently, the assumption that $\mathbb{A}$ is injective becomes redundant in the set-up presented in Section \ref{ssec : upper bound for the norm of the inverse} (see Theorem \ref{th : approximation of inverse and norm}).
\end{remark}

\subsection{Computation of the bounds}\label{ssec : bounds pde case}
We now expose the set-up to apply Theorem \ref{th: radii polynomial} to the PDE \eqref{eq : f(u)=0 on H^l}. In other words, we choose $X= H^l$, $Y=L^2$, and we study $\mathbb{F} : H^l \to L^2$ defined in \eqref{eq : f(u)=0 on H^l}, which, under Assumptions \ref{ass:A(1)} and \ref{ass : LinvG in L1}, is smooth on $H^l$. We present explicit computations for the bounds $\mathcal{Y}_0$, $\mathcal{Z}_1$ and $\mathcal{Z}_2$ using the constructions of $\A$ in Section~\ref{sec:bound_inverse} and of $u_0$ in Section~\ref{ssec : construction of u0}. In particular, we derive upper bounds that only require the norm computation of finite dimensional objects. Such norms can then be obtained on a computer using the arithmetic on intervals (cf. \cite{Moore_interval_analysis} and \cite{julia_interval}). Consequently, once $\A$ and $u_0 \in H^l_\om$ are obtained, one can verify the hypotheses of Theorem \ref{th: radii polynomial} on a computer to prove the existence of a solution to \eqref{eq : f(u)=0 on H^l}.    

We fix some $N_0, N \in \mathbb{N}$, where $N_0$ represents the numerical size of the sequences (that is $U = \pi^{N_0}U$ for all the sequences $U$ we consider numerically) and $N$ represents the numerical size of the operators (that is $M = \pi^N M \pi^N$ for all the operators on Fourier series we consider numerically).  Using the construction of Section~\ref{ssec : construction of u0} and \eqref{eq : projection in X^k_0}, from a numerical approximation $\tU$ of $\pi^{N_0}F(\pi^{N_0}U)=0$, one can build an approximate solution $u_0 \in H^l_\om$ represented by its Fourier series $U_0 \in X^l$. Moreover, $U_0 = \pi^{N_0}U_0$ only has a finite number of non-zero coefficients.

   To apply Theorem \ref{th: radii polynomial} requires the construction of $\mathbb{A}$. This is achieved using the construction of Section \ref{ssec : construction of the approximate inverse}. Indeed, we define $\mathbb{A} \bydef \mathbb{L}^{-1}\left(\out + \Gamma\left(\pi_N + B^N\right)\right) : L^2\to H^l$, as in \eqref{def : definition of mathbb A}, where $B^N = \pi^N B^N \pi^N$ is chosen numerically.
   
   We first take care of the bound $\mathcal{Z}_1$. This bound is decomposed into a first part $\mathcal{Z}_u$, that we compute thanks to Theorems \ref{th : approximation of inverse and norm} and \ref{th : Zu computation}, and a second part $Z_1$ that is based on Fourier series quantities. In particular, $Z_1$ can be computed numerically thanks to vectors and matrices norms. We summarize the obtained results in the following lemma.

   \begin{lemma}[\bf Bound \boldmath$\mathcal{Z}_{1}$\unboldmath]\label{lem : bound Z1}
  Let $Z^N_1$ and $Z_1$ be such that
\begin{align}
   \left\|\pi^N - B^N\left(I_d + DG(U_0)L^{-1}\right)\pi^{N_\mathbb{G}N}\right\|_2^2 + \left\|(\pi^{N_\mathbb{G}N}-\pi^N) DG(U_0)L^{-1}\pi^N\right\|_2^2 & \le (Z_1^N)^2 \label{def : Z1N definition}\\
    \left((Z_1^N)^2 + \left(\sum_{\alpha \in J_\G}\|V_\alpha\|_1\max_{n \in \mathbb{Z}^m\setminus I^N}\frac{|(2\pi n)^\alpha|}{|l(\tilde{n})|}\right)^2\right)^{\frac{1}{2}}  &\leq Z_1 \label{def : Z1 definition in Lemma}.
\end{align}
In particular, $\|I_d - ADF(U_0)\|_l \leq Z_1$. Moreover, let $\mathcal{Z}_u$ satisfying \eqref{eq : Zu defnition with Zu1 Zu2}. If $\mathcal{Z}_1 \bydef Z_1 + \mathcal{Z}_u$, then
\[\|I_d - \mathbb{A}D\mathbb{F}(u_0)\|_l \leq \mathcal{Z}_1.\]
   \end{lemma}

   \begin{proof}
       Using \eqref{eq:Z1_inequality_proof}, we have
       \begin{align*}
           \|I_d - ADF(U_0)\|_l = \|I_d - B\left(I_d + DG(U_0)L^{-1}\right)\|_2.
       \end{align*}
       Then, using that $B = I_d + B^N$ and $B^N = \pi^NB^N \pi^N$, we get
       \begin{align*}
           \|I_d - B\left(I_d + DG(U_0)L^{-1}\right)\|_2^2 &\leq \|\pi_N - \pi_NB\left(I_d + DG(U_0)L^{-1}\right)\|_2^2 \\ 
           & \quad + \|\pi^N - \pi^NB\left(I_d + DG(U_0)L^{-1}\right)\|_2^2\\
           &= \|\pi_N DG(U_0)L^{-1}\|_2^2 + \|\pi^N - B^N\left(I_d + DG(U_0)L^{-1}\right)\|_2^2.
       \end{align*}
     
     From  \eqref{def: G and j}, the monomials of $\G$ are of order at most $N_\mathbb{G}$, and hence $\pi^NDG(U_0) = \pi^NDG(U_0)\pi^{N_\mathbb{G}N}$. Moreover, as $L^{-1}$ is diagonal, we have $\pi^{N_\mathbb{G}N}L^{-1} = L^{-1}\pi^{N_\mathbb{G}N}$ and 
     \begin{equation*}
         \left\|\pi^N - B^N\left(I_d + DG(U_0)L^{-1}\right)\right\|_2 = \left\|\pi^N - B^N\left(I_d + DG(U_0)L^{-1}\right)\pi^{N_\mathbb{G}N}\right\|_2.
     \end{equation*}
     Then, we get $\|\pi_N DG(U_0)L^{-1}\|_2^2 \leq \|\pi_N DG(U_0)L^{-1}\pi^N\|_2^2 + \|\pi_N DG(U_0)L^{-1}\pi_N\|_2^2$.
     Using again \eqref{def: G and j} and that $L^{-1}$ is diagonal, we have $DG(U_0)L^{-1}\pi^N = \pi^{N_\mathbb{G}N}DG(U_0)L^{-1}\pi^N$. This implies
     \begin{equation*}
          \left\|\pi_N DG(U_0)L^{-1}\pi^N\right\|_2 = \left\|(\pi^{N_\mathbb{G}N}-\pi^N) DG(U_0)L^{-1}\pi^N\right\|_2.
     \end{equation*}
     Finally, using \eqref{eq:assumption_on_G} and recalling the notation \eqref{def : discrete conv operator}, we have $DG(U_0) = \sum_{\alpha \in J_\G} \mathbb{V}_\alpha \partial^\alpha.$
     Therefore,
     \[
         \|\pi_N DG(U_0)L^{-1}\pi_N\|_2 = \|\pi_N \sum_{\alpha \in J_\G} \mathbb{V}_\alpha \partial^\alpha L^{-1}\pi_N\|_2 
         \leq  \sum_{\alpha \in J_\G} \|\mathbb{V}_\alpha\|_2 \|\partial^\alpha L^{-1}\pi_N\|_2.
     \]
     However, using Young's inequality for the convolution we have $\|\mathbb{V}_\alpha\|_2 \leq \|V_\alpha\|_1.$ Moreover, using Assumption \ref{ass : LinvG in L1} and that $\partial^\alpha L^{-1}$ is diagonal, we have
     \begin{align*}
         \|\partial^\alpha L^{-1}\pi_N\|_2 = \max_{n \in \mathbb{Z}^m\setminus I^N}\frac{|(2\pi n)^\alpha|}{|l(\tilde{n})|}.
     \end{align*}
  Therefore, using \eqref{def : Z1N definition} and \eqref{def : Z1 definition in Lemma}, we obtain that $\|I_d - ADF(U_0)\|_l \leq Z_1$. The rest of the proof follows directly from Theorems \ref{th : approximation of inverse and norm} and \ref{th : Zu computation}.
  \end{proof}

  \begin{remark}\label{rem : Z1 periodic and diag dominance}
   Note that  ${Z}_1$ (cf. \eqref{def : Z1N definition} and \eqref{def : Z1 definition in Lemma}) is obtained via the computation of norms of vectors and matrices, which can be processed numerically. Moreover, since $U_0 = \pi^{N_0}U_0$, we have $V_\alpha = \pi^{N_\mathbb{G}N_0}V_\alpha$  for each $\alpha \in J_\G$. Therefore, $$\left(V_\alpha^*,V_\alpha*E_i\right)_2 = \left(V_\alpha^*,V_\alpha*(\pi^{2N_\mathbb{G}N_0}E_i)\right)_2$$
  where $i \in \{1, 2\}$. In particular, this implies that both $\mathcal{Z}_{u,1}^\alpha$ and $\mathcal{Z}_{u,2}^\alpha$ (satisfying \eqref{ineq : bound Zu1 in lemma} and \eqref{ineq : bound on Kalpha in Z1} respectively) can be obtained throughout a finite number of computations. Consequently, the bound $\mathcal{Z}_1$ can be fully computed throughout rigorous numerics (see \cite{julia_interval} for instance).\\
  Also, note that Remark \ref{rem : L1 property of J_G} provides that  the quantity
  \[
  \max_{n \in \mathbb{Z}^m\setminus I^N}\frac{|(2\pi n)^\alpha|}{|l(\tilde{n})|}
  \]
  goes to zero as $N$ goes to infinity, for all $\alpha \in J_\mathbb{G}$. As described in the introduction, the differential operator (in its Fourier series representation) allows to obtain a diagonal dominance structure (hence compactness) for the linearized operator. In particular, this justifies the approximation of the inverse of the linearized operator by a finite truncation and a diagonal tail (cf. \eqref{eq : decomposition of B}).  
  \end{remark}

Now, using the smoothness of the map $\mathbb{F}$, the bound $\mathcal{Z}_2$ satisfying \eqref{def : Z2} can directly be computed using the mean value theorem for Banach spaces, a result which we introduce next.
   \begin{lemma}[\bf Bound \boldmath$\mathcal{Z}_{2}$\unboldmath]\label{lem : Z2 bound}
   Let $r>0$ and let $\mathcal{Z}_2(r) >0$ be such that
   \begin{equation}\label{eq : Z2 in lemma}
       \mathcal{Z}_2(r) \geq  \max\{1, \|B^N\|_2\}\sup_{h \in B_r(u_0)}\|D^2\mathbb{F}(h)\|_{\mathcal{B}(H^l,L^2)}
   \end{equation}
   where $\|\cdot\|_{\mathcal{B}(H^l,L^2)}$ is the operator norm from bounded linear operators in $H^l$ to bounded linear operators in $L^2.$  Then $\mathcal{Z}_2(r)r \geq  \|\mathbb{A}\left(D\mathbb{F}(v)-D\mathbb{F}(u_0\right)\|_l$ for all $v \in \overline{B_r(u_0)} \subset H^l.$
   \end{lemma}
   
   \begin{proof}
   Let $v \in B_r(u_0) \subset H^l$, by definition of the $H^l$ norm, we have
   \begin{align*}
       \|\mathbb{A}\left(D\mathbb{F}(v)-D\mathbb{F}(u_0\right)\|_l = \|\mathbb{B}\left(D\mathbb{F}(v)-D\mathbb{F}(u_0\right)\|_{l,2}.
   \end{align*}
   Then, using the mean value inequality for Banach spaces, we get
   \begin{align*}
       \|\mathbb{B}\left(D\mathbb{F}(v)-D\mathbb{F}(u_0\right)\|_{l,2} \leq \sup_{h \in B_r(u_0)} \|\mathbb{B}D^2\mathbb{F}(h)\|_{\mathcal{B}(H^l,L^2)}r\leq \|\mathbb{B}\|_2\sup_{h \in B_r(u_0)}\|D^2\mathbb{F}(h)\|_{\mathcal{B}(H^l,L^2)}r
       \end{align*}
      as $\|v-u_0\|_l \leq r$. We conclude the proof using \eqref{eq:computation_norm_A}.
   \end{proof}
   In practice, the term $D^2\mathbb{F}(h)$ involves a mix of multiplication operators composed with differential operators. Therefore, its norm computation relies on estimates similar to the ones presented in Lemma \ref{Banach algebra}. We illustrate such a computation in the case of the Kawahara equation in Lemma \ref{lem : Z_2 bound Kdv}.   Finally, the computation for the bound $\mathcal{Y}_0$ is obtained using Parseval's identity and the 2 norm of a vector.  The next result provides the details for such a computation.

\begin{lemma}[\bf Bound \boldmath$\mathcal{Y}_{0}$\unboldmath]\label{lem : bound Y0}
Recall \eqref{eq : f(u)=0 on X^l} and $N_{\mathbb{G}} \in \mathbb{N}$ from \eqref{def: G and j}. Let $\mathcal{Y}_0 >0$ be such that 
\begin{equation}
    \mathcal{Y}_0 \geq  |\om|^{\frac{1}{2}}\left(\left\|B^NF(U_0)\right\|_2^2 + \left\|(\pi^{N_0}-\pi^N)LU_0 + (\pi^{N_\mathbb{G}N_0}-\pi^N) G(U_0) - \pi_N \Psi\right\|_2^2 \right)^{\frac{1}{2}}.
\end{equation}
Then $\|\mathbb{A}\mathbb{F}(u_0)\|_l \leq \mathcal{Y}_0.$
\end{lemma}

\begin{proof}
Using \eqref{eq : norm_equality_Hl_L2}, we have $\|u\|_l = \|\mathbb{L}u\|_2$ for all $u \in H^l$. This implies
\[
\|\mathbb{A}\mathbb{F}(u_0)\|_l = \|\mathbb{L}\mathbb{A}\mathbb{F}(u_0)\|_2 = \|\mathbb{B}\mathbb{F}(u_0)\|_2.
\]
But now, combining Lemma \ref{lem : gamma and Gamma properties} and $u_0 = \gamma^\dagger\left(U_0\right) \in H^l_\om$, we obtain that $\mathbb{B}\mathbb{F}(u_0) = \gdag\left(B F(U_0)\right)$  and $$\|\mathbb{B}\mathbb{F}(u_0)\|_2 = \sqrt{|\om|} \|BF(U_0)\|_2.$$ 
Moreover, we have that $U_0 = \pi^{N_0}U_0$. In particular, $LU_0 = \pi^{N}LU_0 + (\pi^{N_0}-\pi^N)LU_0$, $\Psi = \pi^{N}\Psi + \pi_N\Psi$ and $G(U_0) = \pi^N G(U_0) + \pi_N G(U_0)$. We thus obtain
\begin{align*}
  \|BF(U_0)\|_2^2 =& \|B^N(LU_0 + G(U_0) -  \Psi)\|_2^2 + \|\pi_N (L(U_0) + G(U_0) -  \Psi )\|_2^2 \\
  =& \|B^NF(U_0)\|_2^2 + \|(\pi^{N_0}-\pi^N)LU_0 + \pi_N G(U_0) - \pi_N \Psi\|_2^2 .  
\end{align*}
Since $G$ is polynomial of order $N_{\mathbb{G}}$, $G(U_0) = \pi^{N_\mathbb{G}N_0}G(U_0)$ and so $\pi_N G(U_0) = (\pi^{N_\mathbb{G}N_0}-\pi^N)G(U_0)$. 
\end{proof}

\begin{remark}
    The computation of the bounds $\mathcal{Z}_1$ and $\mathcal{Z}_2$ depend on the computation of 2-norms of matrices using interval arithmetic. If a sharp precision is required, one can compute rigorous bounds for singular values (see \cite{rump_spectral_norm} for instance). Otherwise, one can use generalized Holder's inequality 
    \[
    \|M\|_2 \leq \|M\|_1 \|M\|_\infty.
    \]
    Moreover, in practice one might need to take $d$ big in order for $\mathcal{Z}_1$ to be strictly smaller than 1 (cf. condition \eqref{th: radii polynomial}). However, this comes at the cost of damaging the decay of the Fourier coefficients of $u_0.$ Therefore, it can be useful to choose $N_0 > N$ to gain precision on the approximate solution and obtain a sharp $\mathcal{Y}_0$ bound.
\end{remark}

Due to the foundation of our analysis on Fourier series, one can notice a lot of similarities with the computer-assisted proofs of periodic solutions using the Radii Polynomial approach (see \cite{periodic_navier_spontaneous}). More specifically, the bounds $\mathcal{Y}_0$, $\mathcal{Z}_1$ and $\mathcal{Z}_2$ have corresponding bounds $Y_0$, $Z_1$ and $Z_2$ associated to the periodic problem on $\om.$ These similarities allow to derive a condition \eqref{condition for Z2 periodic}, under which a proof of a solution on $\mathbb{R}^m$ using Theorem \ref{th: radii polynomial} implies a proof of existence of a periodic solution.

\begin{theorem}[\bf Existence of periodic solutions]\label{th : radii periodic}
Let $\mathbb{A} : L^2 \to H^l$ be defined in  Section \ref{sec:bound_inverse} and let $A : \ell^2 \to X^l$ be its Fourier series correspondent. Let $\mathcal{Y}_0, \mathcal{Z}_1 \bydef Z_1 + \mathcal{Z}_{u}$ (where $Z_1$ is defined in \eqref{def : Z1 definition in Lemma}) and $\mathcal{Z}_2$ be the associated bounds defined in Lemmas \ref{lem : bound Y0}, \ref{lem : bound Z1} and \ref{lem : Z2 bound} respectively. Assume that there exists $r>0$ such that \eqref{condition radii polynomial} holds. Moreover, assume that 
\begin{equation} \label{condition for Z2 periodic}
    \|A(DF(V)-DF(U_0))\|_l \leq \mathcal{Z}_2(\sqrt{|\om|} r)\sqrt{|\om|}r, \quad \text{for all } V \in B_r(U_0).
\end{equation} 
Then, defining $Y_0 \bydef \frac{\mathcal{Y}_0}{\sqrt{|\om|}}$ and $Z_2(r) \bydef \mathcal{Z}_2(\sqrt{|\om|}r)\sqrt{|\om|}$, we have
\begin{align*}
    \|AF(U_0)\|_l &\leq Y_0\\
    \|I_d - ADF(U_0)\|_l &\leq Z_1\\
     \|A(DF(V)-DF(U_0))\|_l &\leq Z_2(r)r, \quad \text{for all } V \in B_r(U_0).
\end{align*}
Letting $\tilde r \bydef \frac{r}{\sqrt{|\om|}}$, we get that 
\[
\frac{1}{2}Z_2(\tilde r)\tilde r^2 - (1-Z_1)\tilde r + Y_0 < 0 \text{ and } {Z}_1 + {Z}_2(\tilde{r})\tilde{r} < 1,
\]
and hence, there exists a unique solution $\tilde{U}$ to \eqref{eq : f(u)=0 on X^l} in $\overline{B_{\tilde r}(U_0)}$, where $B_{\tilde r}(U_0)$ is the open ball in $X^l$ of radius $\tilde r$ and centered at $U_0.$
\end{theorem}
\begin{proof}
 First, using the proof of Lemma \ref{lem : bound Y0}, we have
\begin{align*}
    \|\mathbb{A}\mathbb{F}(u_0)\|_l =  \sqrt{|\om|}\|AF(U_0)\|_l \leq \mathcal{Y}_0. 
\end{align*}
This yields that $\|AF(U_0)\|_l \leq Y_0$. Then $\|I_d - ADF(U_0)\|_l \leq Z_1$ is true by definition of $Z_1$ in Lemma \ref{lem : bound Z1} and $ \|A(DF(V)-DF(U_0))\|_l \leq Z_2(r)r $ for all $V \in B_r(U_0)$ by assumption. 

Now, suppose that $r>0$ satisfies \eqref{condition radii polynomial} and let $\tilde{r} \bydef \frac{r}{\sqrt{|\om|}}$. Then let  $T : X^l \to X^l$ be defined as 
\begin{align*}
    T(U) \bydef U - AF(U).
\end{align*}
We want to prove that $T : \overline{B_{\tilde{r}}(U_0)} \to \overline{B_{\tilde{r}}(U_0)}$ is well-defined and is a contraction. Notice that $Z_1 \leq \mathcal{Z}_1$, then by construction of $Y_0$ and $Z_2$, we have
\begin{align*}
    \frac{1}{2}Z_2(\tilde{r})\tilde{r}^2 - (1-Z_1)\tilde{r} + Y_0 &\leq  \frac{1}{2}\mathcal{Z}_2(r)\frac{r^2}{\sqrt{|\om|}} -  (1-\mathcal{Z}_1)\frac{r}{\sqrt{|\om|}} + \frac{\mathcal{Y}_0}{\sqrt{|\om|}}\\
    &= \frac{1}{\sqrt{|\om|}}\left(\frac{1}{2}\mathcal{Z}_2(r){r^2}-  (1-\mathcal{Z}_1){r}+ {\mathcal{Y}_0}\right)<0
\end{align*}
 using \eqref{condition radii polynomial}. Similarly, $ {Z}_1 + {Z}_2(\tilde{r})\tilde{r} < 1$. Moreover, by construction of $\mathbb{A}$, we have that $A$ is injective if and only if  $\mathbb{A}$ is injective, which is the case if \eqref{condition radii polynomial} is satisfied (see Remark \ref{rem : Z1 < 1 remark}).
 Therefore, we apply  Theorem \ref{th: radii polynomial} and  obtain that $T : \overline{B_{\tilde{r}}(U_0)} \to \overline{B_{\tilde{r}}(U_0)}$ is well-defined and is a contraction. This concludes the proof.
\end{proof}

Once a solution $\tilde u \in B_r(u_0) \subset H^l$ to \eqref{eq : f(u)=0 on H^l} has been obtained through the use of Theorem \ref{th: radii polynomial}, i.e bounds $\mathcal{Y}_0, \mathcal{Z}_1$ and $\mathcal{Z}_2(r)$ satisfying \eqref{condition radii polynomial} for some $r>0$ have been obtained, then one can try to check if \eqref{condition for Z2 periodic} is satisfied in order to obtain a solution to \eqref{eq : f(u)=0 on X^l} in $X^l$. Using a Riemann sum argument (as illustrated in Remark \ref{rem : riemann sum idea}), one can show that the condition \eqref{condition for Z2 periodic} tends to the condition \eqref{def : Z2} as $d \to \infty$. Therefore, if $d$ is big enough, then \eqref{condition for Z2 periodic} is usually easy to reach.

\section{Constrained PDEs and Eigenvalue Problems}\label{sec:constraigned}

In this section, we introduce an extension of the analysis presented in Sections \ref{sec:bound_inverse} and \ref{sec : computer assisted analysis} to treat the case of a PDE coupled with an extra equation and for which we solve for an extra parameter. The most obvious example of this extension is the case of an eigenvalue problem (see Section \ref{ssec:isolated}) where one looks for an  eigencouple $(\lambda, u) \in \mathbb{C} \times H^l$, where $u$ solves a linear PDE which depends on an eigenvalue $\lambda$ and where an extra equation is imposed to fix the scaling invariance of the eigenvector $u$. A second application of this extension is given by PDEs having conserved quantities leading to invariances, where one can use those conserved quantities as extra equations and use unfolding parameters to desingularize the problem (see \cite{Lessard2021conserved,Lessard2021conserved_nbody} for instance). The process of desingularization is useful to prove existence of solutions using a contraction mapping argument. A last example worth mentioning is the case of computing branches of solutions via a pseudo-arclength continuation algorithm, where a natural parameter of the PDE becomes an extra variable and where a hyperplane equation is appended to the problem (cf. \cite{MR0910499,lessard_global_smooth_curves}).

While in this paper, we only consider the case of a PDE associated with one extra equation and one extra parameter, our approach can easily be generalized for multiple constraints with the same number of extra parameters. We begin by introducing the set-up of the problem.

\subsection{Set-up of the augmented zero finding problem}

We assume that the linear operator $\mathbb{L}=\mathbb{L}(\nu)$ and the differential operators in $\mathbb{G}=\mathbb{G}(\nu)$ (denoted $\mathbb{G}^p_{i,k}$ in Assumption~\ref{ass : LinvG in L1}) depend polynomially on a parameter $\nu \in \mathbb{C}$. As in Section~\ref{sec:bound_inverse}, we fix some $u_0 \in H^l$ such that supp$(u_0) \subset \overline{\Omega_0}$ ($u_0 \in H^l_\om$). Moreover, we fix some $\nu_0 \in \mathbb{C}$ such that $\mathbb{L}(\nu_0)$ is invertible in order to ensure that Assumption~\ref{ass:A(1)} is still satisfied at $\nu_0$. In particular, we define the norm on $H^l$ as 
\[
\|u\|_l^2 \bydef \int_{\mathbb{R}^m} |\hat{u}(\xi)|^2|l_{\nu_0}(\xi)|^2d\xi
\]
where $l_{\nu_0}$ ($|l_{\nu_0}| > 0$) is the Fourier transform of $\mathbb{L}(\nu_0)$.
Denote by $H_1 \bydef \mathbb{C} \times H^l$ and $H_2 \bydef \mathbb{C} \times L^2$ the Hilbert spaces endowed respectively with the norms
\[
    \|(\nu,u)\|_{H_1} \bydef \left( |\nu|^2 + \|u\|_l^2\right)^{\frac{1}{2}} ~~ \text{ and } ~~
    \|(\nu,u)\|_{H_2} \bydef \left( |\nu|^2 + \|u\|_2^2\right)^{\frac{1}{2}}.
\]
Let $\phi \colon H^l \to \mathbb{C}$ be a bounded linear operator  such that
\[\phi(u) \bydef (\mathbb{L}(\nu_0)^{-1}\phi_0,u)_l = (\phi_0,\mathbb{L}(\nu_0)u)_2\]
for all $u \in H^l$, where $\phi_0 \in L^2_\om$ is fixed. We also denote by $\phi_0^*$ the dual of $\phi_0$, that is the bounded linear functional associated to $\phi_0$ in $L^2.$ In other words,
\[
\phi_0^*(u) = (\phi_0,u)_2, \quad \text{for all } u \in L^2. 
\]
Then, denote $x = (\nu,u) \in H_1$ and define the augmented zero finding problem $\overline{\mathbb{F}} : H_1 \to H_2$ as 
\begin{equation}\label{eq : f(u) = 0 on H1}
    \overline{\mathbb{F}}(x) \bydef \begin{pmatrix}\mathbb{\phi}(u)+ \eta\\
    \mathbb{F}(x)\end{pmatrix} =0
\end{equation}
for all $x = (\nu,u) \in H_1$, where ${\mathbb{F}}(x) \bydef \mathbb{L}(\nu)u + \mathbb{G}(\nu,u) -\psi$ and for some fixed constant $\eta \in \mathbb{R}$. Denoting $x_0 \bydef (\nu_0, u_0)$, we have that
\[
D\overline{\mathbb{F}}(x_0) = 
\begin{pmatrix}
0  & \phi(\cdot)\\
\partial_\nu\mathbb{F}(x_0) &\mathbb{L}(\nu_0) + \partial_u\mathbb{G}(x_0)
\end{pmatrix}.
\]
For the sake of simplicity of the presentation we write $\mathbb{L} = \mathbb{L}(\nu_0)$ and denote 
\[
     \oL \bydef \begin{pmatrix}
             1  & 0\\
             0 &\mathbb{L}
        \end{pmatrix}.
\]
 The augmented zero finding problem \eqref{eq : f(u) = 0 on H1} has a periodic boundary value problem correspondence. In fact, as $\phi_0 \in L^2_\om$, we can define $\Phi \in \ell^2$ to be the corresponding sequence of Fourier coefficients  of $\phi_0$ in $\ell^2$. Then we define $X_1 \bydef \mathbb{C}\times X^l$ and $X_2 \bydef \mathbb{C} \times \ell^2$ the corresponding Hilbert spaces associated to their natural norms
\[
    \|(\nu,U)\|_{X_1} \bydef \left( |\nu|^2 + {|\om|}\|U\|_l^2\right)^{\frac{1}{2}}  ~~ \text{and} ~~
    \|(\nu,U)\|_{X_2} \bydef \left( |\nu|^2 + {|\om|}\|U\|_2^2\right)^{\frac{1}{2}}.
\]
 Note that the normalization by $\frac{1}{|\om|}$ is useful to obtain an isometry (cf. Lemma \ref{lem : properties constraigned gamma}). 
 Finally, we  define a corresponding operator $\overline{F} : X_1 \to X_2$ as 
 \begin{align*}
    \overline{F}(\nu,U) \bydef \begin{pmatrix} |\om|(\Phi,L(\nu_0)U)_2 +\eta\\
    {F}(\nu,U)\end{pmatrix} 
\end{align*}
where \begin{equation}\label{def : F tilde}
    {F}(\nu,U)  \bydef L(\nu)U + G(\nu,U) - \Psi
\end{equation} for all $(\nu,U) \in X_1$. As in Section \ref{ssec:periodic_space}, $L(\nu)$ and $G$ are the Fourier coefficients operators corresponding to $\mathbb{L}(\nu)$ and $\mathbb{G}$ respectively. Moreover, we have $\Psi = \gamma\left(\psi\right)$, where $\psi$ is introduced in \eqref{eq : f(u)=0 on S}.

\subsection{Approximate inverse of \boldmath$D\overline{\mathbb{F}}(x_0)$\unboldmath}\label{sec : approximate inverse constrained}

Similarly as what was achieved in Section \ref{sec:bound_inverse}, denote $H_{2,\om} \bydef \{(\nu,u) \in H_2, ~~ \text{supp}(u) \subset \overline{\om}\}$ and $H_{1,\om}\bydef \{(\nu,u) \in H_1, ~~ \text{supp}(u) \subset \overline{\om}\}$. In particular, recall that we assumed $x_0 = (\nu_0,u_0) \in H_{1,\om}$. Now, denote $\overline{\mathbb{1}}_\om$ and $\overline{\mathbb{1}}_{\R^m\setminus \om}$ as 
\begin{align*}
    \overline{\mathbb{1}}_\om \bydef \begin{pmatrix}
        1&0\\
        0&\cha
    \end{pmatrix}~~ \text{ and }  ~~\overline{\mathbb{1}}_{\R^m\setminus \om} \bydef \begin{pmatrix}
        0&0\\
        0&\out
    \end{pmatrix}. 
\end{align*}
Then, let us define $\mathcal{B}(H_2)$ (respectively $\mathcal{B}(X_2)$) as the set of bounded linear operators on $H_2$ (respectively $X_2$). Moreover, denote by $\mathcal{B}_\om(H_2)$ the following  subset of $\mathcal{B}(H_2)$ 
\[
\mathcal{B}_\om(H_2) \bydef \left\{\overline{\B}_\om \in \mathcal{B}(H_2), ~~ \overline{\B}_\om = \overline{\mathbb{1}}_\om \overline{\B}_\om \overline{\mathbb{1}}_\om\right\}.
\]
Now,  let $\og : H_2 \to X_2$ and $\ogd : X_2 \to H_2$ be defined as follows
\begin{align}
    \og \bydef \begin{pmatrix}
        1 & 0\\
        0 & \gamma
    \end{pmatrix}
    ~~ \text{ and } ~~ \ogd \bydef \begin{pmatrix}
        1 & 0\\
        0 & \gamma^\dagger
    \end{pmatrix}.
\end{align}
Similarly, define $\oGG : \mathcal{B}(H_2) \to \mathcal{B}(X_2)$ and $\oGGD : \mathcal{B}(X_2) \to \mathcal{B}(H_2)$ as 
\begin{align}
    \oGG\left(\overline{\B}\right) \bydef \og \overline{\B} \ogd
    ~~ \text{ and } ~~ \oGGD\left(\overline{B}\right) \bydef \ogd \overline{B} \og
\end{align}
for all $\overline{\B} \in \mathcal{B}(H_2)$ and all $\overline{B} \in \mathcal{B}(X_2).$

\begin{lemma}\label{lem : properties constraigned gamma}
    The map $\og : H_{2,\om} \to X_2$ (respectively $\oGG : \mathcal{B}_\om(H_2) \to \mathcal{B}(X_2)$) is an isometric 
    isomorphism whose inverse is $\ogd : X_2 \to H_{2,\om}$ 
    (respectively $\oGGD : \mathcal{B}(X_2) \to \mathcal{B}_\om(H_2)$). In particular,
    \begin{align*}
        \|(\nu,u)\|_{H_2} = \|(\nu,U)\|_{X_2} ~~ \text{ and } ~~ \|\overline{\B}_\om\|_{H_2} = \|\overline{B}\|_{X_2} 
    \end{align*}
    for all $(\nu,u) \in H_{2,\om}$ and all $\overline{\B}_\om \in \mathcal{B}_\om(H_2)$, where $(\nu,U) \bydef \og\left(\nu,u\right)$ and $\overline{B} \bydef \oGG\left(\overline{\B}_\om\right)$.
\end{lemma}
Finally, let $N \in \mathbb{N}$ ($N$ is the size of the numerical approximation) and define $\overline{\pi}^N$ and $\overline{\pi}_N$ as 
\begin{align}
    \overline{\pi}^N \bydef \begin{pmatrix}
        1 & 0\\
        0 & \pi^N
    \end{pmatrix} 
  ~~  \text{ and } ~~ \begin{pmatrix}
        0 & 0\\
        0 & \pi_N
    \end{pmatrix}, 
\end{align}
where $\pi^N$ and $\pi_N$ are given in \eqref{def : projection on size N}.

 In  the same fashion as what is achieved in Section \ref{sec:bound_inverse}, we want to approximate the inverse of $D\oF(x_0)$ by some operator $\oA : H_2 \to H_1$ through the construction of an operator $\oB : H_2 \to H_2$, where $\oA = \oL^{-1} \oB.$ In particular, let us define
 \[
     \overline{L} \bydef \begin{pmatrix}
             1  & 0\\
             0 &{L}(\nu_0)
        \end{pmatrix}.
\]
We start by computing numerically (using floating point arithmetic), an approximate inverse $\overline{B}^N : X_2 \to X_2$ for $\overline{\pi}^ND\overline{F}(\nu_0,U_0)\overline{L}^{-1}\overline{\pi}^N$ such that $\overline{B}^N  = \overline{\pi}^N\overline{B}^N \overline{\pi}^N$. Then, we define $\oA : H_2 \to H_1$ and $\overline{A} : X_2 \to X_1$ as 
\begin{align}\label{def : operator A constraigned}
    \oA \bydef \oL^{-1}\left(\overline{\mathbb{1}}_{\R^m\setminus\om} + \oGGD\left(\overline{\pi}_N + \overline{B}^N\right) \right) ~~\text{ and } ~~ \overline{A} \bydef \overline{L}^{-1}\left(\overline{\pi}_N + \overline{B}^N\right).
\end{align}

The above construction allows to derive the equivalent of Theorem \ref{th : approximation of inverse and norm} in the case of constrained PDEs. More specifically, the next Theorem \ref{th : approximation of inverse and norm constrained PDE} provides control on the approximation $\oA$ for the inverse of $D\oF(x_0)$.

\begin{theorem}\label{th : approximation of inverse and norm constrained PDE}
Let $\oA : H_2 \to H_1$ be given in \eqref{def : operator A constraigned}. Then, let $Z_1 >0$ and  $\mathcal{Z}_{u}>0$ be such that
\begin{align*}
    \|I_d - \overline{A}D\overline{F}(\nu_0,U_0)\|_{X_1} &\leq Z_1\\
    \max\left\{1,\left\|\overline{B}^N\right\|_{X_2}\right\} \sum_{\alpha \in J_\mathbb{G}} \left\|\left(\Gamma^\dagger\left((\partial^\alpha L^{-1})^*\right) - (\partial^\alpha\iL)^*\right)\mathbb{v}_\alpha^*\right\|_2&\leq \mathcal{Z}_{u}.
\end{align*}
If $\mathcal{Z}_1 \bydef Z_1 + \mathcal{Z}_u$, then 
\begin{equation}\label{ineq : upper bound Z1 in lemma constrained}
    \left\|I_d - \oA D\oF(x_0)\right\|_l \leq \mathcal{Z}_1.
\end{equation}
Moreover, if $\mathcal{Z}_1 <1$, then both $\oA : H_2 \to H_1$ and  $D\oF(x_0) : H_1 \to H_2$ have a bounded inverse and
\begin{equation}
    \|D\oF(x_0)^{-1}\|_{H_2,H_1} \leq \frac{\max\left\{1,\left\|\overline{B}^N\right\|_{X_2}\right\}}{1-\mathcal{Z}_1}.
\end{equation}
\end{theorem}
\begin{proof}
The proof is obtained by combining Lemma \ref{lem : properties constraigned gamma} and the reasoning of the proof of Theorem \ref{th : approximation of inverse and norm} on the augmented system $\oF$. 
\end{proof}

The previous Theorem \ref{th : approximation of inverse and norm constrained PDE} provides a way to compute the defect $\|I_d - \oA D\oF(x_0)\|_l$ under the construction given in \eqref{def : operator A constraigned}. Moreover, the previous Theorem \ref{th : approximation of inverse and norm constrained PDE} provides that $\oA$ is injective if $\mathcal{Z}_1<1$ (which is required in order to use Theorem \ref{th: radii polynomial}). Therefore, using Remark \ref{rem : Z1 < 1 remark}, if \eqref{condition radii polynomial} is satisfied then assuming that $\oA$ is injective becomes redundant.  

\begin{remark}
Similarly as Remark \ref{rem : link norm of inverses}, if one has access to the quantity $\|D\overline{F}({\nu_0},U_0)^{-1}\|_{X_2,X_1}$, then
\begin{align*}
    \left\|D\overline{F}(x_0)^{-1}\right\|_{H_2,H_1} \leq \frac{\left\|D\overline{F}({\nu_0},U_0)^{-1}\right\|_{X_2,X_1}}{1-\mathcal{Z}_{u}}
\end{align*}
where 
$
    \|D\overline{F}({\nu_0},U_0)^{-1}\|_{X_2,X_1} \sum_{\alpha \in J_\mathbb{G}} \|\left(\Gamma^\dagger\left((\partial^\alpha L^{-1})^*\right) - (\partial^\alpha\iL)^*\right)\mathbb{v}_\alpha^*\|_2 \leq \mathcal{Z}_{u}.$
\end{remark}




In a similar fashion as Section \ref{ssec : bounds pde case}, one can use Theorem \ref{th: radii polynomial} in order to prove the existence of zeros of $\oF$ in $H_1$. Indeed, taking $X= H_1$ and $Y= H_2$, Theorem \ref{th: radii polynomial} is applicable to \eqref{eq : f(u) = 0 on H1}. In particular, using the construction of $\oA$ in \eqref{def : operator A constraigned}, we can compute the bound $\mathcal{Y}_0$, $\mathcal{Z}_1$ and $\mathcal{Z}_2$ in the case of a PDE with constraints. The computation of $\mathcal{Z}_1$ has already been detailed in Theorem \ref{th : approximation of inverse and norm constrained PDE}. Moreover, the computation of $\mathcal{Z}_2$, similarly as what was achieved in Lemma \ref{lem : Z2 bound}, can be obtained using the mean value inequality for Banach spaces. Finally, Lemma \ref{lem : properties constraigned gamma} allows to compute the bound $\mathcal{Y}_0$ though the isometry generated by $\og$. We illustrate such an application in Section \ref{sec : kdv eigen}. 
 

\begin{remark}
    Similarly as what was achieved in Theorem \ref{th : radii periodic}, the existence of a zero of $\oF$ in $H_1$ using Theorem \ref{th: radii polynomial} implies the existence of a zero of $\overline{F}$ in $X_1$ under a mild assumption of the bound $\mathcal{Z}_2$.
\end{remark}

\subsection{Application to the proof of isolated eigenvalues of the Jacobian}\label{ssec:isolated}

 As described earlier, one of the main applications of the present section is eigenvalue problems. In this section, we define the needed  operators for the proof of eigenvalues of $\mathbb{L} + D\mathbb{G}(\tilde{u})$, where $\tilde{u} \in H^l$. In particular, supposing that $\mathbb{F}(\tilde{u}) =0$ and that one is interested in the dynamics of  $\partial_t v = \mathbb{F}(v)$ around $\tilde{u}$, the eigenvalues of $\mathbb{L} + D\mathbb{G}(\tilde{u})$ can help obtain the spectral stability of the stationary solution $\tilde{u}$. 
 In particular, if  one eigenvalue of $\mathbb{L} + D\mathbb{G}(\tilde{u})$ has a positive real part, then, under some possible additional conditions, one obtain that $\tilde{u}$ is unstable. Let us define $\mathbb{L}(\nu)$ as 
\begin{align*}
    \mathbb{L}(\nu) \bydef \mathbb{L} - \nu I_d.
\end{align*}
Then, let $\nu_0 \in \mathbb{C}$ be an approximate eigenvalue and assume that $\mathbb{L}(\nu_0)$ is invertible so that the set-up presented in the previous section is applicable. Note that imposing that $\mathbb{L}(\nu_0)$ is invertible restricts the range of eigencouples we can hope to prove. Then, denote by $H^l$ the Hilbert space associated to $\mathbb{L}(\nu_0)$. Let $v_0 \in H^l_{\om}$ be an approximation of an eigenvector associated to $\nu_0$ and denote $x_0 \bydef (\nu_0, v_0)$. 

Let us denote $\phi_0 \bydef \mathbb{L}v_0 \in L^2_\om$ and define $\phi : H^l \to \mathbb{C}$ as 
\[
    \phi(u) \bydef -(\mathbb{L}(\nu_0)u,\phi_0)_2
\]
for all $u \in H^l$. Define the following zero finding problem $\overline{\mathbb{F}} : H_1 = \mathbb{C}\times H^l \to H_2 = \mathbb{C} \times L^2$ as 
\begin{equation}\label{eq : eigenvalue problem}
    \overline{\mathbb{F}}(x) \bydef \begin{pmatrix}(\phi_0-\mathbb{L}(\nu_0)v,\phi_0)_2\\
    \mathbb{L}(\nu)v + D\mathbb{G}(\tilde{u})v  \end{pmatrix} = 0
\end{equation}
for all $x = (\nu,v) \in H_1$. If the support of $\tilde{u}$ is not a subset of $\Omega_0$, we cannot apply directly the analysis of Section \ref{sec : approximate inverse constrained} to $D\overline{\mathbb{F}}(x_0)$. However, we  suppose instead that $\tilde{u}$ was obtained by applying Theorem \ref{th: radii polynomial} to \eqref{eq : f(u)=0 on H^l} with some $u_0 \in H^l_\om$. In practice, $\|\tilde{u}-u_0\|_l$ is supposedly small. Then we can define $\overline{\mathbb{F}}_0$ as 
\begin{align*}
    \overline{\mathbb{F}}_0(\nu,v) \bydef \begin{pmatrix}(\phi_0-\mathbb{L}(\nu_0)v,\phi_0)_2\\
    \mathbb{L}(\nu)v + D\mathbb{G}(u_0)v  \end{pmatrix}.
\end{align*}
Now we can apply Theorem \ref{th : approximation of inverse and norm constrained PDE} and build an approximate inverse $\oA$ for $D\overline{\mathbb{F}}_0(x_0)$. In particular, Theorem \ref{th : approximation of inverse and norm constrained PDE} allows to compute an upper bound for $\|I_d - \oA D\overline{\mathbb{F}}_0(x_0)\|_{H_1}$.  Moreover, the following lemma allows  to pass from $\overline{\mathbb{F}}$ to $\overline{\mathbb{F}}_0$ and obtain the needed bounds to apply Theorem \ref{th: radii polynomial} to $\eqref{eq : eigenvalue problem}$.

\begin{lemma}\label{lem : passage_T0_T}
Let $\oA : H_2 \to H_1$, then we have that
\[
    \left\|\oA\overline{\mathbb{F}}(x_0) \right\|_{H_1} \leq  \left\|\oA\overline{\mathbb{F}}_0(x_0) \right\|_{H_1} + \left\|\oA \begin{pmatrix}
    0\\
    D\mathbb{G}(\tilde{u})v_0-D\mathbb{G}(u_0)v_0
\end{pmatrix}\right\|_{H_1}.
    \]
  Moreover,
  \begin{align*}
    \left\|I_d - \oA D\oF(x_0)\right\|_{H_1} \leq \left\|I_d - \oA D\oF_0(x_0)\right\|_{H_1} + \left\|\oA\right\|_{H_2,H_1} \left\|D\mathbb{G}(\tilde{u}) - D\mathbb{G}(u_0)\right\|_{l,2}.
\end{align*}
\end{lemma}

\begin{proof}
First, notice that \[\oF(x_0) - \oF_0(x_0) = \begin{pmatrix}
    0\\
    D\mathbb{G}(\tilde{u})v_0-D\mathbb{G}(u_0)v_0
\end{pmatrix}.\] Then, the first inequality is obtained using the triangle inequality. Now, we have
\begin{align*}
    \left\|I_d - \oA D\oF(x_0)\right\|_{H_1} \leq \left\|I_d - \oA D\oF_0(x_0)\right\|_{H_1} +\left\|\oA \left(D\oF(x_0)-D\oF_0(x_0)\right)\right\|_{H_1}.
\end{align*}
We conclude the proof using that
$
    D\oF(x_0)-D\oF_0(x_0) = \begin{pmatrix}
    0 & 0\\
    0 &   D\mathbb{G}(\tilde{u})- D\mathbb{G}(u_0)
    \end{pmatrix}.
$
\end{proof}
Notice that, using Lemma \ref{lem : passage_T0_T} combined with Theorem \ref{th : approximation of inverse and norm constrained PDE}, one can compute the bounds of Theorem \ref{th: radii polynomial} applied to $\eqref{eq : eigenvalue problem}$. 
We illustrate these computations in Section \ref{sec : kdv eigen} for the proof of eigencouples of the linearization of the Kawahara equation about a proven soliton. In particular, the proof of the eigencouples provides the necessary ingredients for the proof of stability of the soliton (cf. Proposition \ref{prop:stable}).


\section{Application to Kawahara equation}\label{sec:kawahara}
    
 We apply our techniques to a one-dimensional example: the Kawahara equation. This model can be seen as a standard third order KdV equation complexified with an extra fifth order term taking into account the surface tension. This is a model for the study of fluid dynamics which simplifies the Whitham equation. In particular, the coefficients $a(T)$ and $b(T)$ below come from the Taylor expansion of the full dispersion relation in the Whitham equation. Details on the study of such equations can be found in \cite{whitham2011linear}.

In fact, the Kawahara equation is known to have soliton solutions (traveling wave with constant velocity of propagation and vanishing at infinity). In particular, soliton solutions have been proven to exist in \cite{Weinstein_existence_dynamic_solitary} and \cite{LEVANDOSKY_stability_analysis}. The proof is theoretical and is obtained via a constrained minimization problem. The solutions are ground states and are proven to be exponentially decaying to zero. Moreover, some soliton in $\text{sech}^4(x)$ are known explicitly (see \cite{wazwaz2006new} and \cite{solitary2004147}). These solutions are obtained using a tanh-type of method and using auxiliary equations. Finally,  numerical experiments have also been conducted on the study of soliton solutions (see \cite{shock_wave_hoefer} or \cite{yusufouglu2008periodic}) in order to understand their formation and dynamics. Our results provide constructed proofs of solutions (based on a numerical approximation $u_0$) of the Kawahara equation. In this sense, we provide new results for the Kawahara equation and our proofs allow to make the bridge between some known theoretical and numerical results.

We  introduce the Bond number $T \geq 0$ and  look for solitons in
\begin{align}\label{eq : time dependent kdv}
    u_t + \frac{3}{2}uu_y + u_y + a(T)u_{yyy} + b(T)u_{yyyyy} = 0
\end{align}
where $a(T) \bydef  \frac{1}{6}(1-3T)$ and $b(T)\bydef \frac{1}{360}(19-30T -45T^2)$. By using the standard traveling wave change of variable $x = y-ct$ and integrating once (the constants of integration disappear as we look for solutions going to zero at infinity), we study the following ordinary differential equation
  \begin{equation*}
    (1-c)u + \frac{3}{4}u^2 + a(T)u_{xx} + b(T)u_{xxxx} = 0.
\end{equation*}  
Normalizing the linear part, we define $\lambda_1 \bydef -\frac{a(T)}{1-c}$, $\lambda_{2} \bydef \frac{b(T)}{1-c}$ and $\lambda_{3} \bydef \frac{3}{4(1-c)}$. Finally, we obtain the following zero finding problem
\begin{equation}\label{eq : kdv after normalization}
    \mathbb{F}(u)  \bydef u + \lambda_{3} u^2 - \lambda_1 u_{xx} + \lambda_{2} u_{xxxx} = 0
\end{equation}
where $\mathbb{F}(u) = \mathbb{L}u + \mathbb{G}(u)$ and  \[
\mathbb{L}u \bydef u - \lambda_1 u_{xx} + \lambda_{2} u_{xxxx}
\quad \text{and} \quad
\mathbb{G}(u) \bydef \lambda_{3} u^2.
\]
At this point, we verify the required Assumptions \ref{ass:A(1)} and \ref{ass : LinvG in L1}. First, denote by $l : \R \to \R$ the Fourier transform of $\mL$ and $l(\xi) \bydef  1 + \lambda_1 (2\pi\xi)^2 + \lambda_{2} (2\pi\xi)^4$ for all $\xi \in \mathbb{R}$. In particular, the linear part $\mL$ is invertible if and only if $(1-c)-a(T)\xi^2 + b(T)\xi^4 >0$ for all $\xi \in \mathbb{R}$. Studying the previous fourth order polynomial, if
\begin{equation}\label{eq : condition kdv T and c}
    \text{$\frac{1}{3} < T \le \frac{-10+\sqrt{480}}{30}$
and  $ c \le 1 - \frac{a(T)^2}{4b(T)}$},
\end{equation}
then   Assumption \ref{ass:A(1)} is satisfied and we can define the Hilbert space $H^l$ given in \eqref{eq:H^l}.  Consequently, we suppose in the rest of this paper that $c$ and $T$ are chosen such that \eqref{eq : condition kdv T and c} is satisfied. In particular, having $T> \frac{1}{3}$ allows to obtain $\lambda_1,\lambda_2, \lambda_3 >0$ and $l(\xi) \geq 1$ for all $\xi \in \R$, simplifying the analysis. Then, we have $\mathbb{G}(u) = \lambda_3 I_d (u) I_d (u) = \mathbb{G}^1_{2}(u)\G^2_2(u)$ where $\G^1_2 \bydef \lambda_3 I_d$ and $\G^2_2 \bydef I_d$. Using the notations of Assumption \ref{ass : LinvG in L1}, we have $g^1_2 = \frac{\lambda_3}{l}$ and $g^2_2 = \frac{1}{l}$. Consequently, Assumption \ref{ass : LinvG in L1} is satisfied if and only if $\frac{1}{l} \in L^1$, which is trivial to verify. Both Assumptions \ref{ass:A(1)} and \ref{ass : LinvG in L1} are satisfied and the analysis derived in the previous Sections is applicable to \eqref{eq : kdv after normalization}.

Note that the set of solutions in $H^l$ of equation \eqref{eq : kdv after normalization} possesses a natural translation invariance. One way to fix this invariance is to restrict to even solutions. Moreover, we look for real valued solutions as the Kawahara equation is real-valued.   
Therefore we define  $H^l_e$ as the subspace of even functions in $H^l$
\[
    H^l_e \bydef \{u \in H^l : u(x) = u(-x) \in \R, \text{ for all } x \in \mathbb{R}\}
\]
and look for zeros of $\mathbb{F}$ in $H^l_e.$ 
Similarly, $L^2_e$ is  the subspace of even functions in $L^2$. We also define the restriction of our operators on the even restriction.
We denote $\mathbb{L}_e : H^l_e \to L^2_e$ the restriction of $\mathbb{L}$, $\mathbb{F}_e, \mathbb{G}_e : H^l_e \to L^2_e$ the ones of $\mathbb{F}$ and $\mathbb{G}$ respectively. We define in the same manner the even restrictions for the operators and spaces on Fourier series.
Equation \eqref{eq : kdv after normalization} is turned into 
\begin{equation}\label{eq:kdv_eq}
\mathbb{F}_e(u)  = 0 \text{~~ and ~~} u \in H^l_e. 
\end{equation}
Now, denote by $\Omega_0 \bydef (-d,d)$ our domain of study (for some $d>0$) and define $X^l_e$ and $\ell^2_e$ as 
\begin{align*}
    X^l_e \bydef \left\{U = (u_n)_{n \in \mathbb{N}\cup\{0\}}, ~~ \|L_eU\|_2 < \infty\right\} \text{ and } \ell^2_e \bydef \left\{U = (u_n)_{n \in \mathbb{N}\cup\{0\}}, ~~ \|U\|_2 < \infty\right\},
\end{align*}
where $L_e : X^l_e \to \ell^2_e$ is the diagonal infinite matrix with entries $\left(l(\tilde{n})\right)_{n \in \mathbb{N}\cup\{0\}}$ on the diagonal. Similarly, denote by $F_e$ and $G_e$ the operators corresponding to $\mathbb{F}_e$ and $\mathbb{G}_e$, indexed on $\mathbb{N}\cup\{0\}$. Moreover, define $\gamma_e : L^2_e \to \ell^2_e$ and $\gamma_e^\dagger : \ell^2_e \to L^2_e$ as 
\begin{align*}
     \left(\gamma_e(u)\right)_n  &\bydef  \begin{cases}
     \displaystyle\frac{1}{|\om|}\int_\om u(x)dx &\text{ if } n=0\\
       \displaystyle \frac{\sqrt{2}}{|\om|}\int_\om u(x)e^{-2\pi i \tilde{n}\cdot x}dx &\text{ otherwise}
     \end{cases}\\
     \gamma^\dagger_e(U)(x) &\bydef \cha(x) \left(u_0 + \sqrt{2}\sum_{n \in \mathbb{N}}u_n\cos(2\pi \tilde{n} x)\right).
\end{align*}
Then, $\sqrt{|\om|}\gamma_e : L^2_{e,\om} \to \ell^2_e$ and $\sqrt{|\om|}\gamma_e^\dagger :\ell^2_e \to L^2_{e,\om}$ are isometric isomorphism. Similarly, $\Gamma_e : \mathcal{B}(L^2_e) \to \mathcal{B}(\ell^2_e)$ and $\Gamma_e^\dagger : \mathcal{B}(\ell^2_e) \to \mathcal{B}(L^2_e)$ can be defined as in \eqref{def : Gamma and Gamma dagger} with the use of $\gamma_e$ and $\gamma_e^\dagger$. Consequently, the analysis derived in the previous sections is applicable.
\begin{remark}
    There exists values for $c$ and $T$, other than the ones given in \eqref{eq : condition kdv T and c}, for which $l(\xi)>0$ for all $\xi \in \R.$ For these values, the analysis derived in Section \ref{sec:kawahara} is still applicable under some mild modifications.
\end{remark}

\subsection{Computation of the theoretical bounds} \label{sec:Kawahara_bounds}

 As in Section \ref{ssec : bounds pde case}, let $N \in \mathbb{N}$ be the size of the numerical projection and denote by $U_0$ our numerical guess that was projected on $X^{N,4}_0 \bigcap X^l_e$ and
$u_0 \in H^l_e $ its function representation on $\mathbb{R}$. In particular, supp$(u_0) \subset \overline{\Omega_0}$ by construction. Then, define the projections $\pi^N_e$ and $\pi_{N,e}$ as follows
\vspace{-.2cm}
\begin{align}
    \left(\pi^N_e(U)\right)_n  =  \begin{cases}
          U_n,  & n \in I_e^N \\
              0, &n \notin I_e^N
    \end{cases} 
     ~~ \text{and} ~~
     \left(\pi_{N,e}(U)\right)_n  =  \begin{cases}
          0,  & n \in I_e^N \\
              U_n, &n \notin I_e^N
    \end{cases}
    \vspace{-.2cm}
\end{align}
for all $n \in \mathbb{N}\cup\{0\}$, where $I^N_e \bydef \{n \in \mathbb{N}\cup\{0\},~ n \leq N\}.$ Using the construction of Section \ref{ssec : construction of the approximate inverse}, we start by numerically constructing an approximate inverse $B^N_e = \pi^N_e B^N_e \pi^N_e$ for $\pi^N_e DF_e(U_0)L^{-1}_e \pi^N_e$. Then, we construct $\mathbb{A} : H^l_e \to L^2_e$ as 
\begin{align*}
    \mathbb{A} \bydef \mathbb{L}^{-1}_e \mathbb{B}, ~~ \text{ where } ~\mathbb{B} \bydef \out + \Gamma^\dagger_e \left(\pi_{N,e} + B^N_e\right).
\end{align*}
We present in the following subsections the explicit computations of the bounds $\mathcal{Z}_1$ and $\mathcal{Z}_2$ using the analysis derived in Sections \ref{sec:bound_inverse} and \ref{sec : computer assisted analysis}. Note that the computation of the bound $\mathcal{Y}_0$ is already explicited in Lemma \ref{lem : bound Y0}.

\subsubsection{Computation of \boldmath$\mathcal{Z}_2$\unboldmath}

We begin with the computation of the bound $\mathcal{Z}_2$. First, 
notice that the proof of Lemma \ref{Banach algebra} implies that for $\kappa \geq \max_{\xi \in \R} \frac{\lambda_3}{l(\xi)} \|\frac{1}{l}\|_2$ then
\begin{equation}\label{eq : banach algebra kdv}
    \|\lambda_3 w v\|_2 \leq  \kappa \|w\|_l\|v\|_l
\end{equation} for all $w, v \in H^l$. In order to obtain an explicit value for $\kappa$, we use the following Proposition \ref{prop: estimates_on_l}. 
\begin{prop}\label{prop: estimates_on_l}
Under \eqref{eq : condition kdv T and c}, we have
\[
\max_{\xi \in \mathbb{R}}{\frac{1}{l(\xi)}} = 1 ~~ \text{ and } ~~ \left\|\frac{1}{l}\right\|_2 \leq \sqrt{\frac{3}{8\sqrt{2}\lambda_2^{\frac{1}{4}}}}.
\]
\end{prop}

\begin{proof}
The computation for $\displaystyle\max_{\xi \in \mathbb{R}}{\frac{1}{l(\xi)}}$ is a simple study of the extrema of  the polynomial $l$. For the value of $\|\frac{1}{l}\|_2$ we notice that
\begin{align*}
  \int_\mathbb{R} \frac{1}{l(\xi)^2} d\xi = \int_\mathbb{R} \frac{(1+\lambda_2(2\pi\xi)^4)^2}{(1+\lambda_2(2\pi\xi)^4)^2l(\xi)^2} d\xi &\leq 
  \int_\mathbb{R} \frac{1}{(1+\lambda_2(2\pi\xi)^4)^2} d\xi = \frac{3}{8\sqrt{2}\lambda_{2}^{\frac{1}{4}}}.
\end{align*}
\end{proof}

Using Proposition~\ref{prop: estimates_on_l}, we fix some $\kappa >0$ such that 
\begin{equation}\label{eq : kappa in kdv}
    \kappa \geq \lambda_3\sqrt{\frac{3}{8\sqrt{2}\lambda_2^{\frac{1}{4}}}}
\end{equation}
and we have $\|\lambda_3 w v\|_2 \leq \kappa \|w\|_l\|v\|_l$ for all $w,v \in H^l$.  Moreover, we can compute an upper bound for $\mathcal{Z}_2$ defined in Theorem~\ref{th: radii polynomial}.

\begin{lemma}\label{lem : Z_2 bound Kdv}
Let $r>0$ and let $\mathcal{Z}_2(r)$ be defined as 
\begin{align*}
    \mathcal{Z}_2(r) \geq 2 \kappa \max\left\{1,\left\|B^N_e\right\|_2\right\},
\end{align*}
where $\kappa$ satisfies \eqref{eq : kappa in kdv}.
Then, $\|\mathbb{A}\left(D\mathbb{F}_e(u_0)-D\mathbb{F}_e(v)\right)\|_l \leq \mathcal{Z}_2(r) r$ for all $v \in \overline{B_r(u_0)} \subset H^l_e$.
\end{lemma}
\begin{proof}
Using Lemma \ref{lem : Z2 bound}, we can choose 
\begin{align*}
    \mathcal{Z}_2(r) \geq \max\left\{1, \|B^N_e\|_2\right\}\sup_{h \in B_r(u_0)}\|D^2\mathbb{F}_e(h)\|_{\mathcal{B}(H^l,L^2)}.
\end{align*}
However, we have $D^2\mathbb{F}_e(h)w = 2\lambda_3 \mathbb{w}$ for all $w \in H^l_e$ and all $h \in B_r(u_0).$ Moreover, using \eqref{eq : banach algebra kdv}, we have
$
    \|\lambda_3 w v\|_2 \leq \kappa \|w\|_l \|v\|_l 
$
for all $v \in H^l_e$. Consequently, we get 
\[
    \sup_{h \in B_r(u_0)}\|D^2\mathbb{F}_e(h)\|_{\mathcal{B}(H^l,L^2)} \leq 2 \kappa. 
\]
\end{proof}

 \subsubsection{Computation of \boldmath$\mathcal{Z}_1$\unboldmath}

Using Lemma \ref{lem : bound Z1}, one can explicitly compute a value of $Z_1$. To obtain a value for $\mathcal{Z}_1$, it remains to compute an upper bound $\mathcal{Z}_u$. This is achieved using Theorem \ref{th : Zu computation}.
First, using the definition of $\mathbb{G}$ in the Kawahara equation, we have 
\begin{align*}
    D\mathbb{G}(u_0) = 2\lambda_3 \mathbb{u}_0 \bydef \mathbb{v}_0 \text{ and } \mathbb{V}_0 \bydef D{G}(U_0),
\end{align*}
where $V_0 = \gamma_e(v_0) \in X^l_e$. Now, Theorem \ref{th : Zu computation} requires the computations of $a>0$ and of the functions $(f_\alpha)$ defined in Lemma~\ref{lem : computation of falpha}. However, as $\mathbb{G}$ does not possess derivatives, we only need to compute 
 $f_0(x) \bydef \mathcal{F}^{-1}(\frac{1}{l(\xi)})$. Now, straightforward computations show that under \eqref{eq : condition kdv T and c}, the roots of  $l$ are given by $$\{b \pm ia, -b \pm ia\}$$ for some $a>0$ and some $b \geq 0$. Note that $a\neq 0$ since $l>0$ by assumption.
The next lemma then provides the computation of $f_0$ using some basic Fourier analysis.
\begin{lemma}\label{lem:kdv_f}
Let $f_0(x) \bydef  \mathcal{F}^{-1}(\frac{1}{l(\xi)}) (x)$ and suppose that $l$ has 4 roots $\{b \pm ia, -b \pm ia\}$ with $a>0$ and $b \geq 0$. Then we have
\begin{align*}
   &f_0(x) = \frac{e^{-a|x|}}{4\lambda_{2} ab(a^2+b^2) }(sign(x) a\sin(bx) + b \cos(bx)).
\end{align*}
In particular, 
\begin{align}
    |f_0(x)| \leq C_0e^{-a|x|}
\end{align}
for all $x \in \mathbb{R}$ where $C_0 \bydef \frac{a+b}{4|\lambda_2 |ab(a^2+b^2)}.$
\end{lemma}
\begin{proof}
The proof can be found in \cite{poularikas2018transforms} (Section 2.3.14) where Fourier transforms of rational functions are presented.
\end{proof}
Now that the constants $a$ and $C_0$ are given in the previous Lemma \ref{lem:kdv_f}, the bound $\mathcal{Z}_u$ can be computed using Theorem \ref{th : Zu computation}.

\begin{lemma}
Let $C_0$ and $a$ be defined in Lemma \ref{lem:kdv_f}. Moreover, let $E_{2,e} \bydef \gamma_e\left(f_{E_2}\right)$ where $f_{E_2}(x) \bydef \cha(x) \cosh(2ax)$ for all $x \in \om$. Moreover, define
\begin{align*}
    \mathcal{Z}_{u,1} \bydef   \|\mathbb{1}_{\R^m \setminus \om } ((\iL_e)^*-\Gamma^\dagger\left(L^{-*}_e\right))\mathbb{v}_0^*\|_2 ~~ \text{ and } ~~
     \mathcal{Z}_{u,2} \bydef  \|\mathbb{1}_{\om } ((\iL_e)^*-\Gamma^\dagger\left(L^{-*}_e\right))\mathbb{v}_0^*\|_2.
\end{align*}
Then $E_{2,e}$ can be computed explicitly using \eqref{eq : coefficients of E1 and E2} and
\begin{align*}
    (\mathcal{Z}_{u,1})^2 \leq |\om|\frac{C_0^2e^{-2ad}}{a}(V_0,V_0*E_{2,e})_2 ~~ \text{ and } ~~
    (\mathcal{Z}_{u,2})^2 \leq (\mathcal{Z}_{u,1})^2 +   e^{-4ad}C(d)C_0^2|\om| (V_0,V_0*E_{2,e})
\end{align*}
where 
\[C(d) \bydef 4d + \frac{4e^{-ad}}{a(1-e^{-\frac{3ad}{2}})} + \frac{2}{a(1-e^{-2ad})}.\]
If $\mathcal{Z}_u \geq \max\{1, \|B^N_e\|_2\}\left((\mathcal{Z}_{u,1})^2 + (\mathcal{Z}_{u,2})^2\right)^{\frac{1}{2}}$, then $\mathcal{Z}_u$ satisfies $\|(\mathbb{L}_e^{-1}-\Gamma^\dagger_e\left(L^{-*}_e\right))\mathbb{v}_0\mathbb{B}^*\|_2\leq\mathcal{Z}_{u}$.
\end{lemma}

\begin{proof}
First, notice that since $l$ is real valued, then $\mathbb{L}_e^* = \mathbb{L}_e$ and  $L_e^*  = L_e$. Similarly, we have $v_0 = v_0^*$ as $u_0$ is real-valued. Therefore, $\|((\mathbb{L}_e^{-1})^{*}-\Gamma^\dagger\left(L^{-*}_e\right))\mathbb{v}_0^*\|_2 = \|(\mathbb{L}_e^{-1}-\Gamma^\dagger\left(L^{-*}_e\right))\mathbb{v}_0\|_2$.
    Let $u \in L^2_e$ such that $\|u\|_2 = 1$. Then, define $v \bydef v_0 u$. By construction we have $v \in L^2_e$ and
     supp$(v) \subset \overline{\Omega_0}.$ Using Theorem \ref{th : Zu computation} and more specifically equation \eqref{eq : quantity needed for Z1}, we need to compute

\[
  \int_\om \int_\om |v(x)v(z)| \left( \sum_{n \in \mathbb{Z}, n \neq 0} I_n(x,z) \right) dx dz
   = 2\int_\om \int_\om |v(x)v(z)| \left( \sum_{n=1}^\infty I_n(x,z) \right) dx dz,
\]
where 
\[
I_n(x,z) \bydef \int_{\mathbb{R}\setminus (\om \cup (\om +2dn))} e^{-a|y-x|}e^{-a|y-2dn-z|}dy,
\]
and where we used that $v \in L^2_e$.
Notice that 
\begin{align*}
        I_n(x,z) &= \int_{\mathbb{R}} e^{-a|y-x|}e^{-a|2dn+z-y|}dy - \int_{\om \bigcup (\om+2dn) } e^{-a|y-x|}e^{-a|2dn+z-y|}dy  \\
        & = 2d(n-1)e^{-a(2dn+z-x)} + \frac{e^{-2ad(n+1)}}{2a}\left(e^{-a(x+z)}+e^{a(x+z)}\right)
\end{align*}
where the last equality follows from elementary computations. Now, notice that since $v \in L^2_e$, we have
\begin{align}\label{Z1 kdv step 1}
\nonumber
    \int_\om \int_\om |v(x) v(z)|  \left(2d(n-1)e^{-a(2dn+z-x)} + \frac{e^{-2ad(n+1)}}{2a}\left(e^{-a(x+z)}+e^{a(x+z)}\right)\right)dxdz\\
    = \left(2d(n-1)e^{-2adn}+ \frac{e^{-2ad(n+1)}}{a}\right)\left(\int_\om |v(x)|e^{ax}dx \right)^2.
\end{align}
Then, using Cauchy-Schwartz inequality, we get
\begin{align*}
    \left(\int_\om |v(x)|e^{ax}dx \right)^2  =  \left(\int_\om |v_0(x)u(x)|e^{ax}dx \right)^2  
    \leq \|u\|^2_2 \int_\om v_0(x)^2e^{2ax}dx 
    \leq  \int_\om v_0(x)^2e^{2ax}dx.
\end{align*}
 Moreover, using that $v_0 \in L^2_e$, we have
\begin{align}\label{Z1 parseval kdv}
    \int_\om v_0(x)^2e^{2ax}dx = \int_\om v_0(x)^2e^{-2ax}dx = \int_\om v_0(x)^2\cosh(2ax)dx = |\om| (V_0, V_0*E_{2,e})
\end{align}
where we used Parseval's identity. Then, notice that 
\begin{align}\label{sum of the Ik}
    \sum_{n=1}^\infty \left(2d(n-1)e^{-2adn} + \frac{e^{-2ad(n+1)}}{a}\right) \leq C(d)e^{-4ad}.
\end{align}
Combining \eqref{Z1 kdv step 1}, \eqref{Z1 parseval kdv} and \eqref{sum of the Ik}, we conclude the proof using Theorem~\ref{th : Zu computation} and Lemma~\ref{lem : bound Z1}.
\end{proof}

\subsection{Numerical results}\label{sec ; numerical results}

 Let $T =0.35$ and $c = 0.9$ (in particular $\frac{1}{3}<T<\frac{-10+\sqrt{480}}{30}$ and $c < 1- \frac{a(T)^2}{4b(T)}$) and define $w_1 \bydef \frac{-\lambda_1}{2\lambda_{2}}$ and $w_2 \bydef \frac{\sqrt{|\lambda_1^2-4\lambda_{2}|}}{4\lambda_{2}}$. Then, denote
\[
a \bydef \sqrt{\frac{\sqrt{w_1^2+4w^2_2}-w_1}{2} } ~~ \text{ and } ~~ b \bydef \frac{w_2}{a}.
\] 
It is well know that the classical KdV equation $
u - \lambda_1 u'' + \lambda_3 u^2 = 0
$
has even soliton solutions in $\text{sech}^2(x)$. Using these solutions, we can initialize Newton's method and then converge to some approximate solutions of the Kawahara equation. That enables us to build a vector of Fourier series $\tilde{U}_0$. Then, we need to project $\tilde{U}_0$ into the trace-free functions of order 4. In particular, the objects introduced in Section \ref{ssec : finite dimensional trace theorem} possess a restriction to even function. Indeed, first notice that the first and third derivatives of a periodic and even function vanish at $x=d$. Therefore, we only need to ensure that the function and its second derivative vanish at $x=d.$ Therefore, using the even symmetry, we can define a finite-dimensional trace operator $ \mathcal{T}^N_{4,e} : X^l_e \to \R^2$ as
\begin{align*}
    \mathcal{T}^N_{4,e}(U) = \begin{pmatrix}
        (\mathcal{T}_{4,e}^N(U))_{0}\\
        (\mathcal{T}_{4,e}^N(U))_{2}
    \end{pmatrix},
\end{align*}
where
\[
(\mathcal{T}_{4,e}^N(U))_{0} \bydef u_0 +  \sqrt{2} \displaystyle\sum_{p \in I^N_e}(-1)^pu_{p} ~~\text{ and }~~  (\mathcal{T}_{4,e}^N(U))_{2} \bydef \sqrt{2} \displaystyle\sum_{p \in I^N_e}(-1)^p\left(\frac{\pi p}{d}\right)^2u_{p} 
\]
for all $U \in X^l$. In particular, if $U = \pi^N_e U$ and $U \in \text{Ker}\left(
    \mathcal{T}^N_{4,e}\right)$, then $U$ has a function representation on $\om$ which is even and  in $H^4_0(\om)$. Using Theorem \ref{th:trace_theorem} and \eqref{eq : projection in X^k_0}, we  project $\tilde{U}_0$ in the kernel of $\mathcal{T}_{4,e}^N$ and denote $U_0$ the projection. In particular, using the notations of the theorem, we use $M = \mathbb{M}^*$ and $D = \pi^N_e L^{-1}_e\pi^N_e$.
The computations in \eqref{eq : projection in X^k_0} are achieved using IntervalArithmetic on Julia \cite{julia_interval} and the code is available at \cite{julia_cadiot}. In particular, $U_0$ has a representation $u_0 \bydef \gamma^\dagger_e(U_0) \in H^l_e$ such that supp$(u_0) \subset \overline{\Omega_0}.$

Using the approach of Section \ref{ssec : bounds pde case}, we were able to obtain the following bounds
\begin{align*}
    \|D\mathbb{F}_e(u_0)^{-1}\|_{2,l} &\leq 4.4\\
    \mathcal{Y}_0 &\leq 2.26 \times 10^{-14}.
\end{align*}
Moreover, using Theorem \ref{th: radii polynomial}, there exists a unique solution in some ball of $H^l_e$ around $u_0$.  The representation of $u_0$ is given in Figure~\ref{fig:solution} below.

\begin{figure}[h!]
\centerline{
\includegraphics[width=0.7\textwidth]{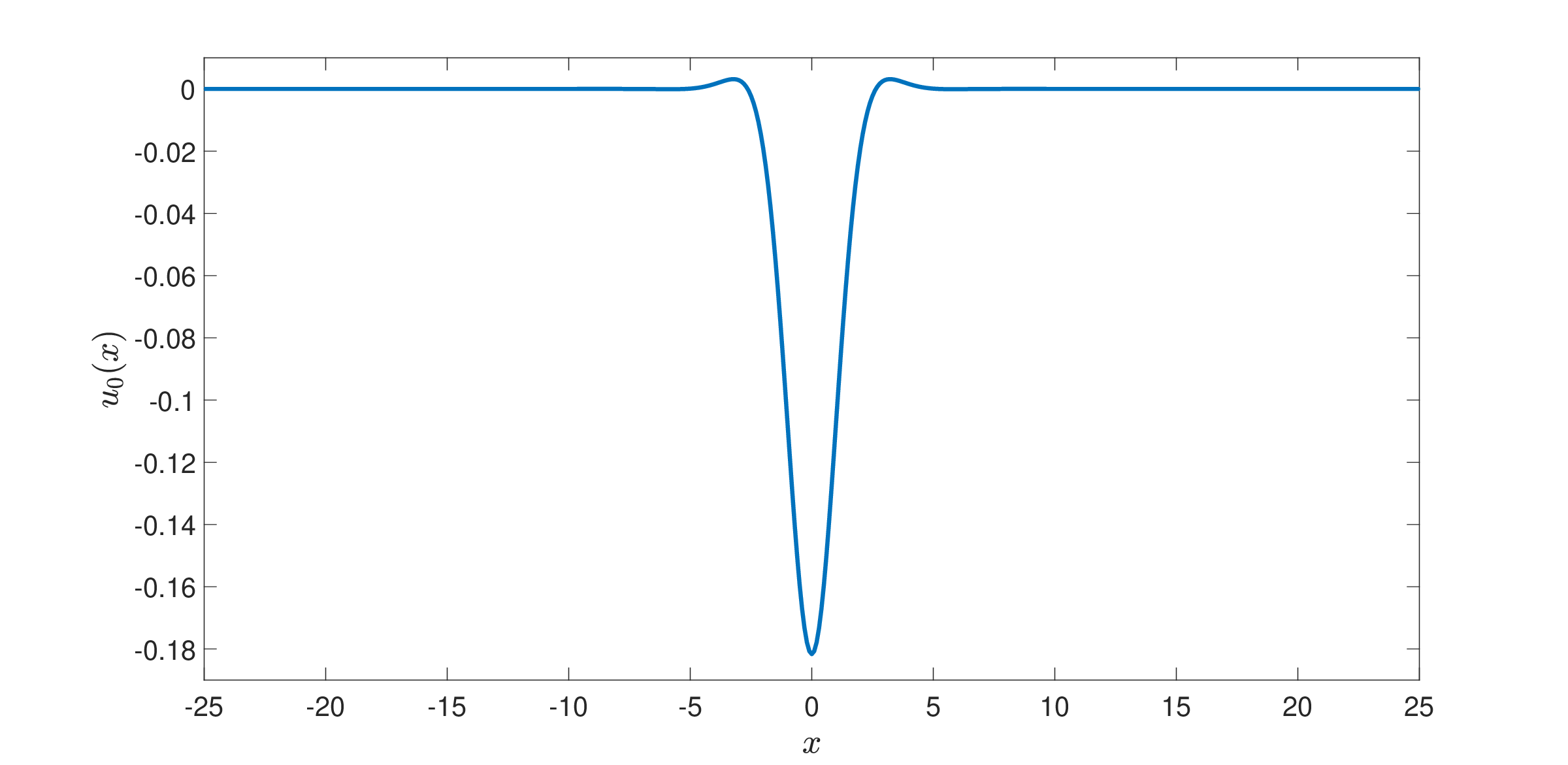}
}
\vspace{-.4cm}
\caption{Numerical approximation $u_0$ on $\Omega_0$.}
\label{fig:solution}
\end{figure}
\begin{theorem}\label{th:proof_kdv}(Proof of a soliton in the Kawahara equation)
Let $r_0 \bydef 2.27\times 10^{-14}$, then there exists a unique solution $\tilde{u}$ to \eqref{eq:kdv_eq} in $\overline{B_{r_0}(u_0)} \subset H^l_e$ and we have that $\|\tilde{u}-u_0\|_l \leq r_0$. Moreover, the solution is unique in $\overline{B_{0.015}(u_0)}$. Finally, there exists a unique $\tilde{U} \in \overline{B_{\frac{r_0}{\sqrt{|\Omega_0|}}}(U_0)} \subset X^l_e$ such that $\tilde{U}$ is a solution to the periodic equivalent of \eqref{eq:kdv_eq} on $\Omega_0$. 
\end{theorem}

\begin{proof}
The rigorous computations for the needed bounds of Theorem \ref{th: radii polynomial} are obtained via \cite{julia_cadiot}. 
In particular, the proof of $\tilde{u}$ is a direct application of the aforementioned theorem. For the proof of $\tilde{U} \in X^l_e$, we use  Theorem \ref{th : radii periodic}. Let $r>0$ and let $V \in \overline{B_r(U_0)}$, then for all $W \in X^l_e$ we have 
\begin{align*}
    \|DF_e(U_0)W - DF_e(V)W\|_2 = \|2\lambda_{3}(U_0-V)*W\|_2 \leq \max_{\xi \in \mathbb{R}}\frac{1}{l(\xi)}\|2\lambda_{3}(U_0-V)*(LW)\|_2.
\end{align*}
But similarly as the reasoning of Lemma \ref{Banach algebra}, we obtain that
\[
 \|DF_e(U_0)W - DF_e(V)W\|_2  \leq 2\lambda_3 \left( \sum_{n \in \mathbb{Z}}\frac{1}{l(\tilde{n})^2}\right)^{\frac{1}{2}}\|U_0-V\|_l\|W\|_l.
\]
Numerically we made sure to build $\kappa$  such that
\begin{equation}\label{condition for kappa kdv periodic}
    \kappa \geq \max\left\{\left\|\frac{1}{l}\right\|_2, \frac{1}{\sqrt{|\om|}} \left\|\left(\frac{1}{l(\tilde{n})}\right)_{n \in \mathbb{Z}}\right\|_2\right\} 
\end{equation}
and therefore
\begin{align*}
    \|DF_e(U_0)W - DF_e(V)W\|_2 
     \leq 2 \kappa\sqrt{|\Omega_0|}~ \|U_0-V\|_l\|W\|_l.
\end{align*}
Choosing $\mathcal{Z}_2(r) \geq 2 \kappa \max\{1,\|\pi^N + B_\om\|_2\}$, we conclude the proof combining Lemma \ref{lem : Z_2 bound Kdv} and Theorem \ref{th : radii periodic}.
\end{proof}

\begin{remark}\label{rem : riemann sum idea}
    Notice that the condition \eqref{condition for Z2 periodic} in Theorem \ref{th : radii periodic} is transformed into \eqref{condition for kappa kdv periodic} in the previous lemma. However, using a Riemann sum argument, notice that 
    \[
    \frac{1}{\sqrt{|\om|}} \left( \sum_{n \in \mathbb{Z}}\frac{1}{l(\tilde{n})^2}\right)^{\frac{1}{2}} \to \left\|\frac{1}{l}\right\|_2
    \]
    as $d \to \infty$. Consequently, as $d$ increases, choosing $\kappa$ satisfying \eqref{condition for kappa kdv periodic} is less and less constraining. 
\end{remark}

\subsection{First three eigencouples of \boldmath$\mathbb{L}+ D\mathbb{G}(\tilde{u})$\unboldmath}\label{sec : kdv eigen}

 In this section, we return to the original setting where $\mathbb{L}$ and $D\mathbb{G}(\tilde{u})$ are defined on the full $H^l$ without any restriction to even functions. Here, $\tilde{u}$ denotes the soliton obtained in Theorem~\ref{th:proof_kdv}. We first notice that $\mathbb{L} + D\mathbb{G}(\tilde{u})$ only has real eigenvalues by self-adjointness on $L^2$. Therefore we replace $\mathbb{C}$ by $\mathbb{R}$ when using the results of Sections \ref{ssec:isolated} and \ref{sec : approximate inverse constrained}.
First, we compute (see \cite{julia_cadiot}) numerical approximations of the first three  eigencouples of $L + DG(U_0)$ which we denote by $(\nu_1,U_1),(\nu_2, U_2), (\nu_3,U_3)$ with $\nu_1 \leq \nu_2 \leq \nu_3$. Then, for all $1 \leq i \leq 3$, we project $U_i$ into $X^{N_0,k}_0$ and still call $U_i$ the projection. In particular, $U_i$ has a representation $u_i \in H^l_\om$.
We want to use the set up of Section \ref{ssec:isolated} to prove that the above pairs are isolated eigencouples of the linearization. In particular, for each eigencouple we want to build an approximate inverse $\oA$ of $D\overline{\mathbb{F}}(\nu_i,u_i)$ using the construction derived in Section \ref{sec : approximate inverse constrained}.

\subsubsection{Computation of the needed bounds}

As for the proof of the soliton, we need to compute the needed constants introduced in Section \ref{sec:constraigned}. Let us fix some $k \in \{1, 2, 3\}$ and  define 
\[
\mathbb{L}_\nu \bydef \frac{1}{1-\nu_k}(\mathbb{L}-\nu I_d)
\]
for all $\nu \in \mathbb{R}$ and the norm $H^l$ is redefined as explained in Section \ref{sec:constraigned}. This definition yields
$l_{\nu_k} \bydef \frac{l-\nu_k}{1-\nu_k} \geq 1.$
In particular, we define $\lambda_{1,k} \bydef \frac{\lambda_1}{1-\nu_k}$, $\lambda_{2,k}\bydef \frac{\lambda_{2}}{1-\nu_k}$ and $\lambda_{3,k} \bydef \frac{\lambda_{3}}{1-\nu_k}$, so that
\begin{align*}
    \mathbb{L}_{\nu_k} = I_d - \lambda_{1,k} \partial^2_x + \lambda_{2,k} \partial^4_x ~~\text{ and }~~
    \frac{1}{1-\nu_k}D\mathbb{G}(u_0) = 2\lambda_{3,k}u_0
\end{align*}

Recall that $H_1 = \mathbb{R} \times H^l$ and $H_2 = \mathbb{R}\times L^2$ (cf. Section \ref{sec:constraigned}). Then let us introduce the following maps  $\overline{\mathbb{F}}: H_1 \to H_2$ and $\overline{\mathbb{F}}_0: H_1 \to H_2$ 
\begin{align*}
    \overline{\mathbb{F}}(\nu,v) \bydef  \begin{pmatrix}(u_k-v,u_k)_l\\
    \mathbb{L}_\nu v + \frac{1}{1-\nu_k}D\mathbb{G}(\tilde{u})v  \end{pmatrix} ~~ \text{ and } ~~ \overline{\mathbb{F}}_0(\nu,v) \bydef  \begin{pmatrix}(u_k-v,u_k)_l\\
    \mathbb{L}_\nu v + \frac{1}{1-\nu_k}D\mathbb{G}(u_0)v  \end{pmatrix}.
\end{align*}
We search for zeros of $\oF$ and 
we need to define $\overline{\mathbb{F}}_0$ intermediately as explained in Section \ref{ssec:isolated}. Moreover, as $u_0$ and $\tilde{u}$ are close by Theorem \ref{th:proof_kdv}, we  use Lemma \ref{lem : passage_T0_T} to link $\overline{\mathbb{F}}$ and $\overline{\mathbb{F}}_0$.
Moreover,  define the equivalent periodic operator $F : X_1  \to X_2$
\begin{align*}
    F(\nu,U) \bydef  \begin{pmatrix} |\om|(U_k-U,U_k)_l\\
    \frac{1}{1-\nu_k}(L-\nu I_d) U + \frac{1}{1-\nu_k}DG(U_0)U  \end{pmatrix}
\end{align*}
where $U_k$ is the sequence of Fourier coefficients of $\sqrt{|\Omega_0|}u_k$ and $X_1, X_2$ are defined in Section \ref{sec:constraigned}. Finally, in order to apply Theorem \ref{th: radii polynomial},  the passage between $\overline{\mathbb{F}}_0$ and $\overline{\mathbb{F}}$ has to be controlled. To do so, Theorem \ref{th:proof_kdv} provides $\|\mathbb{L}(\tilde{u}-u_0)\|_2 \leq r_0 \bydef 2.27 \times 10^{-14}$ and one can use Lemma \ref{lem : passage_T0_T}. The following Theorem provides the results of the computer-assisted proofs obtained in \cite{julia_cadiot}.
 
\begin{theorem}\label{th:eigen_kdv}
Let $(\nu_1,u_1), (\nu_2,u_2) ,(\nu_3,u_3)$ with $\nu_1 \leq \nu_2 \leq \nu_3$, then for all $1 \leq i \leq 3$, there exists $R_i$ (given in the table below) such that there exists a unique eigencouple $x_i^*$ of $\mathbb{L} + D\mathbb{G}(\tilde{u})$ such that $x_i^* \in \overline{B_{R_i}((\nu_i,u_i))}  \subset H_1 = \mathbb{R}\times H^l$. Moreover, $\nu_1, \nu_2$ and $\nu_3$ are simple eigenvalues.
\end{theorem}

\begin{table}[h!]
\centering
\begin{tabular}{|c|c|c|}
 \hline
 Eigenvalue& Value& Radius of the ball of contraction ($R_i$)\\
 \hline
 \hline
 $\nu_1$  & -1.145218189524111  & $2.8\times 10^{-11}$\\
 $\nu_2$&   0  & $5.7\times 10^{-11}$  \\
  $\nu_3$ & 0.8408132697971273 & $4.4\times 10^{-10}$\\
 \hline
\end{tabular}
 \caption{First three eigenvalues of $\mathbb{L} + D\mathbb{G}(\tilde{u})$.}\label{table : eigenvalues}
\end{table}

\begin{remark}
It is easy to check that $(\tilde{u})'$ is in the kernel of $\mathbb{L} + D\mathbb{G}(\tilde{u})$. Indeed, one has that
\begin{align*}
    \frac{d}{dx}(\mathbb{L}\tilde{u} + \lambda_{3} (\tilde{u})^2) = \mathbb{L}(\tilde{u})' + 2\lambda_{3} \tilde{u} (\tilde{u})' = (\mathbb{L} + D\mathbb{G}(\tilde{u})) (\tilde{u})' =0
\end{align*}
as $\mathbb{L}\tilde{u} + \lambda_{3} (\tilde{u})^2 =0$. The previous theorem proves that the eigenvalue $0$ is simple, which is not a trivial task in general. Moreover, the fact that $\nu_1, \nu_2$ and $\nu_3$ are simple is a direct conclusion of the contraction mapping theorem.
\end{remark}

The following result helps in proving that $\nu_1, \nu_2 ,\nu_3$ are precisely the first  eigenvalues of $\mathbb{L} + D\mathbb{G}(\tilde{u})$.

\begin{lemma}\label{kdv_neumann}
Let $\nu^* < 1$ and suppose that $\|(\mathbb{L} -\nu^*I_d + D\mathbb{G}(\tilde{u}))^{-1}\|_{2} \leq \mathcal{C}$. Then $\mathbb{L} - \nu I_d + D\mathbb{G}(\tilde{u})$  is invertible for all $\nu \in (\nu^*-\frac{1}{\mathcal{C}},\nu^*+\frac{1}{\mathcal{C}}) \bigcap(-\infty, 1)$ and \begin{align*}
    \|(\mathbb{L} -\nu I_d + D\mathbb{G}(\tilde{u}))^{-1}\|_{2} \leq \frac{\mathcal{C}}{1-\mathcal{C}|\nu -\nu^*|}.
\end{align*}
\end{lemma}
\begin{proof}
The proof is a direct application of Neumann series combined with the fact that $\mathbb{L} -\nu I_d + D\mathbb{G}(\tilde{u})$ is a Fredholm operator for all $\nu < 1$ as $l \geq 1$ ($\lambda_1, \lambda_2 >0$). 
\end{proof}

We numerically use Lemma \ref{kdv_neumann} to prove Theorem \ref{th:no_eig_outside} below. Indeed, using IntervalArithmetics.jl  \cite{julia_interval} and the construction of Theorem \ref{th : approximation of inverse and norm}, we ensure that $\mathbb{L} -\nu I_d + D\mathbb{G}(\tilde{u})$ is non invertible only at isolated values that we found in the previous Table \ref{table : eigenvalues}. The algorithmic details are exposed in \cite{julia_cadiot}.
\begin{theorem}\label{th:no_eig_outside}
The operator $\mathbb{L} + D\mathbb{G}(\tilde{u})$ does not possess eigenvalues smaller than $\nu_1$ nor in $(\nu_i,\nu_{i+1})$ for all $i \in \{1, 2\}$. In particular, all the other eigenvalues are strictly bigger than $\nu_3$.
\end{theorem}
\begin{proof}
Using the previous Lemma \ref{kdv_neumann} we ensure, using interval arithmetics in Julia (cf. \cite{julia_interval}) and the construction of Theorem \ref{th : approximation of inverse and norm}, that there is no eigenvalue in $(\nu_i,\nu_{i+1})$ for all $i \in \{1, 2 \}$. Moreover, as $D\mathbb{G}(\tilde{u})$ is relatively compact with $\mathbb{L}$ and as $\mathbb{L}$ is positive, we know that $\mathbb{L} + D\mathbb{G}(\tilde{u})$ only has finitely many negative eigenvalues. Therefore, if we verify that no eigenvalues are smaller than $\nu_1$, then all the other eigenvalues will  be strictly larger than $\nu_3$. We give a description on the algorithmic proof that provides that $\nu_1$ is the smallest eigenvalue. 

First, we prove that $\nu_1$ is an isolated eigenvalue using a contraction argument. Therefore, we obtain that there exists $r>0$ such that $\nu_1$ is the only eigenvalue in $(-r+\nu_1,\nu_1+r)$. Then we consider $\nu^* = \nu_1-r $, and using the construction of Theorem \ref{th : approximation of inverse and norm} (we build an operator $\mathbb{A}$ approximating the inverse of $\mathbb{L} - \nu^* I_d + D\mathbb{G}(\tilde{u})$), we numerically prove that $\mathbb{L} - \nu^* I_d +D\mathbb{G}(\tilde{u})$ is invertible and that the norm of its inverse is bounded by $\mathcal{C}$. Finally Lemma \ref{kdv_neumann} gives that $\mathbb{L} - \nu I_d + D\mathbb{G}(\tilde{u})$ is invertible for $\nu \in (\nu^*-\frac{1}{\mathcal{C}}, \nu^*+\frac{1}{\mathcal{C}})$ (as $\nu_3 <1$). 

Now we can repeat the same reasoning by choosing $\tilde{\nu} = \nu^*-\frac{1}{\mathcal{C}}$. We iterate the previous steps until we obtain $\|\mathbb{L} - \nu^* I_d\|_{l,2} > \|D\mathbb{G}(\tilde{u})\|_{l,2}$ for some $\nu^*<0$. Then, using that $\|\mathbb{L} - \nu I_d\|_{l,2} \geq \|\mathbb{L} - \nu^* I_d\|_{l,2}$ for all $\nu < \nu^*$ (by positivity of $\mathbb{L}$), it yields that $\mathbb{L} - \nu I_d + D\mathbb{G}(\tilde{u})$ is invertible for all $\nu \leq \nu^*$. In the end it allows to prove that there is no eigenvalue smaller that $\nu_1$.\\
Proving that there is no eigenvalue in $(\nu_i,\nu_{i+1})$ ($i \in \{1,2\}$) works out identically.
\end{proof}

\subsection{Proof of stability}\label{sec : stability kawahara}

Determining the stability of traveling waves is in general a non-trivial problem. Indeed, in some cases, a spectral theory is not enough to provide a conclusion. The particular case of the Kawahara equation falls into the class of model equations for one-dimensional long nonlinear waves (cf. \cite{stability_Albert}).   Albert in \cite{stability_Albert} studied the stability of solitary solutions to this specific kind of equations. Other works, such as \cite{Weinstein_existence_dynamic_solitary} and \cite{LEVANDOSKY_stability_analysis}, also derived some criteria for obtaining the stability of the soliton.
In this section we focus on the results of Albert \cite{stability_Albert}.
The author derived a theorem (see Theorem 3.1 in \cite{stability_Albert}) that gives necessary conditions to obtain the (orbital) stability of the soliton. Our goal is to prove that these conditions are satisfied in the case of the soliton proved in Theorem \ref{th:proof_kdv}. We first derive a necessary intermediate result.

\begin{prop}\label{prop : norm inverse kdv}
Let $\mathcal{C} >0$ be such that $\|D\mathbb{F}_e(u_0)^{-1}\|_{2,l} \leq \mathcal{C}$ and let $r_0 >0$ and $\tilde{u} \in H^l$ defined in Theorem $\ref{th:proof_kdv}$. If \[
2\kappa\mathcal{C}r_0 <1,
\]
where $\kappa$ satisfies \eqref{eq : kappa in kdv}, then $D\mathbb{F}_e(\tilde{u})$ is invertible and 
\[
\|D\mathbb{F}_e(\tilde{u})^{-1}\|_{2,l} \leq \mathcal{C}_1, \text{ where } ~\mathcal{C}_1 \bydef \frac{\mathcal{C}}{1 - 2\kappa\mathcal{C}r_0}.
\]
\end{prop}

\begin{proof}
The proof is a simple application of a Neumann series argument combined with the fact that $\|D\mathbb{F}_e(\tilde{u})^{-1}\|_{2,l} \leq \mathcal{C}$ and $\|D\mathbb{F}_e(\tilde{u}) - D\mathbb{F}_e(u_0)\|_{l,2} = \|2\lambda_3 \mathbb{u}_0\|_{l,2} \leq 2 \kappa r_0$ using \eqref{eq : banach algebra kdv}.
\end{proof}

Proposition \ref{prop : norm inverse kdv} ensuring the invertibility of $D\mathbb{F}_e(\tilde{u})$, we can state Theorem 3.1 from \cite{stability_Albert}.

\begin{lemma}{(Theorem 3.1 \cite{stability_Albert})}
Let $\tilde{u} \in H^l$ be a solution to \eqref{eq:kdv_eq} and assume that 

{\em (P1)} $\mathbb{L} + D\mathbb{G}(\tilde{u})$ has a simple negative eigenvalue $\nu_1$\\
{\em (P2)} $\mathbb{L} + D\mathbb{G}(\tilde{u})$ has no negative eigenvalue other than $\nu_1$\\
{\em (P3)} 0 is a simple eigenvalue of $\mathbb{L} + D\mathbb{G}(\tilde{u})$.

Now let $\psi \bydef D\mathbb{F}_e(\tilde{u})^{-1}\tilde{u}$. If 
 \[
\int_{\mathbb{R}} \tilde{u} \psi < 0,
\]
then $\tilde{u}$ is stable.
\end{lemma}
\begin{proof}
First, notice that after changing $u \to -u$ and $y \to -y$ in \eqref{eq : time dependent kdv}, we fall into the required set-up to use the results from \cite{stability_Albert}.
Then, in order to apply Theorem 3.1 in \cite{stability_Albert}, one needs to satisfy the condition (i) stated in the aforementioned article. However, for the Kawahara equation, one can easily prove that this condition is satisfied using Theorem 2.1 (b) in \cite{stability_Albert}.  
\end{proof}

We proved conditions $(P1), (P2), (P3)$ in Theorem \ref{th:eigen_kdv}. The next proposition gives a numerical method to prove rigorously that $\displaystyle\int_{\mathbb{R}} \tilde{u} \psi < 0$ is satisfied.

\begin{prop}\label{prop:stable}
Let $\tilde{u}$ be the solution obtained in Theorem \ref{th:proof_kdv}. Then, let $V_0 \bydef AU_0,$ where $A$ is defined in Section \ref{ssec : bounds pde case}, and let $V = (V_n)_n \in X^4_0$ be the projection of $V_0$ in $X^4_0$ (using the construction of Theorem \ref{th:trace_theorem}). Finally, we introduce $\epsilon >0$ where
\[
\epsilon \bydef  \mathcal{C}_1|\Omega_0|^{\frac{1}{2}}\| U_0 - (L_e + DG_e(U_0))V\|_2 + 2|\Omega_0|^{\frac{1}{2}}\kappa\mathcal{C}_1 r_0 \|V\|_l.
\]
If there exists $\tau<0$ such that
\[
    |\Omega_0| \sum_{n \in \mathbb{Z}}(U_0)_nV_n + \epsilon |\Omega_0|^{\frac{1}{2}}\|U_0\|_2 + 2\mathcal{C}_1(|\Omega_0|^{\frac{1}{2}}\|U_0\|_2 + r_0)r_0 \leq \tau
\]
then
\[
    \int_{\mathbb{R}} \tilde{u} D\mathbb{F}_e(\tilde{u})^{-1}\tilde{u} < \tau.
\]
\end{prop}

\begin{proof}
We first notice that,
\begin{align*}
    \int_{\mathbb{R}} \tilde{u} D\mathbb{F}_e(\tilde{u})^{-1}\tilde{u} 
    = \int_{\mathbb{R}} u_0 D\mathbb{F}_e(\tilde{u})^{-1}u_0  + \int_{\mathbb{R}} \tilde{u} D\mathbb{F}_e(\tilde{u})^{-1}(\tilde{u}-u_0)  + \int_{\mathbb{R}}( \tilde{u}-u_0) D\mathbb{F}_e(\tilde{u})^{-1}u_0.
\end{align*}
 By construction, we have that $V \in X^4_0$ and the function representation $\tilde{v} \in H^l_0(\Omega_0)$ of $V$ can be extended to a function $v \in H^l(\mathbb{R}).$ Now, notice that since $l(\xi) \geq 1$ for all $\xi \in \mathbb{R}$, we have
 \[\|D\mathbb{F}_e(\tilde{u})^{-1}\|_{2} \leq \|D\mathbb{F}_e(\tilde{u})^{-1}\|_{2,l}.\]
Therefore,
\begin{align*}
    \|D\mathbb{F}_e(\tilde{u})^{-1}u_0 - v\|_2 = &\|D\mathbb{F}_e(\tilde{u})^{-1} \bigg( u_0 - (\mathbb{L}_e + D\mathbb{G}_e(\tilde{u}))v\bigg)\|_2\\
    &\leq \|D\mathbb{F}_e(\tilde{u})^{-1}\|_{2}\| u_0 - (\mathbb{L}_e + D\mathbb{G}_e(\tilde{u}))v\|_2\\
    &\leq  \mathcal{C}_1 |\Omega_0|^{\frac{1}{2}}\| U_0 - (L_e + DG_e(U_0))V\|_2 + \mathcal{C}_1 \|D\mathbb{G}_e(\tilde{u})v - D\mathbb{G}_e(u_0)v\|_2\\
    &\leq  \mathcal{C}_1|\Omega_0|^{\frac{1}{2}}\| U_0 - (L_e + DG_e(U_0))V\|_2 + 2|\Omega_0|^{\frac{1}{2}} \kappa\mathcal{C}_1 r_0 \|V\|_l
    =  \epsilon,
\end{align*}
where $\mathcal{C}_1$ is defined in Proposition \ref{prop : norm inverse kdv}. 
Then we obtain that
\begin{align*}
    \int_{\mathbb{R}} \tilde{u}(x) \bigg(\mathbb{L}_e + D\mathbb{G}_e(\tilde{u})\bigg)^{-1}\tilde{u}(x) dx
    &\le \int_{\Omega_0} u_0 v + \epsilon \|u_0\|_2 + 2\mathcal{C}_1(\|\tilde{u}\|_2 + r_0)r_0 dx \\
    & = |\Omega_0| \sum_{n \in \mathbb{Z}}(U_0)_nV_n + \epsilon \|u_0\|_2 + 2\mathcal{C}_1(\|u_0\|_2 + r_0)r_0 
\end{align*}
where $\|\tilde{u}-u_0\|_2 \leq \|\tilde{u}-u_0\|_l \leq r_0$ (as $l \geq 1$) and $r_0$ is defined in Theorem \ref{th:proof_kdv}.
\end{proof}

\begin{theorem}
The solution $\tilde{u} \in H^l_e$ obtained in Theorem \ref{th:proof_kdv} is (orbitally) stable.
\end{theorem}
\begin{proof}
The proof is accomplished using rigorous numerics in \cite{julia_cadiot} and using Proposition \ref{prop:stable}. In particular, we can choose $\tau \bydef -0.048$.
\end{proof}

\section{Conclusion and Future directions} \label{sec:conclusion}

  We presented a new approach to prove strong solutions to semi-linear PDEs in $H^l \subset H^s(\mathbb{R}^m)$ via computer-assisted proofs. More specifically, our main result enables a construction of a rigorous inverse of the linearization operator using Fourier series analysis.

Despite Assumptions~\ref{ass:A(1)} and \ref{ass : LinvG in L1}, the setting is still general enough to be applied to a large class of equations (e.g. KdV equation, Kawahara equation,  Swift-Hohenberg equation, etc). However, the relaxation of these assumptions is a future direction of interest and is currently under investigation. We present some ideas to generalize the current method.

First, note that Assumption \ref{ass : LinvG in L1} is a sufficient condition for $\mathbb{G} : H^l \to L^2$ to be well-defined but it is not necessary. Indeed,  
in the cases for which $\mathbb{G} : H^l \to L^2$ is not well-defined,  one can apply $(I-\Delta)^{\frac{s}{2}}$ to \eqref{eq : f(u)=0 on H^l} (if $\psi$ has enough regularity) for some $s \in \mathbb{N}$ big enough and define $\tilde{l}(\xi) \bydef l(\xi)(1+|\xi|^2)^{\frac{s}{2}}$. If $s$ is sufficiently big, then one can use the usual Banach algebra property for Sobolev spaces (see \cite{adams2003sobolev}) to obtain the well-definedness of $\mathbb{G} : H^{\tilde{l}} \to L^2.$ However, in the current set-up, Assumption \ref{ass : LinvG in L1} is needed for the analysis of Section \ref{sec:bound_inverse} (cf. Lemma \ref{lem : computation of falpha}). Therefore, one could study how the analysis of Section \ref{sec:bound_inverse} could be generalized under the assumption that $\G : H^l \to L^2$ is well-defined and smooth, without necessarily imposing Assumption \ref{ass : LinvG in L1}.

Notice that the method exposed in this paper can be generalized to a larger class of unbounded domains, including half-spaces and unbounded stripes. Indeed, under suitable boundary conditions, the computer-assisted approach and its associated Fourier analysis can be conducted. For instance, on an half-space, one could impose Neumann boundary conditions and a cosine series in the ``halfed" direction. On unbounded stripes, one could impose a periodic boundary condition in the bounded direction. Then, the ``fully" unbounded directions can be treated similarly as the case $\R^m$.

 Another future direction of interest would be the investigation of Assumption~\ref{ass:A(1)} to determine if a weaker condition could be used instead. Indeed, as discussed previously in Remark \ref{rem : essential spectrum of L}, the invertibility of $\mathbb{L}$ is needed in the current set-up in order to ensure invertibility of $D\mathbb{F}(u_0)$.
Moreover, we saw in Section \ref{sec:bound_inverse} that the inverse of $\mathbb{L}$ is  used to build a compact operator $ D\mathbb{G}(u_0)\mathbb{L}^{-1}$. Furthermore, it allows to generate an exponential decay that allows us to build an approximate inverse $\mathbb{A}$ (cf. Theorem \ref{th : approximation of inverse and norm}). Therefore, Assumption $\ref{ass:A(1)}$ is central in the current analysis to build a computer-assisted approach. It would be interesting to see if similar features could be obtained without the invertibility of $\mathbb{L}$ on the whole $H^l$. One way would be to impose restrictions on $H^l$ to obtain a subspace on which $\mathbb{L}$ would be invertible.

Moreover, in the case of  Kawahara equation, it seems that soliton solutions in $H^4(\mathbb{R})$ only exist if $\mathbb{L}$ is invertible (see \cite{shock_wave_hoefer} for instance). This might only be a specific feature of the Kawahara equation but it raises the following question :  in the case of an homogeneous PDE, is there any correlation  between the existence of solutions in $H^s(\mathbb{R}^m)$ and the invertibility of $\mathbb{L}$ ? In particular, it would be of interest to determine if there is a class of equations for which the invertibility of $\mathbb{L}$ is necessary.





\bibliographystyle{siam}
\bibliography{biblio}

\end{document}